\documentclass[twoside,a4paper,11pt,limits]{amsart}
\usepackage{graphicx}
\usepackage{color}
\usepackage{todonotes}

 \usepackage{amsmath}
 \usepackage{amssymb}
 \usepackage{amsthm}
 \usepackage{amscd}

 \usepackage{bbm}
 \usepackage{color}
 \usepackage[english=american]{csquotes}

\allowdisplaybreaks
\makeatletter
\long\def\unmarkedfootnote#1{{\long\def\@makefntext##1{##1}\footnotetext{#1}}}
\makeatother

\newtheorem{theorem}{Theorem}[section]
\newtheorem{lemma}[theorem]{Lemma}
\newtheorem{corollary}[theorem]{Corollary}
\newtheorem{proposition}[theorem]{Proposition}
\newtheorem{remark}[theorem]{Remark}

\numberwithin{equation}{section}
\newcommand{\meantmp}[2]{#1\langle{#2}#1\rangle}
\newcommand{\mean}[1]{\meantmp{}{#1}}

\newcommand{\setC}{{\mathbb{C}}}
\newcommand{\setN}{{\mathbb{N}}}
\newcommand{\setR}{{\mathbb{R}}}
\newcommand{\Rn}{{\mathbb{R}^n}}
\newcommand{\RN}{{\mathbb{R}^N}}
\newcommand{\RNn}{\mathbb{R}^{N\times n}}

\newcommand{\diam}{\text{\rm diam}}

\providecommand{\normtmp}[2]{{#1\lVert{#2}#1\rVert}}
\newcommand{\norm}[1]{\normtmp{}{#1}}
  \providecommand{\bigset}[1]{\settmp{\big}{#1}}

\newcommand{\dx}{\ensuremath{\,{\rm d} x}}
\newcommand{\dy}{\ensuremath{\,{\rm d} y}}
\newcommand{\dz}{\ensuremath{\,{\rm d} z}}

\newcommand{\bfw}{\mathbf{w}}
\newcommand{\bfu}{\mathbf{u}}
\newcommand{\bu}{\mathbf{u}}
\newcommand{\bfp}{\mathbf{p}}
\newcommand{\bfc}{\mathbf{c}}
\newcommand{\bfv}{\mathbf{v}}

\newcommand{\bfg}{\mathbf{g}}

\newcommand{\bff}{\mathbf{f}}

\newcommand{\bfF}{\mathbf{F}}
\newcommand{\bF}{\mathbf{F}}
\newcommand{\bfC}{\mathbf{C}}
\newcommand{\bfP}{\mathbf{P}}
\newcommand{\bfB}{\mathbf{B}}
\newcommand{\bfG}{\mathbf{G}}
\newcommand{\bfQ}{\mathbf{Q}}
\newcommand{\bfR}{\mathbf{R}}
\newcommand{\bfA}{\mathbf{A}}
\newcommand{\bfI}{\mathbf{I}}
\newcommand{\bfJ}{\mathbf{J}}
\newcommand{\bfT}{\mathbf{T}}

\newcommand{\bfphi}{\boldsymbol{\varphi}}
\newcommand{\bfvarphi}{\boldsymbol{\varphi}}

\newcommand{\bfpsi}{\boldsymbol{\psi}}
  \providecommand{\bfxi}{{\boldsymbol{\xi}}}

  \providecommand{\Xint}[1]{\mathchoice
    {\XXint\displaystyle\textstyle{#1}}%
    {\XXint\textstyle\scriptstyle{#1}}%
    {\XXint\scriptstyle\scriptscriptstyle{#1}}%
    {\XXint\scriptscriptstyle\scriptscriptstyle{#1}}%
    \!\int}
  \providecommand{\XXint}[3]{{\setbox0=\hbox{$#1{#2#3}{\int}$}
      \vcenter{\hbox{$#2#3$}}\kern-.5\wd0}}
  
  \providecommand{\dashint}{\mathop{\Xint-}}

  \providecommand{\abstmp}[2]{{#1\lvert{#2}#1\rvert}}
  \providecommand{\abs}[1]{\abstmp{}{#1}}
    \providecommand{\bigabs}[1]{\abstmp{\big}{#1}}
  \providecommand{\Bigabs}[1]{\abstmp{\Big}{#1}}
    \providecommand{\settmp}[2]{{#1\{{#2}#1\}}}
    \providecommand{\set}[1]{\settmp{}{#1}}

\newcommand{\pbar}{{p}}
\newcommand{\qbar}{{p}}

\DeclareMathOperator{\setVMO}{VMO}
\DeclareMathOperator{\setBMO}{BMO}
\DeclareMathOperator{\trace}{tr}

  \DeclareMathOperator{\divergence}{div}
\newcommand{\omegao}{\overline{\omega}}

\newcommand{\seb}[1]{{\textcolor[rgb]{1.00,0.30,0.00}{ #1}}}

\newcommand{\Lcal}{{\mathcal{L}}}

\allowdisplaybreaks

\newcommand{\db}[1]{\textcolor{black}{  #1}}

\begin{document}

\title[Global Schauder estimates for the $p$-Laplace system ]{Global Schauder estimates for the $\bfp$-Laplace system 
%Global non linear Schauder estimates for systems
}

\author{D.~Breit, A.~Cianchi, L.~Diening and S.~Schwarzacher}

\address{Dominic Breit,
School of Mathematical \& Computer Science, Heriot-Watt University,
Riccarton Edinburgh EH14 4AS UK} \email{d.breit@hw.ac.uk}
\address{Andrea Cianchi, Dipartimento di Matematica e Informatica \lq\lq U.Dini", Universit\`{a} di Firenze, Viale Morgagni 67/A, 50134, Firenze, Italy}
\email{andrea.cianchi@unifi.it}
\address{Lars Diening, 
Universit\"at Bielefeld
Fakult\"at f\"ur Mathematik
Postfach 10 01 31
D--33501, Bielefeld, Germany
} \email{lars.diening@uni-bielefeld.de}
\address{Sebastian Schwarzacher, Department of mathematical analysis, Faculty of Mathematics and Physics,  Charles University, Prague,
Sokolovsk\'{a} 83, 186 75, Prague, Czech Republic}
\email{schwarz@karlin.mff.cuni.cz}

%
%
%
%\address{Edinburgh}
%\address{Firenze}
%\address{Bielefeld}
%\address{Prague}
%
%\email{breit@math.lmu.de}
%\email{cianchi@unifi.it}
%\email{diening@math.lmu.de}
%\email{schwarz@karlin.mff.cuni.cz}

\begin{abstract}
An optimal  first-order global regularity theory, in spaces of functions defined in terms of oscillations, is established for solutions to Dirichlet problems for the $p$-Laplace equation and system, with right-hand side in divergence form. The exact mutual dependence among the regularity of the solution, of the datum on the right-hand side, and of the boundary of the domain in these spaces is exhibited. A comprehensive formulation of our results is given in terms of Campanato seminorms. New regularity results in customary function spaces, such as H\"older, $\setBMO$ and $\setVMO$ spaces, follow as a consequence. Importantly, the conclusions are new even in the linear case when $p=2$, and hence  the differential operator is the plain Laplacian. Yet in this classical linear setting, our contribution completes and augments the celebrated Schauder theory in H\"older spaces, and complements the Jerison-Kenig  gradient theory  in Lebesgue spaces with a parallel in the oscillation spaces realm. The sharpness of our results is demonstrated by apropos examples.
%
%under diverse aspectsThus,   they complement the Jerison-Kenig linear $L^p$ gradient regularity  theory  with a parallel in the oscillation spaces realm \seb{ and complete in this sense the celebrated H\"older regularity estimates of Schauder~\cite{Sch34}.} 
\end{abstract}

\maketitle
%\tableofcontents

\unmarkedfootnote {
\par\noindent {\it Keywords:} Quasilinear elliptic systems, global gradient regularity, $p$-Laplacian,
Dirichlet problems, Campanato spaces, H\"older spaces, $\setBMO$.
%
%
%\smallskip
%\par\noindent 
%This research was partly funded by:   Research Project of the
%Italian Ministry of University and Research (MIUR) Prin 2012 \lq\lq Elliptic and
%parabolic partial differential equations: geometric aspects, related
%inequalities, and applications"  (grant number 2012TC7588);    GNAMPA   of the Italian INdAM - National Institute of High Mathematics (grant number not available);    Ministry of Education and Science of the Russian Federation  (grant number 02.a03.21.0008).
%\\ The authors declare that they have no conflict of interest.
}

\section{Introduction}\label{intro}

\iffalse

\seb{TO DO LIST}
\begin{enumerate}
\item Add counterexample to show that our assumptions on the boundary is sharp. For the coeficients it follows from Jin-Maz'ya-Schaftingen 2009. The example for Proposition 1.6 exactly does not satisfy our condition, for homogenious right hand side. Include our counterexample that is derived from Maz'ya's book. 
\item Change the part with Neuman boundary condition. If pointwisely avaliable, we have to assume that $F_j\cdot \nu=0$ on $\partial\Omega$. Otherwise we need that the even reflection has mean value $0$ or similar.
The right condition for the Neuman problem is 
\begin{align}
\sup_{r\in (0,R_0)}\sup_{\bfx\in \partial \Omega}\dashint_{B_r(\bfx)}\abs{\bff(\bfy)\cdot \nu(\Pi_{\partial \Omega}(\bfy))}\, dy<\infty,
\end{align}
here $\Pi_{\partial \Omega}$ is the closest point projection to the boundary, which is well defined for $C^1$-boundary as long $R_0$ is chosen accordingly.
\end{enumerate}

\fi
We are concerned with  the Dirichlet problem for the
 $p$-Laplace system
\begin{align}
\label{eq:sysA}
\begin{cases}
  -\divergence(|\nabla \bfu|^{p-2} \nabla \bfu) = -\divergence
  \bfF\, & \quad \hbox{in $\Omega$}
\\ \bfu =0 & \quad  \hbox{on $\partial \Omega$.}
\end{cases}
\end{align}
Here,   the exponent $p \in (1, \infty)$, $\Omega$ is a bounded open set in $\Rn$, with $n \geq 2$,  the
function $\bfF \,:\, \Omega \to \RNn$, with $N \geq 1$, is  given,  $\bfu \,:\,
\Omega \to \RN$ is the unknown, and $\nabla \bfu : \Omega \to \RNn$ denotes its gradient.
Under the assumption that  $\bfF \in L^{p'}(\Omega)$, where  $p'=\tfrac p{p-1}$, one has that $\divergence
  \bfF$  belongs to the dual of the Sobolev space  $ W^{1,p}_0(\Omega)$.  Hence, a weak solution $\bfu \in  W^{1,p}_0(\Omega)$ to problem \eqref{eq:sysA} is well defined, and its existence and uniqueness follow via  standard variational methods.
\par The 
present paper focuses  on 
global -- namely up to the boundary -- higher regularity properties of 
 $\nabla \bfu$ inherited from those of
$\bfF$. Specifically, we offer a sharp global Schauder regularity theory for norms depending on oscillations of  $\nabla \bfu$. Campanato type norms provide a suitable framework for a unified formulation of such a theory. Membership of the gradient of the solution $\bfu$ to problem \eqref{eq:sysA} in Campanato type spaces depends  on both the regularity of  the datum $\bfF$ and that of the boundary $\partial \Omega$  in the same kind of spaces. 
Our results provide an exact description of the interplay among these three pieces of information, and show that the required balance among them is qualitatively independent of the dimensions $n$ and $N$, and of $p$ . Their optimality  is  demonstrated via a precise analysis of the behaviour of the solutions in suitable model problems.  Proofs entail the development of new decay estimates on balls near the boundary, that rely upon an unconventional flattening technique exploiting   local coordinates which depend on the radius of the balls.
%Global $\setBMO$ regularity is an  important new consequence our general estimates, which cover the whole region between $\setBMO$ and H\"older spaces, including, for instance, spaces defined in terms of a general modulus of continuity.  In fact, on providing sharp quantitative information, our results  also enhance the classical nonlinear H\"older theory.
%
%
%In particular, our regularity estimates  cover the region between $\setBMO$ and H\"older spaces, including, for instance, spaces defined in terms of a general modulus of continuity, 
%thus enhancing the nonlinear  theory of H\"older regularity.
%
%In particular, H\"older estimates,  more general estimates in spaces of uniformly continuous functions, and  $\setBMO$ estimates for the gradient  are included as special instances of our  discussion. 
\par Although our primary interest is in  nonlinear problems, the conclusions to be presented are new, and best possible, even in the linear case when $p=2$, namely when  the differential operator in \eqref{eq:sysA} is just the Laplacian. Interestingly, since our results also admit a local version, they provide novel optimal gradient regularity properties up to the boundary also for harmonic functions vanishing on a subset of $\partial \Omega$. This can be regarded as a counterpart, in the scale of norms depending on oscillations, of the sharp $L^q$ gradient regularity theory   for linear equations developed in \cite{JeKe},  and of the $L^\infty$ gradient bounds of \cite{Mazya1}.
%\seb{Whereby the results provide a complete characterization of the gap between H\"older continuity on the one hand and the space of bounded mean oscillations ($\setBMO$) on the other hand. Hence the results can also be regarded as a completion of the H\"older regularity estimates of Schauder~\cite{Sch34}.}
%\todo{Results are new even for Laplacian}
%
%
%belongs to
%borderline function spaces, which are \lq\lq close " to $L^\infty
%(\Omega)$, and are defined in terms of of the oscillation of $\bfF$.
%They are spaces of Campanato type, and hence our approach enables to
%include, in a unified framework, various  instances, such as the
%space of functions of Bounded Mean Oscillation
% $\setBMO (\Omega )$, and  the  spaces $C^{0, \beta}(\Omega
% )$ of H\"older continuous functions with exponent $\beta$. 
\par As far as  genuinely nonlinear problems -- corresponding to $p \neq 2$ --  are concerned, 
global $\setBMO$ regularity is an  important new consequence of our general estimates, which actually cover the whole region between $\setBMO$ and H\"older spaces, including, for instance, spaces defined in terms of a general modulus of continuity.  In fact, on providing sharp quantitative information, our results  also enhance the classical global theory, where H\"older norms are employed to describe the regularity of $\bF$, $\partial \Omega$ and $\nabla \bu$.
\par
The
 %$C^{1,\beta}$ -- i.e.
local  H\"older gradient regularity   of solutions to the system in
 \eqref{eq:sysA}, in the homogeneous case when $\bfF =0$ and for $p \geq 2$, goes back
 to the paper \cite{Uhl77}, after which systems involving differential operators depending
  only on the length of the gradient are usually called with Uhlenbeck structure.  The same result in the scalar case ($N=1$) had been earlier
established in \cite{Ura68} for every $p \in (1,\infty )$. Unlike the case of systems, the Uhlenbeck structure is not needed for the regularity of solutions in the scalar case, as shown in the papers \cite{DiBe, Ev, Le, To}.  The
contribution \cite{Uhl77} was  extended to the situation when
$1<p<2$ in \cite{AceF89} and in \cite{ChenDiBe}, the latter paper also including non-vanishing right-hand sides and parabolic problems.
%\db{{\color{yellow} There is another elliptic paper by Tolkdsdorf from '83 and a parabolic paper by DiBenedetto-Friedmann from 84}.} 
 The
local $\setBMO$ gradient regularity for solutions to \eqref{eq:sysA}
is proved, for $p \geq 2$, in \cite{DiBMan93}. A version of
that result, which holds for every $p \in (1, \infty)$, and for
$\bfF$ in more  general Campanato type spaces, has recently been
obtained in \cite{DieKapSch11}, but still in local form.
\par The global regularity theory is not as developed as the local one. The H\"older gradient regularity for equations of $p$-Laplacian type, in domains whose boundary has also H\"older continuous first-order derivatives, can be traced back to \cite{Lie88}. The result for systems (with Uhlenbeck structure)
%A global Uhlenbeck type result
%, namely  H\"older gradient regularity up to the boundary for the systems, 
was achieved in 
\cite{ChenDiBe} for domains of the same kind. However, 
we stress that our result provide us with the best possible H\"older exponent for  first-order derivatives of the solution depending on the H\"older exponent of the first-order derivatives of the boundary,  whereas the conclusions of \cite{Lie88} and \cite{ChenDiBe} do not yield any explicit mutual dependence of these exponents. Moreover, 
the right-hand sides considered in
\cite{Lie88} and \cite{ChenDiBe} are not in divergence form, and hence 
%.  Since borderline
%spaces for $\bfF$ are in question,
 the system in \eqref{eq:sysA} cannot
be reduced to the form of those papers in general.  Right-hand sides in non-divergence form   also appear in  \cite{cm-2, cm-1}, where $L^\infty$ gradient estimates, under minimal boundary regularity assumptions, are established.   Results on global $\setBMO$ regularity seem  to be still completely missing in the existing literature. Filling a gap in this major special instance was one of the original motivations for our research. 
\par  Further contributions 
on gradient regularity up to the boundary
 for  systems and variational problems with
 Uhlenbeck structure, or perturbations of it,
% and non-homogenous Dirichlet conditions
 are  \cite{BMSV, Fo, FPV, Ham}. Partial boundary
regularity, i.e. regularity at the boundary outside subsets of zero
$(n-1)$-dimensional Hausdorff measure, for nonlinear elliptic
systems with general structure, is proved in  \cite{DKM, KrM2} (see also \cite{Ham2} for a special case).  Related results on regular boundary points can be found in \cite{Be, Gr}.

\section{Main results}\label{sec:main}

Our comprehensive result, stated in Theorem \ref{thm:campanato},  is formulated in terms of Campanato type
seminorms $\norm{\cdot }_{\mathcal L ^{\omega (\cdot)} (\Omega )}$, associated
with  parameter-functions  $\omega : [0, \infty) \to
[0, \infty)$ which will be assumed to be   continuous and  non-decreasing in what follows. 
These seminorms are defined as
\begin{align}\label{campnorm}
  \norm{\bff }_{\mathcal L^{\omega (\cdot)}(\Omega)}=
%\sup_{x \in \Omega, r>0}}\frac{1}{\omega(r)}
\sup_{
\begin{tiny}
 \begin{array}{c}{
    x \in \Omega} \\
r>0
 \end{array}
  \end{tiny}
}
  \dashint_{\Omega \cap B_r(x)}\abs{\bff-\mean{\bff}_{\Omega\cap B_r(x)}}\dy,
\end{align}
for a real, vector or matrix-valued integrable function  $\bff$ in $\Omega$.  Here, $B_r(x)$ denotes a ball of radius $R$ and centered at $x$,
$\dashint$ stands for averaged integral, and $\mean{\bff}_{E}$ for the mean value of $\bff$ over a set $E$. As hinted above, and will be specified below,  the spaces $\mathcal L ^{\omega (\cdot)} (\Omega )$    are a family of spaces that, depending on the choice of $\omega$, may consist of continuous functions with modulus of continuity $\omega$, of continuous functions with a slightly worse modulus of continuity, or also include discontinuous and unbounded functions, but with a degree of integrability depending on $\omega$. In the borderline case corresponding to $\omega (r)=1$,   $\mathcal L ^{\omega (\cdot)} (\Omega )$  agrees with the space $\setBMO(\Omega)$ of functions of bounded mean oscillation in $\Omega$.  Observe that, as a consequence of the John-Nirenberg lemma for functions in $\setBMO(\Omega)$, replacing the integral $ \dashint_{\Omega\cap B_r(x)}\abs{\bff-\mean{\bff}_{\Omega\cap B_r(x)}}\dy$ by $ \Big(\dashint_{ \Omega \cap B_r(x)}\abs{\bff-\mean{\bff}_{\Omega\cap B_r(x)}}^q\dy\big)^{\frac 1q}$ on the right-hand side of equation \eqref{campnorm} results in an equivalent seminorm for every $q > 1$.
%
%The restriction above is therefore natural with respect to regularity of the homogeneous case. More precisely, the known H\"older regularity exponent $\beta_0$ for $\abs{\nabla \bfh}^{p-2}\nabla \bfh$, with $\bfh$ being a local solution to ~\eqref{eq:sysA} with $\bfF\equiv0$ is the upper thrash hold for the regularity that can be transfered from $\bfF$ to $\abs{\nabla \bfu}^{p-2}\nabla \bfu$
%in the inhomogeneous case. 
\par
We denote by $C^{0, \omega (\cdot)}(\Omega)$ the space of  functions $\bff$ in $\Omega$ endowed with the seminorm
\begin{equation}\label{C0om}
\|\bff \|_{C^{0, \omega (\cdot)}(\Omega)}=
\sup_{
\begin{tiny}
 \begin{array}{c}{
    x, y \in \Omega} \\
x \neq y
 \end{array}
  \end{tiny}
}
% \sup_{x, y \in \Omega}
\frac{|\bff (x) - \bff(y)|}{\omega (|x-y|)}\,.
\end{equation}
Plainly, if $\omega (0)=0$, then $C^{0, \omega (\cdot)}(\Omega)$ is a space of uniformly continuous functions in $\Omega$,  with modulus of continuity not exceeding $\omega$. If $\omega (r)= r^\beta$ for some $\beta \in (0, 1]$, then   $C^{0, \omega (\cdot)}(\Omega)$ coincides with the space of H\"older continuous functions with exponent $\beta$, that will simply be denoted by $C^{0,\beta}(\Omega)$, as customary.  The   space of functions obtained on replacing $\bff$ by $\nabla \bff$ in the definition of the seminorm \eqref{C0om} will be denoted by $C^{1, \omega (\cdot)}(\Omega)$. The meaning of the notation $C^{1,\beta}(\Omega)$ is analogous. Of course, when $\Omega$ is bounded, only the behavior of $\omega$ near $0$ is relevant in the definitions of $C^{0, \omega (\cdot)}(\Omega)$ and $C^{1, \omega (\cdot)}(\Omega)$.
\\ It is easily seen that 
\begin{equation}\label{embcamp}
C^{0, \omega (\cdot)}(\Omega) \to \mathcal L ^{\omega (\cdot)}(\Omega)
\end{equation}
for every parameter function $\omega$, where the arrow $\lq\lq \to "$ stands for continuous embedding. A reverse embedding holds if $\omega (r) = r^\beta$, for any $\beta \in (0, 1]$, provided that $\Omega$ is regular enough -- a bounded Lipschitz domain, for instance. However, it may fail if $\omega$ does not decay to $0$ rapidly enough. In fact, functions in $\mathcal L ^{\omega (\cdot)}(\Omega)$ need not even be (locally) bounded on $\Omega$. We shall be more precise about this issue  below.
\par
In view of our applications, an  additional property will be imposed on parameter functions $\omega$. We shall  assume
 that  the function $\omega (r)r^{-\beta_0}$ is almost decreasing
  for a suitable exponent  $\beta _0= \beta _0(n,N,p)$. Such an exponent  depends  on the optimal H\"older exponent for gradient regularity of $p$-harmonic functions, namely local solutions to the system in \eqref{eq:sysA} with $\bfF=0$. % -- see \cite[???]{BCDKS15}. 
  This amounts
to requiring  that
\begin{equation} \label{eq:omega condition}
\omega(r) \leq c_\omega \theta^{-\beta_0} \omega(\theta r) \qquad \hbox{for
 $\theta  \in (0,1)$,}
\end{equation}
for some  constant $c_\omega$. 
%The bound $\beta_0$ is due to the optimal regularity known for p-harmonic gradients. It is known to exist due to the seminal works of Uralc'eva~\cite{Ura68} (for equations) and Uhlenbeck (for systems)~\cite{Uhl77}, the particular $\beta_0$ is given in Proposition~\ref{prop:decay conclusion} and was deduced in~\cite{BCDKS15}. 
Such an assumption is clearly undispensable,  due to the maximal regularity enjoyed by $p$-harmonic functions. The explicit value of $\beta_0$, in the case when $n=2$ and $N=1$, has been detected in \cite{IwMa}.% \db{Mention Iwaniec-Manfredi for $n=2$}
\par When writing $\partial \Omega \in X$ for some function space $X$, we mean that 
$\Omega$ is   a bounded open set, which, in a neighbourhood of each point of $\partial \Omega$,  agrees with the subgraph of a  function of    $(n-1)$ variables that belongs to $X$. Similarly, the notation  $\partial \Omega \in W^1X$ has to be understood in the sense that such function is weakly differentiable, and its weak derivatives belong to the space $X$.
\par Theorem \ref{thm:campanato}, as well as the other results of this paper, are most neatly formulated in terms of the nonlinear expression $|\nabla \bfu|^{p-2} \nabla \bfu$  appearing under the divergence operator in the system in \eqref{eq:sysA}. As shown in several recent contributions, this is a proper expression to use in the description of the regularity   of solutions to $p$-Laplacian type equations and systems -- see e.g. \cite{BCDKS, cm0, cm1, cm2,  DieKapSch11, AKM, KM1, KM2}.
%
%
%Clearly the regularity of the boundary has to be suitable with respect to the regularity estimates wanted. With respect to the techniques invented here it can be best formulated via the following weight:
%\db{this is a bit unclear, also why needed here? $\overline{\omega}$ does not appear in the statements}
%\begin{equation}\label{psi}
%\overline{\omega}(r)=\int^1_r\frac{\omega(\rho)\,d\rho}{\rho}.
%\end{equation}
%In the special case, that $\omega\equiv 1$, $\overline{\omega}(r)=\log\Big(\frac{r}{R}\Big)$. In case $\omega(r)=r^\beta$, which is the classical Campanato setting, $\overline{\omega}$ is simply bounded.
%
%\iffalse
%\seb{Include here different assumption on $\bfF$ in case of Neuman data. In case of $\setBMO$. In some appropriate weak sense we need $f\cdot \nu=0$ on $\partial \Omega$.} 
%\fi
%
%\color{blue}
%What about stating the results also locally in a ball $B$ centered on $\partial \Omega$, for solutions that satisfy the boundary condition just on $\partial \Omega \cap B$? This would be consistent with our counterexamples.

\color{black}

\begin{theorem}
  \label{thm:campanato} {\rm {\bf [Regularity in Campanato spaces]}} Let $\Omega $ be a bounded open set in $\Rn$ such that $\partial \Omega \in W^1\Lcal^{\sigma (\cdot)} \cap C^{0,1}$  for some parameter function $\sigma$.  Let $\omega$ be a parameter function 
satisfying condition \eqref{eq:omega condition}. Assume that
 $\bfF \in \mathcal L ^{\omega (\cdot)} (\Omega )$ and let $\bfu$ be the solution to the Dirichlet problem
  \eqref{eq:sysA}.
There exists  constants  $\delta=\delta (p, N, \omega, \Omega)$ and $C=C(p, N, \omega, \Omega)$ such that, if 
\begin{equation}\label{balance}
\sup _{r\in (0,1)} \frac{\sigma(r)}{\omega(r)}\int^1_r\frac{\omega(\rho)}{\rho} \,\mathrm{d}\rho\leq \delta,
% \qquad \hbox{for $r\in (0,1)$,}\,
\end{equation}
 then $|\nabla \bfu|^{p-2} \nabla \bfu\in \mathcal L ^{\omega (\cdot)} (\Omega )$, and  
  \begin{align}\label{feb1}
    \norm{|\nabla \bfu|^{p-2} \nabla \bfu}_{\mathcal L ^{\omega (\cdot)} (\Omega )} \leq
% c\,\dashint_{\Omega}\abs{ (A(\nabla
   %   u))}\, dx
  C\norm{\bfF}_{\mathcal L ^{\omega (\cdot)} (\Omega )}\,.
  \end{align}
  %The constant $c$ depends only on the characteristics of $\phi,\gamma$ and the $C^{1,\gamma}$ constant of the boundary.
\end{theorem}

%\begin{remark}\label{remsharp}
%{\rm
%We emphasize that the balance condition \eqref{balance} is sharp. This is can be shown by considering the simplest possible situation when $n=2$, $N=1$ and $p=2$, namely  the (linear) Poisson equation in the plane -- see Theorem \ref{counterex}, Section \ref{sharp}.}
%\end{remark}
%
%\begin{remark}\label{remlocal}
%{\rm The result of Theorem \ref{thm:campanato} has a local nature. Indeed, it will be clear from the proof   that, under the same assumptions on $\bfF$ and $\partial \Omega$, one has that $|\nabla \bfu|^{p-2} \nabla \bfu \in \mathcal L ^{\omega (\cdot)} (B_R \cap \Omega)$   for any solution of the system in \eqref{eq:sysA} in $B_{2R}\cap \Omega$ that fuflills the Dirichlet boundary condition on $B_{2R}\cap \partial \Omega$. A similar remark applies to the results stated hereafter, that rely upon Theorem \ref{thm:campanato}.
%}
%\end{remark}

\medskip

Condition \eqref{balance} in
Theorem \ref{thm:campanato} is sharp, when 
\begin{equation}\label{divint}
\int_0\frac{\omega(r)}r\, {\rm d}r=\infty\,,
\end{equation} 
 in the sense that not only the finiteness of the supremum in \eqref{balance}, but also its smallness  cannot be dispensed with.
This can be demonstrated yet  in the simplest situation when $n=2$, $N=1$ and $p=2$ to which we alluded above, namely for the scalar Dirichlet problem for the Poisson equation in the plane
\begin{equation}\label{poisson}
\begin{cases}
- \Delta u = - \divergence \bF & \quad \hbox{in $\Omega$}
\\
u =0 & \quad \hbox{on $\partial \Omega$.}
\end{cases}
\end{equation}
This is the content of Theorem \ref{sharpness}, that tells us that the conclusion of Theorem  \ref{thm:campanato} may fail if the supremum on the left-hand side of equation \eqref{balance}, though finite, is not small enough.
\par On the other hand, if instead $\omega (r)$ decays so  fast to $0$ as $r \to 0^+$ that
\begin{equation}\label{dini} \int _0 \frac{\omega (r)}r\, {\rm d} r < \infty,
\end{equation}
 then condition \eqref{balance} can still be slightly relaxed, by requiring that its left-hand side is just finite, and hence allowing for the choice $\omega = \sigma$.
This is stated in Therorem \ref{continuity} below, which also asserts that, under condition \eqref{dini}, the function $|\nabla \bfu|^{p-2} \nabla \bfu$ is uniformly continuous, with a modulus of continuity depending on $\omega$
and $p$.

\begin{theorem}\label{sharpness}{\rm {\bf [Sharpness]}}
Let $\omega \in  C^1(0, \infty)$ be any concave parameter function,
satisfying conditions \eqref{eq:omega condition} and \eqref{divint},  and such that   $\lim_{r\to 0^+}\frac{r\omega '(r)}{\omega (r)}$ exists.  
%and
%\begin{equation}\label{divint}
%\int_0\frac{\omega(r)}r\, {\rm d}r=\infty.
%\end{equation}
Then there exist a parameter function $\sigma$, a bounded open set $\Omega \subset \setR^2$, and a function $\bfF  \in \mathcal L ^{\omega (\cdot)} (\Omega )$ such that, if $u$ is the solution to the Dirichlet problem \eqref{poisson},
then:
\begin{equation}\label{Nov30}
\partial \Omega \in W^1\Lcal^{\sigma (\cdot)} \cap C^{0,1}
\end{equation}
and
\begin{equation}\label{Nov31}
 \sup _{r \in (0,1)}\frac{\sigma(r)}{\omega(r)}\int^1_r\frac{\omega(\rho)}{\rho} \,\mathrm{d}\rho< \infty,
\end{equation}
but
\begin{equation}\label{Nov32}
  \nabla u \notin \mathcal L ^{\omega (\cdot)} (\Omega ).
\end{equation}
\end{theorem}

\medskip

\begin{remark}\label{sharplocal}
{\rm 
The result of Theorem \ref{thm:campanato} has a local nature. Indeed, it will be clear from its proof   that, under the same assumptions on $\bfF$ and $\partial \Omega$, if $B_R$ is a ball centered on $\partial \Omega$, then  $|\nabla \bfu|^{p-2} \nabla \bfu \in \mathcal L ^{\omega (\cdot)} ( \Omega \cap B_R)$   for any solution of the system in \eqref{eq:sysA} in $ \Omega \cap B_{2R}$ that fulfills the Dirichlet boundary condition on $\partial \Omega \cap B_{2R}$.  The sharpness of Theorem \ref{thm:campanato} can also be shown in its local version, as  is apparent from the proof of Theorem  \ref{sharpness}. Indeed,  a local formulation of  Theorem  \ref{sharpness} allows for the choice $F\equiv 0$, and  hence applies to scalar harmonic functions that just vanish  on part of the boundary.
The other results of this paper, that rely upon Theorem \ref{thm:campanato}, also admit a local variant.
}
\end{remark}

%
%Additionally to Theorem~\ref{thm:main}, we are able to transfer any modulus of continuity of the mean oscillation from $f$ to
%$A(\nabla u)$.
%Before stating our result enhanced result under assumption \eqref{dini}, let us enucleate

\begin{remark}\label{qualitative}
{\rm  As announced in Section \ref{intro}, 
condition \eqref{balance} in Theorem \ref{thm:campanato} is qualitatively -- namely up to the constant $\delta$ -- independent of the dimension $n$ (and $N$) and of the exponent $p$. On the other hand, the optimality of this condition is shown in Theorem \ref{sharpness} under the presumably smoothest situation corrsponding to the two-dimensional linear scalar case.  Theorem \ref{thm:campanato} and Remark \ref{sharplocal} thus tell us that global regularity properties of the gradient in Campanato type spaces hold in any dimension, and for any power-nonlinearity, under boundary conditions that are qualitatively sharp still for harmonic functions in the plane. The fact that an optimal boundary regularity assumption be dimension-free is a feature that our result shares with the linear  gradient regularity theory  for Lebesgue norms in Lipschitz domains developed in  \cite{JeKe}.  The contribution \cite{ByuWan} generalizes, to  some extent, that theory to nonlinear problems in the  framework of Raifenberg-flat domains. 
}
\end{remark}

\par
 The conclusion of  Theorem
\ref{thm:campanato} corresponding to the special    choice $\omega (r)=1$ is enucleated in the next corollary.  In this case, Theorem
\ref{thm:campanato}   implies that regularity in $\setBMO (\Omega)$ of $\bfF$ is reflected into the same regularity for $|\nabla \bfu|^{p-2} \nabla \bfu$,  provided that  the derivatives of  the functions describing $\partial \Omega$ belong to a Campanato type space associated with a logarithmic parameter function $\omega$. An  additional argument ensures that a parallel conclusion holds if $\setBMO (\Omega)$ is replaced with $ \setVMO
(\Omega)$, the space of functions of vanishing mean oscillation on $\Omega$. Recall that 
a real, vector or matrix-valued integrable function  $\bff$ in $\Omega$  is said to belong to $\setVMO (\Omega)$ if
\begin{equation}\label{VMO} \lim
_{r \to 0^+} \Bigg(
\sup_{
\begin{tiny}
 \begin{array}{c}{
    x  \in \Omega} \\
0<s \leq r
 \end{array}
  \end{tiny}
}
%\sup _{
%%\begin{tiny}
%% \begin{array}{c}{
%    %B_r \subset \Omega} \\
%                     x\in \Omega,   s\leq r
%%                             \end{array}
%%                              \end{tiny}
%}
\, \dashint _{\Omega\cap B_s(x)} |\bff - \mean{\bff}_{\Omega \cap B_s(x)}|\,
  \dy\Bigg) =0.
  \end{equation}
These results are stated in the next    corollary,  whose sharpness  is a consequence  of Theorem~\ref{sharpness}.

%
% immediately yields the following regulairty result in $\setBMO (\Omega)$. In fact, 
%
%
%Since $\mathcal L ^1 (\Omega )= \setBMO (\Omega)$, Theorem
%\ref{thm:campanato} immediately yields a $\setBMO (\Omega)$
%regularity result. It also easily implies that, if $\bfF \in \setVMO
%(\Omega)$, the space of functions of Vanishing Mean Oscillation,
%then $\abs{\nabla \bfu}^{p-2}\nabla \bfu \in  \setVMO (\Omega)$ as well.

\begin{corollary}\label{bmovmo}{\rm {\bf [$\setBMO$ and $\setVMO$ regularity]}}
   %Assume that $\partial \Omega \in C^{1, \gamma}$ for some $\gamma \in (0, 1)$. Let $\bfu$ be a solution of
%  \eqref{eq:sysA} coupled with either \eqref{eq:dirichlet} or \eqref{eq:neumann} on
%  $\Omega$.
Let $\Omega $ be a bounded open set in $\Rn$ such that $\partial \Omega \in W^1\Lcal^{\sigma (\cdot)} \cap C^{0,1}$ for some parameter function $\sigma$ satisfying
\begin{equation}\label{vanishlog}
\lim _{r\to 0^+} \sigma (r) \log (1/r) =0\,.
\end{equation}
Assume that $\bfF \in \setBMO (\Omega )$ and let $\bfu$ be the solution to the Dirichlet problem
  \eqref{eq:sysA}.
Then $|\nabla \bfu|^{p-2} \nabla \bfu \in \setBMO(\Omega )$, and there exists a constant $C=C(p,  N,  \Omega)$ such that
  \begin{align*}
    \norm{|\nabla \bfu|^{p-2} \nabla \bfu}_{\setBMO(\Omega )} \leq
% c\,\dashint_{\Omega}\abs{ (A(\nabla
   %   u))}\, dx
  C\norm{\bfF}_{\setBMO(\Omega )}.
  \end{align*}
    %\db{Why mention $\sigma$? Depends on $\Omega$!}
\\ Moreover, if $\bfF \in \setVMO (\Omega )$, then $|\nabla \bfu|^{p-2} \nabla \bfu
\in \setVMO (\Omega)$ as well.

%\seb{This has to be checked!} \todo[inline]{There is no proof of $\setVMO$ regularity yet. Shall we drop it?}
  %The constant $c$ depends only on the characteristics of $\phi,\gamma$ and the $C^{1,\gamma}$ constant of the boundary.
\end{corollary}

\iffalse

When $\omega (t)$ decays so fast to $0$ as $t \to 0^+$ that
\begin{equation}\label{dini} \int _0 \frac{\omega (t)}t\, dt < \infty,
\end{equation}
 Theorem
\ref{thm:campanato} tells us that the expression $|\nabla \bfu|^{p-2} \nabla \bfu$ is locally uniformly continuous in $\Omega$, with a modulus of continuity depending on $\omega$ and $p$. This is a consequence of inclusion relations between Campanato spaces and spaces of uniformly continuous functions \cite[???]{Spa65}. In fact, a variant in the proof shows that this conclusion still holds provided that the left-hand side of inequality \eqref{balance} is just finite, and not necessarily smaller than some constant $\delta$ depending on ???. This is the content of the next result.

\fi
%
%
% tells us that  membership of $\bfF$ to
%$C^{0,{\omega (\cdot)}} ( \Omega )$, the space of uniformly continuous
%functions in $\Omega$  with modulus of continuity $\omega$, implies
%the uniform continuity of $\nabla \bfu$  in $ \Omega$, with a modulus of continuity depending on $\omega$ and $p$. This is a consequence 

Our enhanced result under assumption \eqref{dini} is the subject of Theorem \ref{continuity}.
%Inclusion relations between
Embeddings of Campanato spaces into  spaces of uniformly continuous functions, to which we alluded above, 
have a role in its proof.  They go back to   \cite{Spa65},  and tell us that, if $\Omega$ is a bounded Lipschitz domain, and the parameter function $\omega$ fulfills  condition  \eqref{dini}, then 
\begin{equation}\label{spanne1}
\mathcal L ^\omega(\Omega) \to  C^{\underline \omega}(\Omega),
\end{equation}
where $\underline \omega$ is the parameter function defined by
\begin{equation}\label{dini1}
%\varpi 
\underline \omega (r) = \int _0^r \frac{\omega (\rho)}\rho\,\mathrm{d}\rho \, \quad
\hbox{for $r \geq 0$.}
\end{equation}
Note that, as shown in  \cite{Spa65}, if the function $\tfrac{\omega (r)}r$ is non-increasing, condition \eqref{dini} is necessary  even for the space $\mathcal L ^\omega(\Omega)$ to be included in
$L^\infty (\Omega)$. Also, the function $\underline \omega$ is optimal in \eqref{spanne1}.

\begin{theorem}\label{continuity}{\rm {\bf [Continuity estimates]}}
  % Assume that $\partial \Omega \in C^{1, \gamma}$ for some $\gamma \in (0, 1)$. Let $\bfu$ be a solution of
%  \eqref{eq:sysA} coupled with either \eqref{eq:dirichlet} or \eqref{eq:neumann} on
%  $\Omega$.
Let $\Omega$ be a bounded open set in $\Rn$ such that $\partial \Omega \in W^1\Lcal^{\omega (\cdot)}$ for some parameter function  $\omega$ 
satisfying conditions \eqref{eq:omega condition} and \eqref{dini}.
Assume that
 $\bfF \in \mathcal L ^{\omega (\cdot)} (\Omega )$ and let $\bfu$ be the solution to the Dirichlet problem
  \eqref{eq:sysA}. 
Then $|\nabla \bfu|^{p-2} \nabla \bfu\in \mathcal L ^{\omega (\cdot)} (\Omega )$, and  inequality \eqref{feb1} holds.
 \\ Moreover, if 
 $\underline \omega :
[0, \infty ) \to [0, \infty )$ is the parameter function given by \eqref{dini1},
%\begin{equation}\label{dini1}
%%\varpi 
%\underline \omega (t) = \int _0^t \frac{\omega (\tau)}\tau\,\mathrm{d}\tau \, \quad
%\hbox{for $t \geq 0$,}
%\end{equation}
%\begin{equation}\label{dini1}
%{\underline \omega} (t) = \bigg(\int _0^t \frac{\omega (\tau)}\tau\, d\tau\bigg)^{\frac 1{p-1}}\, \quad
%\hbox{for $t \geq 0$}.
%\end{equation}
%Assume that
%then $\bfF \in {C^{0, {\omega (\cdot)}} (\Omega )}$,
% and let $\bfu$ be the solution to the Dirichlet problem  \eqref{eq:sysA}.     
then $|\nabla \bfu|^{p-2} \nabla \bfu \in C^{0, {\underline \omega} (\cdot)}( \Omega )$, 
and there exists a   constant $C=C(p, N, \omega, \Omega)$ such that
  \begin{align}\label{dec1030}
    \norm{|\nabla \bfu|^{p-2} \nabla \bfu}_{C^{0, {\underline \omega} (\cdot)}( \Omega )} \leq
% c\,\dashint_{\Omega}\abs{ (A(\nabla
   %   u))}\, dx
  C
\norm{\bfF}_{
 \mathcal L^{\omega(\cdot)} (\Omega )}\,.
%
%\norm{\bfF}_{C^{0, {\omega (\cdot)}} (\Omega )},
  \end{align}
%\\ Moreover, if $\bfF \in \setVMO (\Omega )$, then $|\nabla \bfu|^{p-2} \nabla \bfu
%\in \setVMO (\Omega)$.
  %The constant $c$ depends only on the characteristics of $\phi,\gamma$ and the $C^{1,\gamma}$ constant of the boundary.
\\ In particular, the same conclusions hold if $\bfF \in {C^{0, {\omega (\cdot)}} (\Omega )}$, and $\norm{\bfF}_{
 \mathcal L^{\omega(\cdot)} (\Omega )}$ is replaced with the stronger norm $\norm{\bfF}_{C^{0, {\omega (\cdot)}} (\Omega )}$ in inequalities  \eqref{feb1}  and \eqref{dec1030}.
\end{theorem}

%\begin{remark}\label{july1}
%{\rm As will be clear from the proof, the assumption  $\bfF \in {C^{0, {\omega (\cdot)}} (\Omega )}$ in Theorem \ref{continuity} can  be replaced by the slightly weaker hypothesis that $\bfF \in \mathcal L^{\omega(\cdot)} (\Omega )$. \todo{State also with $\nabla u$ in Campanato}
%}
%\end{remark}

The specific choice  $\omega (t) = t^\beta$ in Theorem
\ref{continuity}, with $\beta<\beta_0$, yields
the following H\"older continuity result.

\begin{corollary}\label{holder}{\rm {\bf [H\"older continuity]}}
  % Assume that $\partial \Omega \in C^{1, \gamma}$ for some $\gamma \in (0, 1)$. Let $\bfu$ be a solution of
%  \eqref{eq:sysA} coupled with either \eqref{eq:dirichlet} or \eqref{eq:neumann} on
%  $\Omega$.
Let $\Omega$ be a bounded open set in $\Rn$ such that $\partial \Omega\in C^{1,\beta}$  for some  $\beta \in (0, \beta_0)$. Assume that $\bfF \in C^{0, \beta} ( \Omega )$ and let $\bfu$ be the solution to the Dirichlet problem  \eqref{eq:sysA}.     Then $|\nabla \bfu|^{p-2} \nabla \bfu \in C^{0,  \beta}(\Omega )$, and there exists a  constant   $C=C(p,N, \beta, \Omega)$ such that
%\db{$\beta$ not $\beta/p-1$}
  \begin{align*}
    \norm{|\nabla \bfu|^{p-2} \nabla \bfu}_{C^{0,   \beta}(\Omega )} \leq
% c\,\dashint_{\Omega}\abs{ (A(\nabla
   %   u))}\, dx
  C\norm{\bfF}_{C^{0, \beta}(\Omega )}\,.
  \end{align*}
%\\ Moreover, if $\bfF \in \setVMO (\Omega )$, then $|\nabla \bfu|^{p-2} \nabla \bfu
%\in \setVMO (\Omega)$.
  %The constant $c$ depends only on the characteristics of $\phi,\gamma$ and the $C^{1,\gamma}$ constant of the boundary.
\end{corollary}
%\seb{It is shown below :-).}

Theorems \ref{thm:campanato}  and \ref{continuity} are  in fact consequences of  stronger pointwise estimates, of 
%independent interest and  
potential  use for other issues, between a sharp maximal function of $|\nabla \bfu|^{p-2} \nabla \bfu$ and a sharp maximal function of $\bfF$ -- see Propositions \ref{pro:wahnsinn} and \ref{pro:wahnsinn1}.
Our approach to these estimates entails the choice of  appropriate local coordinates, where the boundary of the
domain is flat, in the sense that the domain is locally mapped into a
half-ball, with a proper radius, after changing variables. Suitable decay oscillation estimates have to be established as the radii of the relevant balls tend to zero.
 Due to the minimal regularity required on $\partial \Omega$,  we have 
 to develop a new strategy, based on the selection of  an ad hoc coordinate systems taylored for each scale of the radii. This is a pivotal step, that makes it possible to derive sharp oscillation bounds, and is flexible enough for prospective implementations in other questions in the  global regularity theory  of elliptic boundary value problems.
%In particular,  suitable decay oscillation estimates have to be established as the radii of the relevant balls tend to zero.  Due to the minimal regularity required on $\partial \Omega$, customary  flattening techniques do not apply. We have thus to develop a new approach that entails the introduction of ad hoc coordinate systems adapted to each scale of the radii.
 Thanks to a suitable
continuation of the differential operator and of the solution beyond
the flattened boundary, the problem is reduced to inner regularity.
However, the new differential operator is not anymore the $p$-Laplacian, and, in particular, it is not of Uhlenbeck type. Therefore, 
standard inner local regularity results cannot be applied.
 A  subsequent task is thus to derive local estimates for  perturbed systems.
%This is a result of independent interest, stated in
%Theorem~\ref{thm:BMOomega}. 
The idea is that our regularity
assumption on $\partial \Omega$ allows for the perturbed
differential operator to be locally still sufficiently close to the
original one for the regularity of solutions not to be destroyed. Our auxiliary result in this connection is also of possible independent interest.

\section{A decay estimate near a flat boundary
%Systems with constant coeficients in a half space
}

The present section is devoted to a decay estimate for the gradient of solutions to the system in \eqref{eq:sysA} satisfying the Dirichlet boundary condition locally on a flat boundary. This is the content of Proposition \ref{cor:decay-hs}. 
%
%\seb{The core of decay estimates with inhomogeneous boundary values is the comparison with local $p$-harmonic functions that goes back to the seminal works of Iwaniecs~\cite{Iwa83} and Giaquinta and Modica~\cite{GiaMod86}.}
\par 
We begin our discussion by fixing  a few notations and conventions. The relation $\lq\lq \approx "$ between two real-vaued expressions means that they are bounded by each other, up to positive multiplicative constants depending on quantities to be specified.
\\ Given $m, n  \in \mathbb N$, we denote by $\setR^{m\times n}$ the space of $m \times n$ matrices, by 
$\lq\lq \cdot"$  the standard scalar product in $\setR^{m \times n}$, and by 
 $\abs{\, \cdot\, }$  the induced norm on $\setR^{m \times n}$. 
\\ A point $x\in \mathbb R^n$ will be regarded as a column vector, namely an element of $\setR^{n\times 1}$,
% $(x_1, \dots \, x_n)^t$, namely as a matrix in $\mathbb R^{n\times 1}$, 
although, for ease of notation, we shall write  $(x_1, \dots , x_n)$ when its components are relevant. We also set $x'=(x_1, \dots , x_{n-1})\in \mathbb R^{n-1}$, whence $x=(x', x_n)$ for $x\in \setR^n$. One has  that
$$ x \cdot y = x^t y = {\rm tr}(xy^t) \qquad \hbox{for  $x, y \in \mathbb R^n$,}$$
where the apex  $\lq\lq\, t\, "$ stands for transpose, and $\lq\lq {\rm tr}"$ for trace. More generally, if $\bfB, \bfC \in \mathbb R^{m\times n}$, then 
$$\bfB \cdot \bfC = {\rm tr}(\bfB \bfC^t) =  {\rm tr}(\bfB^t \bfC).$$
Also,  if $\bfB, \bfC \in \mathbb R^{m\times N}$ and $\bfP,\bfQ\in \setR^{N\times n}$, then
\begin{align}
\label{eq:matrix1}
\bfB \bfP \cdot \bfC \bfQ=\trace(\bfB \bfP ( \bfC \bfQ)^t)=\trace (\bfB \bfP \bfQ^t \bfC^t )= \bfB \bfP \bfQ^t\cdot \bfC.
\end{align}
Given a function
 $$w : \setR^{n} \to \setR,$$
 its gradient $\nabla w$ is   a row vector in $\setR^n$, namely  $\nabla w \in \setR^{1 \times n}$. 
%This is consistent with the fact that $\nabla v(x)$ is an element of the dual of the tangent space to $\setR^n$ at $x$. 
More generally, if 
$$\bfw : \setR^n \to \setR^N,$$
then $\nabla \bfw$ is the matrix in   $\setR^{N\times n}$ whose rows are the gradients of the components $w^1, \, \dots \, , w^N$ of $\bfw$.
% Hence, $\nabla \bfw \in \setR^{N\times n}$.  
With these conventions in place, if $\bfpsi : \setR^d \to \setR^n$, then
$$\nabla (\bfw \circ \bfpsi)= (\nabla \bfw\circ\bfpsi) \nabla \bfpsi,$$
where the product on the right-hand side is just the matrix product. In particular, if $\bfQ \in \setR^{d\times n}$, and 
$$\bfpsi (y)= \bfQ y \qquad \hbox{for $y \in \mathbb R^d$,}$$
then $\nabla \bfpsi =\bfQ$, and 
$$\nabla (\bfw (\bfQ y))= \nabla \bfw (\bfQ y)\bfQ \qquad \hbox{for $y \in \mathbb R^d$.}$$
Notice also that, if $\eta : \setR^n \to \setR$ and $\bfw : \setR^n \to \setR^N$, then
$$\nabla (\eta \bfw)= \eta \nabla \bfw + \bfw \nabla \eta,$$
where $ \bfw \nabla \eta \in \setR^{N\times n}$ is the matrix product between $\bfw \in \setR^{N\times 1}$ and $\nabla \eta \in \setR^{1\times n}$. Therefore, $ \bfw \nabla \eta $ agrees with the tensor product $\bfw \otimes \nabla \eta$.
%, when $\bfv$ and $\nabla \eta$ are regarded as elements from the same space.
\\
Given $p >1$, define the function $\bfA :  \setR^{N\times n} \to  \setR^{N\times n}$ as 
\begin{equation}\label{A}
\bfA (\bfQ) = |\bfQ |^{p-2}\bf Q \quad \hbox{for $\bfQ \in \setR^{N\times n}$\,.}
\end{equation}
If $\bfT \in \setR^{n\times n}$ we also denote by $\bfA _{\bfT} :  \setR^{N\times n} \to  \setR^{N\times n}$ the function defined by
\begin{equation}\label{AT}\bfA _{\bfT}(\bfQ) = \bfA(\bfQ \bfT)(\bfT)^t \quad \hbox{for $\bfQ \in \setR^{N\times n}$\,.}
\end{equation}

\smallskip
\noindent

Let us now recall a few notions of solutions.
Let $\bfF\in L^{p'}(\Omega)$. A function $\bu \in W^{1,p}_0(\Omega)$ is called a weak solution to problem \eqref{eq:sysA} if 
\begin{equation}\label{weaksol}
\int_{\Omega} \bfA(\nabla \bfu)\cdot \nabla \bfvarphi \dx = \int_{\Omega} \bfF \cdot \bfvarphi\dx
\end{equation}
for every function $\bfvarphi \in W^{1,p}_0(\Omega)$.
\\ 
Assume next that  $\bfF\in L^{p'}_{\rm loc}(\Omega)$. A function 
%$\bfu :  \Omega \to \mathbb R^{N}$ such that 
$\bfu \in W^{1,p}_{\rm loc}(\Omega)$ is called   a local weak solution to the system
 \begin{equation}\label{sysloc} -\divergence(\bfA (\nabla \bfu)) = -\divergence
  \bfF\,  \quad \hbox{in $\Omega$}
\end{equation}
if 
\begin{equation}\label{weaksolloc}
\int_{\Omega'} \bfA(\nabla \bfu)\cdot \nabla \bfvarphi \dx = \int_{\Omega'} \bfF \cdot \bfvarphi\dx
\end{equation}
for every function $\bfvarphi \in W^{1,p}_0(\Omega ')$ with $\Omega' \subset\subset \Omega$.
\\ Let $B_R$ be a ball centered on $\partial \Omega$, with radius $R$, and  let $\bfF\in L^{p'}(\Omega\cap B_R)$.  Assume that $\bu$ belongs to
 the closure in $W^{1,p}(\Omega\cap B_R)$ of the space of %\todo[inline]{check this definition and notation} 
those functions  in $C^\infty (\Omega\cap B_R)$ that vanish in a neighbourhood of  $\partial \Omega$. The $\bfu$ is called a weak solution to the problem
%
%We denote by $W^{1,p}_{0, \partial \Omega}(\Omega\cap B_R)$ the closure in $W^{1,p}(\Omega\cap B_R)$ of the space of %\todo[inline]{check this definition and notation} 
%those functions  in $C^\infty (\Omega\cap B_R)$ that vanish in a neighbourhood of  $\partial \Omega$. Let $\bfF\in L^{p'}(\Omega\cap B_R)$.  %\todo{define a proper space for the solution $u$} 
%A function %$\bfu   :  \Omega \to \mathbb R^{N}$ such that
%  $\bfu \in W^{1,p}_{0, \partial \Omega}(\Omega\cap B_R)$ is called   a local weak solution to the problem
\begin{align}
 \label{dirloc}
 \begin{cases}
-\divergence(\bfA (\nabla \bfu))=-\divergence \bfF  & \quad \text{ in }\Omega \cap B_R
\\
\bfu =0   & \quad \text{ on } \partial \Omega \cap B_R
\end{cases}
\end{align}
if
\begin{equation}\label{weaksolBR}
\int_{\Omega \cap B_R}\bfA(\nabla \bfu)\cdot \nabla \bfvarphi \dx = \int_{\Omega \cap B_R} \bfF \cdot \bfvarphi\dx
\end{equation}
for every function $\bfvarphi  \in W^{1,p}_0(\Omega \cap B_R)$.

\smallskip
\noindent

Given 
%an open set $\Omega  \subset \setR^n$, 
a matrix $\bfT \in \setR^{n \times n}$, with ${\rm det} \bfT \neq 0$, and a function $\bfw : \Omega \to \setR^N$, define the function
$$\overline \bfw : \bfT \Omega \to \setR^N$$ as 
\begin{equation}\label{wbar}\overline \bfw (y) = \bfw (\bfT^{-1}y) \quad \hbox{for $y \in \bfT \Omega$.}
\end{equation}
Then $$\nabla \overline \bfw (y)= \nabla \bfw (\bfT^{-1} y)\bfT^{-1} \quad \hbox{for $y \in \bfT \Omega$.}$$
Let  $\bfA _{\bfT^{-1}}$ be function defined as in \eqref{AT}, with $\bfT$ replaced by $\bfT^{-1}$.
%We also define the function $\bfA _{\bfT^{-1}} :  \setR^{N\times n} \to  \setR^{N\times n}$ as
%\begin{equation}\label{AT}\bfA _{\bfT^{-1}}(\bfQ) = \bfA(\bfQ \bfT^{-1})(\bfT^{-1})^t \quad \hbox{for $\bfQ \in \setR^{N\times n}$\,.}
%\end{equation}
%Also, set 
%$$\bfA _{\bfT^{-1}}(\nabla \bfu) = \bfA(\nabla \bfu \bfT^{-1})(\bfT^{-1})^t \quad \hbox{in $\Omega$.}$$
Assume that $\bfu$ is a local solution to the system %\db{D or B?}
\begin{align}\label{eq:pert}
-\divergence(\bfA_{\bfT^{-1}}(\nabla \bfu))=-\divergence \bfF \quad \text{ in }\Omega.
\end{align}
Let $\overline{\bfu}$ and 
$\overline\bfF$ be the functions built upon $ \bfu$ and $\bfF$ as in \eqref{wbar}.
We claim  that the function
$\overline{\bfu}$ is a local solution to the system
\begin{align}\label{eq:pertbar}
-\divergence(\bfA(\nabla \overline\bfu))=-\divergence(\overline\bfF\bfT^t)\quad \text{ in }\bfT \Omega\,.
\end{align}
As a consequence, local results available for the p-Laplacian are translated to systems with constant coefficients of the form \eqref{eq:pert}. Our claim follows from the following chain, that,
owing to \eqref{eq:matrix1},  holds for any function $\bfphi \in W^{1,p}_0(\Omega)$:  %\db{D or B?}
\begin{align}\label{constbis}
\int_{\bfT \Omega}& \bfA(\nabla \overline{\bfu}(y))\cdot \nabla \overline{\bfphi}(y)\dy =
\int_{\bfT \Omega}{\rm tr}[\bfA(\nabla \overline{\bfu}(y)) \nabla \overline{\bfphi}(y)^t]\dy
\\ \nonumber & =
\int_{\bfT \Omega}{\rm tr}[\bfA(\nabla {\bfu}(\bfT^{-1}y)\bfT^{-1}) ( \nabla{\bfphi} (\bfT^{-1}y) \bfT^{-1}) ^t]\dy
\\ \nonumber &
=
\int_{\bfT \Omega}{\rm tr}[\bfA(\nabla {\bfu}(\bfT^{-1}y)\bfT^{-1})  (\bfT^{-1})^t \nabla{\bfphi} (\bfT^{-1}y) ^t]\dy
\\ \nonumber & = \int_{\Omega} {\rm tr}[\bfA(\nabla {\bfu}(x)\bfT^{-1})  (\bfT^{-1})^t \nabla{\bfphi} (x) ^t]|{\rm det} \bfT|\, \dx
\\ \nonumber &
= \int_{\Omega} {\rm tr}[\bfA_{\bfT^{-1}}(\nabla {\bfu}(x))   \nabla{\bfphi} (x) ^t]|{\rm det} \bfT|\, \dx
%\\ \nonumber & 
= \int_{\Omega}  \bfA_{\bfT^{-1}}(\nabla {\bfu}(x))  \cdot \nabla{\bfphi} (x) |{\rm det} \bfT|\, \dx
\\ \nonumber &
= \int_{\Omega} \bfF (x)  \cdot \nabla{\bfphi} (x) |{\rm det} \bfT|\, \dx
%\\ \nonumber & 
= \int_{\bfT \Omega} \bfF (\bfT^{-1}y)  \cdot \nabla{\bfphi} (\bfT^{-1}y)  \, \dy
\\ \nonumber &
= \int_{\bfT \Omega} \overline \bfF ( y)  \cdot \nabla \overline {\bfphi} (y) \bfT \, \dy
%\\ \nonumber & 
= \int_{\bfT \Omega} {\rm tr} [\overline \bfF ( y)   (\nabla \overline {\bfphi} (y) \bfT)^t ]\, \dy
\\ \nonumber &
= \int_{\bfT \Omega} {\rm tr} [\overline \bfF ( y)   \bfT ^t \nabla \overline {\bfphi} (y)^t] \, \dy
%\\ \nonumber &
 = \int_{\bfT \Omega} \overline \bfF ( y)   \bfT ^t \cdot \nabla \overline {\bfphi} (y)  \, \dy.
\end{align}
%So by defining 
%\[
%\bfT D:=\set{\bfT x\,:\,x\in D},
%\]
%we find that 
%
%
%We are interested of coefficients of the following type. For $\bfT\in\setR^{N\times n}$ we define the related matrix $\bfT_N\in  R^{Nn\times Nn}$
%\begin{align*}
%\bfT_N:=
% \begin{pmatrix}
% \bfT&0&\ldots&0
%\\
%0&\bfT&0&\vdots
%\\
% \vdots&0&\ddots&0 \\ 0&\ldots&0&\bfT
%\end{pmatrix}.
%\end{align*}
Now, assume that the matrix $\bfT \in \mathbb R^{n\times n}$ is positive definite, with smallest eigenvalue $\lambda$ and largest eigenvalue $\Lambda$. In particular
\begin{align}
\label{ibp}
B_{\lambda r}(0)\subset \bfT(B_{r}(0))\subset B_{\Lambda r}(0)\text{ and }B_{\frac{r}{\Lambda}}(0)\subset \bfT^{-1}(B_{r}(0))\subset B_{\frac{r}{\lambda }}(0)
\end{align}
for every $r>0$. 
%Let us first consider local solutions to 
%\begin{align}
% \label{eq:pert}
%-\divergence(\bfA_\bfT(\nabla \bfu))=-\divergence(\bfF),\quad\text{for}\quad \bfA_\bfT(\nabla \bfu):=\bfT_N^t\bfA(\bfT_N\nabla \bfu),
%\end{align}
%here and below, we assume that for $ \bfg:\setR^n\to \setR^N$, differentiable that 
%\[
%\nabla \bfg:=(\partial_1\bfg^1,...,\partial_n\bfg^N,\partial_1\bfg^2,.....,\partial_N\bfg^N)^t
%\]
% is a column vector in $\setR^{nN}$.
% We define $\bar{\bfg}(y):=\bfg(\bfT^{-1}y)$. Please observe, that $\nabla_y\bar{\bfg}(y)=\bfT_N\nabla_x\bfg(\bfT^{-1}y)$
% If $\bfu$ is a solution to the respective system with constant coefficients\eqref{eq:pert} on $\Lambda B\supset \bfT^{-1} B$ , then we find for $\phi\in C^\infty_0(T(B))$, that for $\tau=\det(\bfT)$
%\begin{align}
%\label{const}
%\begin{aligned}
%\int_{B}\bfA(\nabla \bar{\bfu})\cdot \nabla \bar{\bfphi}\dy=\int_{\bfT^{-1}(B)}\bfA_{\bfT}(\nabla \bfu) \cdot \nabla \bfphi\tau \dx
%=\int_{\bfT^{-1}(B)}\bfF\cdot \nabla \bfphi\tau \dx=\int_{B}\Tau_N^t\tilde\bfF\cdot \nabla \bar{\bfphi}\dy.
%\end{aligned}
%\end{align}
 We consider solutions to systems of type \eqref{eq:pert} in a half-ball, subject to zero boundary conditions on the flat part of its boundary. 
Precisely, define 
$$
H^+=\Big\{x\in \setR^n:  x_n >0\Big\},
$$
and 
\begin{equation}\label{HB}
H_{\bfT^{-1}}=\Big\{x\in \setR^n: (\bfT^{-1} x)_n >0\Big\},
\end{equation}
and let $\bfu$ be a weak solution to the problem
\begin{align}
 \label{eq:halfspace1}
 \begin{cases}
-\divergence(\bfA_{\bfT^{-1}}(\nabla \bfu))=-\divergence \bfF  & \quad \text{ in }H^+\cap B_r(0)
\\
\bfu =0   & \quad \text{ on }\set{x_n=0}\cap B_r(0).
\end{cases}
\end{align}
%Repeating the argument of 
Choosing $\Omega= H^+\cap B_r(0)$ in
\eqref{eq:pert} tells us that $ \overline\bfu$ is a weak solution to the problem
%implies that
%$\overline{\bfu}(y):=\bfu(\bfT^{-1} y)$,  $\overline{\bfF}(y):=\bfF(\bfT^{-1} y)$ satisfy
\begin{align}
 \label{eq:halfspace2}
 \begin{cases}
-\divergence(\bfA(\nabla \overline\bfu))=-\divergence(\overline\bfF \bfT^t) & \quad \text{ in }H_{\bfT^{-1}}\cap B_{\lambda r}(0)
\\
\overline\bfu =0  & \quad \text{ on }\set{(\bfT^{-1} y)_n=0}\cap B_{\lambda r}(0).
\end{cases}
\end{align}
Since solutions are invariant under orthonormal transformations, we can make use of an even reflection with respect to the half-space $H_{ \bfT^{-1}}$, and obtain a local solution in an entire ball. To this purpose,  consider the linear map from $H_{\bfT^{-1}}$  into $H^+$ associated with an orthonormal matrix $\bfQ \in \mathbb R^{n \times n}$. 
%
%
%
%
%
%
%
%
%Precisely,
%\todo[inline]{Do we really need $\bfB$, or can we just take $\bfB = \bfI$? See for instance the proof of Proposition  \ref{pro:boundarydecay}}
% observe that $(\bfB x)_n=\sum_{j=1}^n \bfB_{nj}x_j$, set 
%\[
%H_{\bfB}=\Big\{x\in \setR^n: (\bfB x)_n >0\Big\},
%\]
%and let $\bfu$ be a weak solution to the problem
%\begin{align}
% \label{eq:halfspace1}
% \begin{cases}
%-\divergence(\bfA_{\bfT^{-1}}(\nabla \bfu))=-\divergence \bfF  & \quad \text{ in }H_\bfB\cap B_r(0)
%\\
%\bfu =0   & \quad \text{ on }\set{(\bfB x)_n=0}\cap B_r(0).
%\end{cases}
%\end{align}
%%Repeating the argument of 
%Choosing $\Omega= H_\bfB\cap B_r(0)$ in
%\eqref{eq:pert} tells us that $ \overline\bfu$ is a weak solution to the problem
%%implies that
%%$\overline{\bfu}(y):=\bfu(\bfT^{-1} y)$,  $\overline{\bfF}(y):=\bfF(\bfT^{-1} y)$ satisfy
%\begin{align}
% \label{eq:halfspace2}
% \begin{cases}
%-\divergence(\bfA(\nabla \overline\bfu))=-\divergence(\overline\bfF \bfT^t) & \quad \text{ in }H_{\bfB\bfT^{-1}}\cap B_{\lambda r}(0)
%\\
%\overline\bfu =0  & \quad \text{ on }\set{(\bfB\bfT^{-1} y)_n=0}\cap B_{\lambda r}(0).
%\end{cases}
%\end{align}
%Since solutions are invariant under orthonormal transformations, we can make use of an even reflection with respect to the half-space $H_{\bfB \bfT^{-1}}$, and obtain a local solution in an entire ball. To this purpose, set $H^+ = H_{\bfI}$ and consider the linear map from $H_{\bfB \bfT^{-1}}$  into $H^+$ associated with an orthonormal matrix $\bfQ \in \mathbb R^{n \times n}$. 
%
 Also, we define $\widehat{\bfu}$ and $\widehat{\bfF}$ as the composition of the inverse of this transformation with $\overline{\bfu}$ and $\overline{\bfF}$, respectively. Namely, we define $\widehat{\bfu}, \widehat{\bfF} : H^+\cap B_{\lambda r}(0) \to \mathbb R^N$ as 
%
%
%further introduce\marginpar{DB: why not write $x=y=z$?} 
$$\widehat{\bfu}(z)=\overline{\bfu}(\bfQ^t z)=\bfu(\bfT^{-1}\bfQ^tz)$$ and 
$$\widehat{\bfF}(z)=\overline{\bfF}(\bfQ^t z)=\bfF({\bfT^{-1}\bfQ^t}z)$$
for $z \in H^+\cap B_{\lambda r}(0)$ .
On making use of the fact that  $\bfQ^t=\bfQ^{-1}$, the argument above implies that $\widehat\bfu$ is a weak solution to the problem
\begin{align}
 \label{eq:halfspace3}
 \begin{cases}
-\divergence(\bfA(\nabla \widehat\bfu))=-\divergence(\widehat{\bfF}\bfT^t\bfQ^t)& \quad \text{ in }H^+\cap B_{\lambda r}(0)
\\
\widehat\bfu=0 & \quad \text{ in }\set{x_n=0}\cap B_{\lambda r}(0).
\end{cases}
\end{align}
Let  $\bfR\in \setR^{N\times n}$ be the matrix given by $\bfR =\text{diag}(1,...,1,-1)$. Set
$$B_{r}^+(0)= H^+ \cap B_{r}(0) \quad \hbox{and} \quad B_{r}^-(0) = B_{r}(0) \setminus B_{r}^+(0)\quad \hbox{for $r>0$\,.}$$
The function $\bfv : B_{\lambda r}(0) \to \mathbb R^N$, defined as
\begin{align*}
\bfv(x)= \begin{cases}
\widehat\bfu(x',x_n) & \quad \textrm{for } (x',x_n)\in B_{\lambda r}^+(0)\\
-\widehat\bfu(x',-x_n) & \quad \textrm{for } (x',x_n)\in B_{\lambda r}^-(0),
\end{cases}
\end{align*}
belongs to $W^{1,p}_{\rm loc}(B_{\lambda r}(0))$,   and 
\begin{align*}
\nabla \bfv(x)=\ \begin{cases}
\nabla\hat\bfu(x',x_n) & \quad \textrm{for }  (x',x_n)\in B_{\lambda r}^+(0)\\
-\nabla \widehat\bfu(x',-x_n) \bfR & \quad\textrm{for }(x',x_n)\in B_{\lambda r}^-(0).
\end{cases}
\end{align*}
Let $\bfC\in \RNn $. We will show that, if $\bfG :  B_{\lambda r}(0) \to \mathbb R^{N\times n}$ is defined as 
%\begin{align*}
%\bfG(x,h)&:=\left\{
%\begin{aligned}
%\widehat{\bfF}(x',h)\bfT_1^{t}\bfQ^t-\bfc_1 &&\textrm{for } \set{(x',h)\in B_{\lambda r}(0)\,:\,h\geq 0},
%\\
%\widehat{\bfF}\bfT_1^{t}\bfQ^t(x',-h)-\bfc_1)\bfR  &&\textrm{for } \set{(x',h)\in B_{\lambda r}(0)\,:\,h<0},
%\end{aligned}
%\right,
%\end{align*}
\begin{align*}\bfG(x)=\begin{cases}
\widehat{\bfF}(x',x_n)\bfT^{t}\bfQ^t-\bfC& \quad \textrm{for } (x',x_n)\in B_{\lambda r}^+(0)
\\
-\big(\widehat{\bfF}(x',-x_n)\bfT^{t}\bfQ^t-\bfC\big)\bfR  & \quad \textrm{for } (x',x_n)\in B_{\lambda r}^-(0),
\end{cases}
\end{align*}
then 
 $\bfv$ is a  local solution to the system
\begin{align}
\label{eq:local}
-\divergence(\bfA(\nabla \bfv))=-\divergence \bfG \quad \text{ in }B_{\lambda r}(0).
\end{align}
%
%Firstly, clearly $\bfv\in W^{1,p}(B_{\lambda r}(0))$. 
%Secondly, we have the following symmetries
%\begin{align*}
%\nabla \bfg(x',h)&=-\bfR_N\nabla \bfg(x',-h)\text{ if }\bfg(x',h)=-\bfg(x',-h)
%\\
%\nabla \bfg(x',h)&=\bfR_N\nabla \bfg(x',-h)\text{ if }\bfg(x',h)=\bfg(x',-h).
%\end{align*}
To verify this assertion,  note that  any function $\bfphi\in C^{\infty}_0(B_{\lambda r}(0))$ can be decomposed as
 $$
 \bfphi(x)=\bfphi_1(x)+\bfphi_2(x) \qquad \hbox{for $x \in B_{\lambda r}(0)$,}
 $$
where we have set 
$$\bfphi_1(x', x_n)= 
\frac{\bfphi(x',x_n)+\bfphi(x',-x_n)}{2}\quad \hbox{and} \quad  \bfphi_2(x', x_n)=\frac{\bfphi(x',x_n)-\bfphi(x',-x_n)}{2} $$
for $(x', x_n) \in B_{\lambda r}(0)$.
In particular,
 \[
 \nabla \bfphi_1(x',x_n)= \nabla \bfphi_1(x',-x_n)\bfR\quad \text{ and } \quad \nabla \bfphi_2(x',x_n)= -\nabla \bfphi_2(x',-x_n)\bfR
 \]
for $(x', x_n) \in B_{\lambda r}(0)$. Also,
 $\bfphi_2 \in W^{1,p}_0(B_{\lambda r}^+(0))$. Hence, % we calculate\db{I don't see why $\bfphi_2=0$ in $B_{\lambda r}^-$}
\begin{align*}
&\int_{B_{\lambda r}(0)}\big(\bfA(\nabla \bfv)-\bfG\big)\cdot \nabla \bfphi\dx
\\
&\quad
 =\int_{B_{\lambda r}^+(0)} \big(\bfA(\nabla \bfv)-\bfG\big) \cdot \nabla \bfphi_1\dx
+
\int_{B_{\lambda r}^-(0)} \big(\bfA(\nabla \bfv)-\bfG\big)\cdot \nabla \bfphi_1\dx
\\
&\qquad +\int_{B_{\lambda r}^+(0)} \big(\bfA(\nabla \bfv)-\bfG\big)\cdot\nabla \bfphi_2\dx
+\int_{B_{\lambda r}^-(0)} \big(\bfA(\nabla \bfv)-\bfG\big)\cdot\nabla \bfphi_2\dx
\\
&\quad 
= \int_{B_{\lambda r}^+(0)} \big(\bfA(\nabla \bfv)-\bfG\big) \cdot \nabla \bfphi_1\dx
\\
&\qquad +\int_{B_{\lambda r}^-(0)} -\big(\bfA(\nabla \bfv(x',-x_n)) -(\widehat{\bfF}(x',-x_n)\bfT^{t\color{black}}\bfQ^t-\bfC)\big)\bfR \cdot \nabla \bfphi_1(x',x_n)\dx'\dx_n
\\
&\qquad+\int_{B_{\lambda r}^-(0)} -\big(\bfA(\nabla \bfv(x',-x_n)) -(\widehat{\bfF}(x',-x_n)\bfT^{t \color{black}}\bfQ^t-\bfC)\big)\bfR \cdot \nabla \bfphi_2(x',x_n)\dx'\dx_n
\\
&\quad 
= \int_{B_{\lambda r}^+(0)} \big(\bfA(\nabla \bfv)-\bfG\big) \cdot \nabla \bfphi_1\dx
\\
&\qquad
-\int_{B_{\lambda r}^-(0)} \big(\bfA(\nabla \bfv(x',-x_n)) -(\widehat{\bfF}(x',-x_n)\bfT^{t}\bfQ^t-\bfC)\big)\bfR \cdot \nabla \bfphi_1(x',-x_n)\bfR\dx'\, \dx_n
\\
&\qquad+\int_{B_{\lambda r}^-(0)} \big(\bfA(\nabla \bfv(x',-x_n)) -(\widehat{\bfF}(x',-x_n)\bfT^{t}\bfQ^t-\bfC)\big)\bfR \cdot \nabla \bfphi_2(x',-x_n)\bfR\dx'\, \dx_n\,.
\end{align*}
Note that in this chain we have made use of \eqref{eq:matrix1},  of the fact that $\bfR=\bfR^t=\bfR^{-1}$, and of the equality
\begin{equation}\label{a1}
\int_{B_{\lambda r}^+(0)} \big(\bfA(\nabla \bfv)-\bfG\big)\cdot\nabla \bfphi_2\dx =0\,,
\end{equation}
which holds since $\bfv = \widehat \bfu$ in $B_{\lambda r}^+$ and $\widehat \bfu$ is a solution to problem  \eqref{eq:halfspace3}.
Consequently,
\begin{align*}
&\int_{B_{\lambda r}(0)}\big(\bfA(\nabla \bfv)-\bfG\big)\cdot \nabla \bfphi\dx
\\
&\quad = \int_{B_{\lambda r}^+(0)} \big(\bfA(\nabla \bfv)-\bfG\big) \cdot \nabla \bfphi_1\dx
\\
&\qquad
-\int_{B_{\lambda r}^-(0)} \big(\bfA(\nabla \bfv(x',-x_n)) -(\widehat{\bfF}(x',-x_n)\bfT^{ t \color{black}}\bfQ^t-\bfC)\big) \cdot \nabla \bfphi_1(x',-x_n)\dx'\dx_n
\\
&\qquad+\int_{B_{\lambda r}^-(0)} \big(\bfA(\nabla \bfv(x',-x_n)) -(\widehat{\bfF}(x',-x_n)\bfT^{ t \color{black}}\bfQ^t-\bfC)\big) \cdot \nabla \bfphi_2(x',-x_n)\dx'\dx_n
\\
&\quad = \int_{B_{\lambda r}^+(0)} \big(\bfA(\nabla \bfv)-\bfG\big) \cdot \nabla \bfphi_1\dx
-\int_{B_{\lambda r}^+(0)} \big(\bfA(\nabla \bfv) -\bfG\big) \cdot \nabla \bfphi_1\dx 
\\
&\qquad+\int_{B_{\lambda r}^+(0)} \big(\bfA(\nabla \bfv) -\bfG\big) \cdot \nabla \bfphi_2\dx =0.
\end{align*}
By the density of  the space $C^\infty _0(B_{\lambda r}(0))$ in $W^{1,p} _0(B_{\lambda r}(0))$, this implies that $\bfv$ is a local weak solution to system \eqref{eq:local}.%\db{ I don't see why last integral vanishes}
\\
Let us now 
set
 \begin{equation}\label{Ds}
D_s = H_{\bfT^{-1}} \cap B_s(0)
\end{equation}
for $s\in (0,\lambda r)$, 
whence $D_s=\set{\bfQ^t z:z\in B_s^+(0)}$. Also,    define the matrix
$\overline{\bfA}_s\in \setR^{N\times n}$  by
\begin{align}
\label{Abar}\begin{cases}
 (\overline{\bfA}_s)_{ki}=0 & \quad \text{ for }k\in \set{1,...,N}, i\in\set{1,...,n-1} 
\\  (\overline{\bfA}_s)_{kn}=\mean{\bfA_{kn}(\nabla \overline{\bfu})}_{D_s}& \quad\text{ for }k\in\set{1,...,N}.
\end{cases}
\end{align}
%The main result of this section is summarized in the following proposition. 
We are now ready to state and prove the main result of this section.
In the statement, we keep in force the notations introduced above.

\begin{proposition} \label{cor:decay-hs}
Let $p>1$, $r>0$,
%\infty$.
and let $\omega$ be a parameter function satisfying condition \eqref{eq:omega condition}. Assume that $\bfF \in L^{p'}(H^+ \cap B_r(0))$. Let $\bfu$ be a local weak solution to problem \eqref{eq:halfspace1} and let $\overline{\bfu}$ be the corresponding weak solution to problem \eqref{eq:halfspace2}. 
There exist constants $c>0$ and $\theta\in (0,1)$, depending only on $n,N,p,c_\omega,\beta,\lambda, \Lambda$, 
such that
\begin{align}\label{june43-hs}
\frac{1}{\omega(\theta s)} & \bigg(\dashint_{D_{\theta s}} \abs{\bfA(\nabla \overline{\bfu})-\overline\bfA_{\theta s}}^{\min\set{2,p'}}\dy\bigg)^\frac{1}{\min{\set{2,p'}}}
 \\ \nonumber &  \leq   \frac{1}{2\omega(s)}
 \bigg(\dashint_{D_s} \abs{\bfA(\nabla \overline{\bfu})-\overline\bfA_{s}}^{\min\set{2,p'}}\dy\bigg)^\frac{1}{\min{\set{2,p'}}}
% \\ \nonumber & \quad
+   \frac{c}{\omega(s)} \bigg(\dashint_{D_s}\abs{\overline{\bfF}
-{\bfF}
 _0}^{p'}\dy\bigg)^\frac1{p'}
\end{align}
for every  ${\bfF}_0\in \RNn$ and for $s \in (0,  \lambda r)$.
\end{proposition}

%Proposition \ref{cor:decay-hs} will be derived from the following result contained in 
% \cite[Inequality (3.11)]{BCDKS}. 

Proposition \ref{cor:decay-hs} will be derived from the following  inner local decay estimate   contained in  \cite[Inequality (3.11)]{BCDKS}. Earlier estimates in a similar spirit  can be traced back to~\cite{CafPer, GiaMod86, Iwa83}.

\begin{proposition} \label{prop:decay conclusion}  {\rm {\bf [\cite{BCDKS}]}}  
Let $p>1$,
and let $\omega$ be a parameter function satisfying condition \eqref{eq:omega condition}. Let $\Omega$ be an open set in $\mathbb R^n$. Assume that $\bfF \in L^{p'}_{\rm loc}(\Omega)$. Let $\bfu$ be a local weak solution to problem \eqref{sysloc}.
There exist  constants $c>0$ and  $\theta\in (0,1)$, depending only on $n,N,p,c_\omega,\beta$, such that
\begin{align}\label{june43}
&\frac{1}{\omega(\theta r)}  \bigg(\dashint_{B_{\theta r}} \abs{\bfA(\nabla \bfu)-\mean{\bfA(\nabla \bfu)}_{ B_{\theta r}}}^{\min\set{2,p'}}\dx\bigg)^\frac{1}{\min{\set{2,p'}}}
 \\ \nonumber &  \quad \leq   \frac{1}{2\omega(r)}
 \bigg(\dashint_{B_r} \abs{\bfA(\nabla \bfu)-\mean{\bfA(\nabla \bfu)}_{B_r}}^{\min\set{2,p'}}\dx\bigg)^\frac{1}{\min{\set{2,p'}}}
% \\ \nonumber & \quad
+   \frac{c}{\omega(r)} \bigg(\dashint_{B_r}\abs{\bfF
-\bfF
 _0}^{p'}\dx\bigg)^\frac1{p'}
\end{align}
for  every $\bfF_0\in \RNn$ and every  ball $B_r \subset \subset \Omega$.
\end{proposition}
%The results can be found precisely in~\cite[(3.11)]{BCDKS}. The above can be used to show a decay for a system with constant coeficients in a given half space. However, it is apparent already in the case of the half space $H^+$, that some quantities are reflected in an even way and others in an odd way. Since we are interested in oscillation estimates we have to keep track of the mean values.

The following observations also play a role in the proof of Proposition \ref{cor:decay-hs}. Assume that $\bff$ is a real, vector or matrix-valued function on $ B_r(0)$  such that $\bff \in L^q(B_r(0))$ for some $q\geq 1$. If 
$\bff(x',x_n)=\bff (x',-x_n)$  for a.e. $(x', x_n) \in B_r(0)$ , then
\begin{align}
\label{boundarybmo1}
&\bigg(\dashint_{B_r(0)} \abs{\bff-\mean{\bff}_{B_r(0)}}^q\dx\bigg)^\frac1q
\leq 2\bigg(\dashint_{B_r(0)} \abs{\bff-\mean{\bff}_{B_r^+(0)}}^q\dx\bigg)^\frac1q
\\ \nonumber &=2\bigg(\dashint_{B_r(0)^+} \abs{\bff-\mean{\bff}_{B_r^+(0)}}^q\dx\bigg)^\frac1q
\leq {4}\bigg(\dashint_{B_r(0)^+} \abs{\bff-\mean{\bff}_{B_r(0)}}^q\dx\bigg)^\frac1q \\ \nonumber &={4}\bigg(\dashint_{B_r(0)} \abs{\bff-\mean{\bff}_{B_r(0)}}^q\dx\bigg)^\frac1q.
\end{align}
If  $\bff(x',x_n)=-\bff(x',-x_n)$  for a.e. $(x', x_n) \in B_r(0)$, then, plainly, $\mean{\bff}_{B_r(0)}=0$. Thus,
\begin{align}
\label{boundarybmo0}
\bigg(\dashint_{B_r(0)} \abs{\bff-\mean{\bff}_{B_r(0)}}^q\dx\bigg)^\frac1q&=\bigg(\dashint_{B_r(0)} \abs{\bff}^q\dx\bigg)^\frac1q.
%= \bigg(\dashint_{B_r^+(0)} \abs{\bff}^q\dx\bigg)^\frac1q\\
%& \nonumber =\bigg(\dashint_{B_r^+(0)} \abs{\bff-\mean{\bff}_{B_r^+(0)}}^q\dx\bigg)^\frac1q.
\end{align}
% under the additional assumption that $\mean{\bff}_{B_r^+(0)}=0$, one has that 
%\begin{align}
%\label{boundarybmo2}
%\bigg(\dashint_{B_r(0)} \abs{\bff-\mean{\bff}_{B_r(0)}}^q\dx\bigg)^\frac1q&=\bigg(\dashint_{B_r(0)} \abs{\bff}^q\dx\bigg)^\frac1q
%= \bigg(\dashint_{B_r^+(0)} \abs{\bff}^q\dx\bigg)^\frac1q\\
%& \nonumber =\bigg(\dashint_{B_r^+(0)} \abs{\bff-\mean{\bff}_{B_r^+(0)}}^q\dx\bigg)^\frac1q.
%\end{align}
Note that, in inequality \eqref{boundarybmo1}, we have made use of the 
fact that, if $E$ is a measurable subset of $\setR^n$,  and $q \in [1,
\infty]$, %then \db{This has been used already before, for instance in (2.21). Postpone it?}
\begin{equation}\label{meanq}
\|\bff - \mean{\bff}_E\|_{L^q(E)} \leq 2\min _{\bfc }\|\bff - \bfc\|_{L^q(E)},
\end{equation}
% for any measurable set $E$ in $\Rn$, any $q \in [1,\infty]$, and
for every measurable function $\bff : E \to \setR^m$ such that $\bff \in L^q(E)$, where the minimum is extended over all $\bfc$ in the range of $\bff$.  This basic property will be  repeatedly  expolited in what follows.

\begin{proof}[Proof of Proposition \ref{cor:decay-hs}]
Analogously to \eqref{Abar}, for $s\in (0,\lambda r)$ we 
 define the matrix $\widehat{\bfA}_s\in \setR^{N\times n}$   as
\begin{align}\label{Ahat}
\begin{cases}
 (\widehat{\bfA}_s)_{ki}=0& \quad \text{ for }k\in \set{1,...,N}, i\in\set{1,...,n-1} \\
  (\widehat{\bfA}_s)_{kn}=\mean{\bfA(\nabla \widehat{\bfu})_{kn}}_{B_s^+(0)}& \quad \text{ for }k\in\set{1,...,N}.
\end{cases}
\end{align}
%Hence we introduce $\widehat\bfD_s\in \setR^{nN}$ for $s\in (0,\lambda r)$ by
%\begin{align}
% \bfA(\widehat\bfD_s)_{i}^{k}=0\text{ for }i\in\set{1,...,n-1},k\in \set{1,...,N} \text{ and }  (\bfA(\widehat\bfD_s))_{n}^{k}=\mean{\bfA^k_n(\nabla \widehat\bfu)}_{B_s^+}\text{ for }\set{1,...,N}
%\end{align}
Given ${\bfF}_0\in \RNn$,  choose $\bfC={\bfF}_0 \bfT^t \bfQ$. From
Proposition~\ref{prop:decay conclusion}, applied  to the solution $\bfv$ to system \eqref{eq:local}, we deduce, via \eqref{eq:halfspace3}, that 
\begin{align}\label{june43-hs-I}
\frac{1}{\omega(\theta s)} & \bigg(\dashint_{B_{\theta s}^+(0)} \abs{\bfA(\nabla \widehat{\bfu})-\widehat\bfA_{\theta s}}^{\min\set{2,p'}}\dz\bigg)^\frac{1}{\min{\set{2,p'}}}
 \\ \nonumber &  \leq   \frac{1}{2\omega(s)}
 \bigg(\dashint_{B_s^+(0)} \abs{\bfA(\nabla \widehat{\bfu})-\widehat\bfA_{s}}^{\min\set{2,p'}}\dz\bigg)^\frac{1}{\min{\set{2,p'}}}
% \\ \nonumber & \quad
+   
\frac{c}{\omega(s)} \bigg(\dashint_{B_s^+(0)}\abs{\widehat{\bfF}-{\bfF}_0}^{p'}\dz\bigg)^\frac1{p'}.
\end{align}
Observe that a proof of inequality  \eqref{june43-hs-I} also calls into play the property   that $\bfA(\nabla \widehat{\bfu})_{ki}$ is odd in the variable $x_n$, and $(\widehat{\bfA}_s)_{ki}=0$ if  $k\in \set{1,...,N}$, $i\in\set{1,...,n-1}$, whence \eqref{boundarybmo0} can be exploited, whereas $\bfA(\nabla \widehat{\bfu})_{kn}$ is even, and hence \eqref{boundarybmo1} can be exploited. 
Now, 
 since $\bfQ$ is an orthonormal matrix, 
% \db{Ony if $\mean{\bfA^k_n(\nabla \widehat{\bfu})}_{B_s^+}=0$?}
$$
\mean{\bfA_{kn}(\nabla \widehat{\bfu})}_{B_s^+(0)}=\mean{\bfA_{kn}(\nabla \overline{\bfu})}_{D_s(0)}.
$$
Hence, inequality \eqref{june43}  follows from \eqref{june43-hs-I}, via a change of variables.
\end{proof}

\section{A decay estimate near a non-flat  boundary}
\label{sec:nonflat}

Our task in the present section is to establish an inequality in the spirit of \eqref{june43-hs} for local solutions $\bfu$ 
%to 
%the $p$-Laplace system satisfying a zero boundary condition on a non-flat boundary, namely solutions 
to problem \eqref{dirloc} in the case when $\partial \Omega \cap B_R$ is not necessarily contained  in a hyperplane.
%In view of a proof of Proposition \ref{pro:boundarydecay}, we begin by introducing an ad hoc  change of coordinates whose effect is that of flattening $\partial \Omega \cap B_R$, and we then reflect our solution across the flattened boundary. 
Decay estimates at the boundary for solutions to $p$-Laplacian type equations are available in the literature. For instance, they can be found in the paper \cite{KinZho01}, where  the case of  boundaries of class $C^{1,\beta}$ is reduced, via a suitable change of cooordinates, to that of a flat boundary  treated in \cite{Lie88}. A flattening technique, combined with a reflection argument, is also 
exploited in \cite{ChenDiBe} to treat  systems. Neither the approach of \cite{KinZho01}, nor that of  \cite{ChenDiBe}, however, applies to deal with boundaries under as weak regularity assumptions as those imposed in this paper. We have thus to resort to a new method adapted to the situation at hand.

\subsection{A Gehring type result near the boundary}

One ingredient in our proof of the decay estimate near the boundary is a higher integrability result for the gradient of the solution to system \eqref{eq:sysA}. This is stated in the following proposition, that applies to any open bounded set $\Omega \subset \setR^n$  such that
\begin{equation}\label{eq:fat}
|B \cap \Omega| \geq C |B|
\end{equation}
for some constant $C>0$ and every ball $B$ centered at a point in $\Omega$.
% Clearly, condition \eqref{eq:fat} is fuliflled if $\Omega$ is a bounded Lipschitz domain, with $C$. depending on $n$ and on the constants $\lambda$ and $\Lambda$ appearing in \eqref{inclusion1} and \eqref{inclusion2}.
\\ In what follows, given a ball $B$ and a positive number $\theta$, we denote by $\theta B$ the ball with the same center as $B$, whose radius is $\theta$ times the radius of $B$.

\begin{proposition}
  \label{cor:VPL} Let $\Omega$ be an open bounded subset of $\setR^n$ fulfilling condition \eqref{eq:fat}.
Let $p>1$ and let $N \geq 1$. There exist constants $q_0>1$ and $c>0$, depending on $n$, $N$, $p$ and on the constant $C$ appearing in \eqref{eq:fat}, such that if $q \in [1, q_0]$, $\bfF \in L^{p'q}(\Omega)$ and
$\bfu$ is the  solution to the Dirichlet problem \eqref{eq:sysA}, then 
%
%
% satisfying \eqref{eq:dirichlet} and \eqref{eq:dirichlet2}. Then 
%$\bfu\in W^{1,pq_0}(\Omega)$ with $q_0$  defined via Gehring's lemma above in \eqref{gehring1}. 
%Moreover, there exists a constant $c$ depending on $n,N,p,q$ and the constant appearing in \eqref{eq:fat}, such that, by
% extending $\bfu$ and $\bfF$ by $0$ outside $\Omega$,
% Then its zero extension satisfies $\bfu\equiv \bfu\chi_\Omega\in W^{1,pq_0}(\setR^n)$ with $q_0$  defined via Gehring's lemma above in \eqref{gehring1}. 
  \begin{align}\label{Nov2}
    &\dashint_B \abs{\nabla \bfu}^{pq}\dx \leq c\,
   \bigg(\dashint_{2B} \abs{\nabla\bfu}\dx \bigg)^{p q}+c \dashint_{2B} \abs{\bfF -\bfF_0}^{p'q}\dx
  \end{align}
  and 
  \begin{align}\label{Nov3}
    &\dashint_B \abs{\bfA(\nabla \bfu)}^{p'q}\dx \leq c\,
   \bigg(\dashint_{2B} \abs{\bfA(\nabla\bfu)}\dx \bigg)^{p' q}+ c \dashint_{2B} \abs{\bfF -\bfF_0}^{p'q}\dx.
  \end{align}
for every matrix $\bfF_0\in\setR^{N \times n}$ and every  ball $B\subset\setR^n$. Here, $\bfF$ and $\bfu$ are extended by $0$ outside $\Omega$.
%  we have that
 %The constants only depend on $n,N,p,q$ and the constant appearing in \eqref{eq:fat}.
\end{proposition}

\begin{proof} A key step in 
the proof of inequalities \eqref{Nov2} and \eqref{Nov3} is a reverse H\"older type inequality, which tells us that that
 \begin{align}\label{Nov4}
      \dashint_B \abs{\nabla \bfu}^p\dx   \leq c\,
     \bigg(\dashint_{2B} \abs{\nabla \bfu}^{\theta p}\dx  \bigg)^{\frac 1\theta}
   + c \dashint_{2B} \abs{\bfF -\bfF_0}^{p'}\dx,
  \end{align}
for some constant $c$ depending on $n$,  $N$, $p$ and on the constant $C$ in \eqref{eq:fat}, and for every matrix $\bfF_0\in\setR^{N \times n}$ and every  ball $B\subset\setR^n$. Here, $\theta =\max\big\{ \frac{n}{n+p}, \frac 1p\big\}$.  In order to prove inequality \eqref{Nov4}, let us distinguish into some cases. If $\frac{3}{2}B\subset \Omega$,  inequality \eqref{Nov4} follows from
%\seb{\todo[inline]{Insert precise lemma}
\cite[Remark 6.12]
%with $P=0$
{Giu}, via  a standard covering argument.  If $\Omega \cap \frac32B = \emptyset$ the result is trivial. It remains to consider the case when ${\partial\Omega}\cap \frac32B\neq \emptyset$. 
Choose a function $\eta \in C^\infty_0(2B)$ such that $0 \leq \eta \leq 1$,   $\eta =1$ in $B$
and $\abs{\nabla \eta} \leq \frac{c}{R}$ for some absolute constant $c$, where $R$ denotes the radius of~$B$.
  Let $\alpha \geq\qbar$, whence $(\alpha-1)\pbar' \geq \alpha$.  Set
  $\bfxi = \eta^{\alpha} \bfu$, and let $\bfF_0\in \RNn$.  Making use of the function $\bfxi$ as a test
  function in the weak formulation \eqref{weaksol} of system \eqref{eq:sysA} yields
%\eqref{eq:pert} 
	\begin{equation}\label{Nov5}
	\int _\Omega \bfA(\nabla \bfu)\cdot \nabla(\eta^\alpha\bfu)\, \dx =\int_\Omega (\bfF- \bfF_0)\cdot  {\nabla(\eta^\alpha\bfu )}\, \dx.
	\end{equation}
	Thereby,
  \begin{align}\label{Nov6}
 \dashint_{2B}\abs{\nabla \bfu}^p\eta^\alpha\, \dx & =  \dashint _{2B}\bfA(\nabla \bfu)\cdot \eta^\alpha\nabla \bfu\, \dx
    \\ \nonumber & = \dashint _{2B} (\bfF-\bfF_0)\cdot \eta^\alpha \nabla \bfu\, \dx+ 
\dashint _{2B} (\bfF-\bfF_0) \cdot \alpha
      \eta^{\alpha-1} \bfu \otimes \nabla \eta\, \dx
    \\ \nonumber
    &\quad 
-  \dashint _{2B} \bfA(\nabla \bfu) \cdot \alpha \eta^{\alpha-1} \bfu\otimes \nabla \eta \,\dx.
    %\\ \nonumber
    %&=: (II) + (III) + (IV).
  \end{align}
%  Due to the coercivity of $\bfA$ we find
%  \begin{align}\label{Nov7}
%     \dashint _{2B}\bfA(\nabla \bfu)\cdot \eta^\alpha\nabla \bfu\, \dx \geq \dashint_{2B}\abs{\nabla \bfu}^p\eta^\alpha\, \dx
%  \end{align}
  %Since $(\alpha-1)\pbar' \geq \alpha$, 
By Young's inequality, there exist  positive constants $c$ and $c'$ such that
  \begin{align}\label{Nov8}
    \dashint _{2B} (\bfF-\bfF_0)\cdot \eta^\alpha \nabla \bfu\, \dx
    &\leq c\, \delta^{1-p'}\, \dashint_{2B}
    \abs{\bfF -\bfF_0}^{p'} \, \dx + \delta \dashint_{2B}\abs{\nabla \bfu}^p\eta^\alpha\, \dx
  \end{align}
for $\delta>0$,
% Owing to Young's inequality again,
  \begin{align}\label{Nov9}
    \dashint _{2B} (\bfF-\bfF_0) \cdot \alpha
      \eta^{\alpha-1} \bfu \otimes \nabla \eta\, \dx &\leq c\, \dashint_{2B} \abs{\bfF -\bfF_0}^{p'}
    \,\dx + c \dashint_{2B} \Bigabs{\frac{\bfu}{R}}^p \,\dx ,
  \end{align}
%By \eqref{eq:young} we obtain for any 
and
  \begin{align}\label{Nov10}
   \dashint _{2B} \bfA(\nabla \bfu) \cdot \alpha \eta^{\alpha-1} \bfu\otimes \nabla \eta \,\dx & \leq c\dashint_{2B} \alpha \eta^{\alpha-1}\abs{\bfA(\nabla\bfu)}\Bigabs{\frac{\bfu}{R}}\,\dx
%\leq c\,\dashint_{2B} \abs{\nabla\bfu}^{p-1}\, \eta^{\alpha-1} \frac{\abs{\bfu}}{R}\,\dx
    \\ \nonumber
    &\leq \delta \dashint_{2B} \eta^{(\alpha-1) p'}\abs{\nabla \bfu}^p  \,\dx + c'\, \delta^{1-p} \dashint_{2B} \Bigabs{\frac{\bfu}{R}}^p\,\dx
\end{align}
for $\delta>0$.  On the other hand, as a consequence of our current assumption that $\frac32B\cap \db{\partial\Omega}\neq \emptyset $ and of \eqref{eq:fat}, the function $\bfu$ vanishes on a subset of $2B$ whose measure exceeds $c|2B|$ for some positive constant $c$. On choosing 
$\delta$ small enough,   exploiting the fact that   $\eta^{(\alpha-1) p'}\leq \eta^\alpha$,  and making use of a Poincar\'e--Sobolev inequality on balls for functions enjoying this property, one can deduce from inequalities \eqref{Nov6}--\eqref{Nov10} that
  \begin{align*}
      \dashint_B \abs{\nabla \bfu}^p\dx 
& \leq
c\,\dashint_{2B} \Bigabs{\frac{\bfu  }{R}}^p\dx+ c \dashint_{2B} \abs{\bfF -\bfF_0}^{p'}\dx
\\
  &\quad    \leq c'\,
     \bigg(\dashint_{2B} \abs{\nabla \bfu}^{p\theta}\dx  \bigg)^{ \frac 1\theta }
   + c \dashint_{2B} \abs{\bfF -\bfF_0}^{p'}\dx
    \end{align*}
for some constants $c$ and $c'$.
Inequality \eqref{Nov4} is thus established. This inequality, via a version of 
 Gehring's lemma  as in  \cite{Iwa98}, implies that there exist an exponent  $q_0>1$ and a constant $c$ such that
\begin{align}
\label{gehring1}
      &\dashint_B \abs{\nabla \bfu}^{pq}\dx \leq c\,
      \bigg(\dashint_{2B} \abs{\nabla \bfu}^{p}\dx\bigg)^q
      + c \dashint_{2B}\abs{\bfF -\bfF_0}^{p'q}
      \dx.
\end{align}
 for every $q\in [1,q_0]$. {Inequalities \eqref{Nov2}  and \eqref{Nov3} follow from \eqref{gehring1}, via \cite[Lemma 3.3]{DieKapSch11}.}
 %the reverse Jensen inequality proved in \todo{What is this? } {\color{red} \cite[Remark 6.3]{Giu03}}.
\end{proof}

\iffalse

\seb{In this paper we use reflection methods which are possible in our setting of Dirichlet boundary values. This was already noticed by Uhlenbeck~\cite{Uhl77} and used by Chen and DiBenedetto to proof H\"older regularity in case of a smooth boundary~\cite{ChenDiBe}. However, the flattening approach both of \cite{KinZho01} (for equations) as well as~\cite{ChenDiBe} (for systems) does not apply in the case of weaker regularity assumptions on the boundary.}   %\todo{What about the method of Chen- Di Benedetto?} 
We have thus to resort to a new approach.

\fi
\subsection{Change of coordinates}\label{sec:boundary}

\iffalse

In the following we show, how to flatten the boundary and reflect the solution. Higher integrability up to the boundary was shown for equations by Kinunnen and Zhou~\cite{KinZho01} for $C^{1,\gamma}$-boundary. They used a boundary decay by Lieberman~\cite{Lie88} for p-harmonic solutions on half balls. This does not exist in the case of systems, so we will proceed differently and use a reflection argument. It is a modification of the calculation done in~\cite{BulDieSch15}.

\fi
Since the system  in  \eqref{dirloc} and the estimate to be derived are invariant under translations and rotations, we may assume, without loss of generality, that  $0 \in \partial \Omega$, that $B_R$ is centered at $0$, and that the outer normal to $\Omega$ at $0$ agrees with the opposite of the $n$-th unit vector of the canonical basis in $\mathbb R^n$.
% vector $(0,\dots ,0,-1)^t$. 
%\todo{I would replace balls by cylinders, \seb{What do you mean??}}.
\\ Assume, for the time being, that $\Omega$ is just a bounded Lipschitz domain, namely that $\partial \Omega \in C^{0,1}$. 
Then, there exists $R>0$, depending on the Lipschitz constant of $\partial \Omega$,  and  a map $\psi: \setR^{n-1}\to \setR$ such that $$\partial\Omega\cap B_R(0)=\set{(x',\psi(x')): \,(x',0)\in B_R(0)}$$  and $$\Omega \cap B_R(0) =\set{(x',x_n)\in B_R(0)\,:\,x_n>\psi(x')}.$$ 
Also, we define $\Psi: \overline{\Omega}\cap B_R(0)\to  B_R^+(0)$ as 
\begin{align}\label{PSI}
\Psi(x',x_n)=(x',x_n-\psi(x')) \quad \hbox{for $(x',x_n) \in \overline{\Omega}\cap B_R(0)$.}
\end{align}
 %where we abbreviated $(x_1,\ldots, x_n):= (x',x_n)$.
Observe that $\Psi(\partial\Omega\cap B_R(0))\subset \set{(x',x_n):x_n=0}$ and $\Psi(0)=0$. Moreover,  the function $\Psi : \overline \Omega\cap B_R(0) \to \Psi(\overline \Omega\cap B_R(0))$ is invertible, with a Lipschitz continuous inverse
%$R$ can be chosen sufficiently small for $\Psi$ to be 
$\Psi^{-1} :   \Psi(\overline \Omega\cap B_R(0)) \to \overline{\Omega}\cap B_R(0)$. Since, at this stage, we are merely assuming that $\partial \Omega \in C^{0,1}$, no additional regularity on $\Psi$ is available yet.   %Observe, that we do not impose any smallness on the oscillation of the gradient yet. 
%In the following we will use the convention (as is standard in the notion of elliptic PDE) that the gardient of a function is a column vector; $\nabla g=(\partial_1 g,...,\partial_n g)$.
Define $\bfJ : \overline{\Omega}\cap B_R(0)\to \setR ^{n\times n}$ as
\begin{equation}\label{J} \bfJ(x) = \nabla \Psi (x) \quad \hbox{for $x \in \overline{\Omega}\cap B_R(0)$.}
\end{equation}
%
%
%
%
% We fix 
% $\bfJ_{ij}:=\partial_i\Psi^j$, hence
Thus,
\begin{align}
 \bfJ (x', x_n)=
  \begin{pmatrix}
\bfI&0
\\
-\nabla \psi (x')& 1
\end{pmatrix}
=
 \begin{pmatrix}
1&0&\ldots&0
\\
0&\ddots&&\vdots
\\
 \vdots&&1&0
%\\
% \vdots&\vdots&&1&0
 \\
 - \psi_{x_1}(x')%&-\partial_2\psi
 &\ldots&-\psi_{x_{n-1}}(x')&1
\end{pmatrix} 
 \end{align} 
for $(x', x_n) \in \overline{\Omega}\cap B_R(0)$.
Moreover, with some abuse of notation, we define $\bfJ^{-1} :  \Psi(\overline \Omega\cap B_R(0))  \to \setR ^{n\times n}$  as 
$\bfJ^{-1}(y) =  \nabla \Psi^{-1} (y)$ for $y \in  \Psi(\overline \Omega\cap B_R(0)) $. Hence,
$$\bfJ^{-1}(y) =  (\nabla \Psi)^{-1} (\Psi ^{-1}(y)) \quad \hbox{for $y \in  \Psi(\overline \Omega\cap B_R(0)) $.}$$
Therefore, $\bfJ^{-1}(y) \bfJ (\psi^{-1}(y)) =\bfI$, the identity matrix, for $y \in  \Psi(\overline \Omega\cap B_R(0)) $.
 Clearly $\det \bfJ (x)=1$ for $x \in \overline{\Omega}\cap B_R(0)$ and  $\det\bfJ^{-1} (y)=1$ for $y \in  \Psi(\overline \Omega\cap B_R(0)) $. Hence $\abs{\Psi(E)}=\abs{E}$ for every measurable set $E \subset \overline{\Omega}\cap B_R(0)$, and $\abs{\Psi^{-1}(E)}=\abs{E}$  or every measurable set $E \subset   \Psi(\overline \Omega\cap B_R(0)) $.
 \\ Owing to the Lipschitz continuity of $\Psi$ and $\Psi^{-1}$,  there exist  constants 
\begin{equation}\label{lL}
\lambda \leq 1 \leq \Lambda
\end{equation}
 such that 
 \begin{align}\label{inclusion1}
  B_{\lambda r}^+(0)\subset \Psi(\Omega\cap B_r(0)) \subset B_{\Lambda r}^+(0)
\end{align}
if$0<r \leq R$,
and
 \begin{align}\label{inclusion2}
 &  \Omega  \cap B_{\frac r\Lambda }(0)\subset \Psi^{-1}(B_r^+(0)) \subset  \Omega \cap B_{\frac{r}{\lambda}}(0)
 \end{align}
if   $B_r^+(0) \subset \Psi(\overline \Omega\cap B_R(0))$. Note that the constant $C$ appearing in \eqref{eq:fat} only depends on a lower estimate for $R$ and $\lambda$ and on an upper estimate for $\Lambda$.
%\db{$\Psi$ is sometimes bold, sometimes not. Unify!}
 \\ %We define $y=\Psi(x)$ and for a
Next,
 given  a function $\bff$ on $\overline{\Omega}\cap B_R(0)$,  we define the function $\widetilde{\bff}$ on $ \Psi(\overline \Omega\cap B_R(0))$ as
\begin{equation}\label{ftilde}\widetilde{\bff}(y)=\bff (\Psi^{-1}(y)) \quad \hbox{for $y \in  \Psi(\overline \Omega\cap B_R(0))$.}
\end{equation}
Hence, if $\bff$ is differentiable, then
%\todo[inline]{$\bfJ^{-1}$ is already defined with $\Psi^{-1}y$}
$$
\nabla_y \widetilde{\bff}(y)=\nabla_x\bff(\Psi^{-1}(y)){\bfJ^{-1}(y) }\quad \hbox{for $y \in  \Psi(\overline \Omega\cap B_R(0))$,}
$$
%$$
%\nabla_y \widetilde{\bff}(y)=\nabla_x\bff(\Psi^{-1}(y))\bfJ^{-1}({\color{purple}\Psi^{-1}y}) \quad \hbox{for $y \in B_R^+(0)$,}
%$$
and
\[
 \nabla_x\bff(x)=\nabla_y\widetilde{\bff}(\Psi(x))\bfJ(x) \quad \hbox{for $x \in \overline{\Omega}\cap B_R(0)$,} %\text{ and }\divergence_x \bfg(x)=\divergence_y(\bfJ^t\circ\Psi^{-1}\widetilde{\bfg})(y),
\]
where $\nabla _x$ and $\nabla _y$ denote gradient with respect to the variables $x$ and $y$, respectively.
 %In particular
%where as before
%\begin{align*}
%\bfJ_N:=
%\begin{pmatrix}
%\bfJ&0&\ldots&0
%\\
%0&\bfJ&0&\vdots
%\\
% \vdots&0&\ddots&0 \\ 0&\ldots&0&\bfJ
%\end{pmatrix}
%%\circ\Psi^{-1}(y).
%\end{align*}
By the boundedness of $\bfJ$, we have that $$\abs{\nabla_x \bff(x)}\approx\abs{\nabla_y \widetilde{\bff}(y)} \quad \hbox{if $y=\Psi(x)$,}$$
up to multiplicative constants  depending only 
%on $\abs{\bfJ},\abs{\bfJ^{-1}}$, i.e. 
the Lipschitz constants of $\Psi$ and $\Psi^{-1}$. 
\\ Our aim is now to show that, if $\bfu$ is a solution to problem \eqref{dirloc}, then the function
$\widetilde{\bfu}$,  associated with $\bfu$ as in \eqref{ftilde}, solves a similar problem, involving  an elliptic system with variable coefficients.  To this purpose,    with define, for each $y \in  \Psi(\overline \Omega\cap B_R(0))$,  the function $\bfA_{\widetilde \bfJ}: \mathbb R^{N\times n} \to \mathbb R^{N\times n}$ as in  \eqref{AT}, with $\bfT^{-1}$ replaced by $\widetilde \bfJ(y)$. 
%Namely,
%$$
%\bfA_{\widetilde \bfJ}(\bfQ)=\bfA\big(\bfQ\,\widetilde \bfJ\big)\widetilde \bfJ^t \quad \hbox{for $\bfQ \in  \mathbb R^{N\times n} $.}
%$$
%
%
%\[
%\bfA_{\widetilde \bfJ}(\nabla_y \widetilde{\bfu})=\bfA\big(\nabla_y \widetilde{\bfu}\,\widetilde \bfJ\big)\widetilde \bfJ^t,
%\]
Thereby, 
$$
\bfA_{\widetilde\bfJ}(\nabla_y \widetilde{\bfu}(y)){=\bfA\big(\nabla_y \widetilde{\bfu}(y)\widetilde\bfJ(y)\big)\widetilde\bfJ^t(y)}=\bfA\big(\nabla_x {\bfu}(\Psi^{-1}(y))\big)\bfJ^{{t}}(\Psi^{-1}(y)) \quad \hbox{for $y \in  \Psi(\overline \Omega\cap B_R(0))$.}
$$
Since $\det{\bfJ}^{-1}=1$,   by \eqref{eq:matrix1} one has that
%, \color{yellow} similar to \eqref{constbis}, by transposition \todo{?????} \color{black} that
\begin{align*}
\int_{{\Omega\cap B_R(0)}} \bfA(\nabla_x \bfu)\cdot \nabla_x\bfphi\dx
&=\int_{\Psi (\Omega\cap B_R(0))} \bfA(\nabla_x \bfu (\Psi^{-1}(y)))\cdot \nabla_x\bfphi(\Psi^{-1}(y)))\det \bfJ^{-1}(y)\dy
\\
&=\int_{\Psi (\Omega\cap B_R(0))} \bfA\big(\nabla_y \widetilde{\bfu} (y)\bfJ( \Psi^{-1}(y))\big)\cdot \nabla_y\widetilde{\bfphi}(y)\bfJ(\Psi^{-1}(y))\dy
\\
&=\int_{\Psi (\Omega\cap B_R(0))} \bfA\big(\nabla_y \widetilde{\bfu} (y)\bfJ( \Psi^{-1}(y))\big)\bfJ^t(\Psi^{-1}(y))\cdot \nabla_y\widetilde{\bfphi}(y)\dy
\\
&=\int_{\Psi (\Omega\cap B_R(0))} \bfA_{\widetilde \bfJ}(\nabla_y \widetilde{\bfu})\cdot \nabla_y\widetilde{\bfphi}(y)\dy
\end{align*}
 for every function $\bfphi\in C^\infty_0({\Omega}\cap B_R(0))$.
A similar chain strarting from the integral 
$$\int_{{\Omega}\cap B_R(0)} (\bfF -\bfF_0)  \cdot \nabla_x\bfphi\dx$$
for an arbitrary matrix $\bfF_0\in \setR^{N\times n}$, and the use of equation \eqref{weaksolBR}
imply that
\begin{align}\label{nov1}
\int_{\Psi (\Omega\cap B_R(0))} \bfA_{\widetilde \bfJ}(\nabla_y\widetilde{\bfu}) \nabla_y\widetilde\bfphi\dy
&=\int_{\Psi (\Omega\cap B_R(0))}(\widetilde{\bfF}-\bfF_0)\widetilde \bfJ^t\cdot \nabla_y\widetilde\bfphi\dy
\end{align}
Equation \eqref{nov1} tells us that the function $\widetilde \bfu$ solves the following
% \begin{align*}
% -\divergence_x (\bfA(\nabla_x\bfu))
% &=-\divergence_y \big(\widetilde{\bfJ}^t\Big( \abs{\widetilde{\bfJ}\nabla_y\tilde{\bfu}}^{p-2}\widetilde{\bfJ}\nabla_y \widetilde{\bfu}\big)\\
% &= -\divergence_y (\widetilde{\bfJ}^t \bfA(\widetilde{\bfJ}\nabla_y\tilde{\bfu}))
% =:-\divergence (\bfA_{\widetilde{\bfJ}}(\nabla \widetilde{\bfu})).
% \end{align*}
% Therefore we gain the following 
problem: %system in $B_{\lambda R}^+(0)\subset \bfPsi(B_R(0)\cap \Omega) $
\begin{align}
\label{eq:boundary}
 \begin{cases}
 -\divergence_y(\bfA_{\widetilde\bfJ}(\nabla_y\widetilde{\bfu}))
= -\divergence_y{\widehat{\bfF}}& \quad \text{ in } B_{\lambda R}^+(0)\\
\widetilde{\bfu}(y',0)=0 & \quad \text{ on }\set{y_n=0} \cap B_{\lambda R}(0),
 \end{cases}
\end{align}
where  we have set $\widehat{\bfF}=(\widetilde\bfF-\bfF_0)\widetilde \bfJ^t$ and exploited \eqref{inclusion1}.
%\db{$\mathcal F$ looks a bit like Fourier transform to me}
Let us introduce the matrix $\bfJ_s \in \setR^{n \times n}$ defined as 
\begin{equation}\label{Js}
\bfJ_s =\mean{\bfJ}_{ \Omega \cap B_s(0)} \qquad \hbox{for $s\in (0,\lambda R)$.} \color{black}
\end{equation}
 Our purpose is to apply inequality \eqref{june43-hs} to  the solution $\widetilde{\bfu}$ to system \eqref{eq:boundary}, that can be rewritten as 
\begin{align}
\label{eq:boundaryfix}
 \begin{cases}
 -\divergence_y(\bfA_{\bfJ_s}(\nabla_y {\widetilde\bfu}))
= -\divergence_y\big(\bfA_{\bfJ_s}(\nabla_y\widetilde{\bfu})-\bfA_{\widetilde\bfJ}(\nabla_y\widetilde{\bfu})+\widehat{\bfF}\big)& \quad \text{ in } B_{\lambda R}^+(0)\\
\widetilde{\bfu}(y',0)=0 & \quad \text{ on } \set{y_n=0} \cap B_{\lambda R}(0).
 \end{cases}
\end{align}
Choose $\Lambda$ so large that,  in addition to \eqref{inclusion1} and \eqref{inclusion2}, one has that $\bfJ_s B_s(0)\subset B_{\lambda R}(0)$ for  $s\in (0,\lambda R)$.
Following the approach of the previous section, we define $\overline{\bfu} :    H_{\bfJ_s} \cap B_{\frac{\lambda}{\Lambda}  R}(0)  \to \setR^N$ as
\begin{equation}
\label{ubar}
\overline{\bfu}(z)=\widetilde{\bfu}(\bfJ_s z) \qquad \hbox{for $z \in  H_{\bfJ_s} \cap B_{\frac{\lambda}{\Lambda}  R}(0)$.}
\end{equation}
 Hence, $$\nabla_z \overline{\bfu}(z)=\nabla_y\widetilde{\bfu}(\bfJ_s z)\bfJ_{s} \quad \hbox{ for $z \in  H_{\bfJ_s} \cap B_{\frac{\lambda}{\Lambda}  R}(0)$.}$$
Notice that, by the special form of $\bfJ$, and hence of $\bfJ_s$,  given $z \in \setR^n$, we have that
\[
(\bfJ_s z)_n\geq 0\quad \text{ if and only if } \quad z_n\geq \mean{\nabla\psi}_{{\Omega \cap B_s}}\cdot z' .
\]
{ Also, define accordingly 
$\overline{\bfF} :    H_{\bfJ_s} \cap B_{\frac{\lambda}{\Lambda}  R}(0)  \to \setR^{N}$ as 
\begin{equation}
\label{Fbar}
\overline{\bfF}(z)=\widetilde{\bfF}(\bfJ_s z) \qquad \hbox{for $z \in  H_{\bfJ_s} \cap B_{\frac{\lambda}{\Lambda}  R}(0)$.}
\end{equation}
and
$\overline{\bfJ} :    H_{\bfJ_s} \cap B_{\frac{\lambda}{\Lambda}  R}(0)  \to \setR^{N\times n}$ as 
\begin{equation}
\label{Jbar}
\overline{\bfJ}(z)=\widetilde{\bfJ}(\bfJ_s z) \qquad \hbox{for $z \in  H_{\bfJ_s} \cap B_{\frac{\lambda}{\Lambda}  R}(0)$.}
\end{equation}}
An analogous argument as in the proof of equation \eqref{eq:pertbar} implies that $\overline \bfu$ is a solution to the problem  %that \db{why is $\overline{\bfu}$ defined in $B$ not only in $B^+$?}
%\todo[inline]{I have changed $\widetilde \bfJ$ to $\overline \bfJ$ in equation \eqref{eq:fixscal1} and below. Please check}
\begin{align}
\label{eq:fixscal1}
 \begin{cases}
 -\divergence_z(\bfA(\nabla\overline{\bfu}))
= -\divergence_z\big(\bfA(\nabla\overline{\bfu}) - \bfA_{\bfJ_s^{-1}\overline\bfJ}(\nabla\overline{\bfu}) \color{black} +\underline{\bfF}\big)& \quad \text{ in } H_{\bfJ_s} \cap B_{\frac{\lambda}{\Lambda}  R}(0)\\
\overline{\bfu}=0 & \quad \text{ on } \set{(\bfJ_s z)_n = 0}\cap B_{\frac{\lambda}{\Lambda} R}(0),
 \end{cases}
\end{align}
where we have set  
$$\underline {\bfF}(z)=(\overline{\bfF}(z)-\bfF_0)\overline \bfJ^t(z)(\bfJ_s^{-1})^t \quad \hbox{for $z \in   H_{\bfJ_s} \cap B_{\frac{\lambda}{\Lambda}  R}(0)$.}$$
%
%
%$$\overline{\bfF}(z)=\widehat{\bfF}(\bfJ_s z)\bfJ_s^{-t} \quad \hbox{for $z \in   H_{\bfJ_s} \cap B_{\frac{\lambda}{\Lambda}  R}(0)$.}$$
Thus,
\begin{align}
\label{eq:fbar}
\underline{\bfF}(z)
%= (\widetilde{\bfF}(\bfJ_s z)-\bfF_0)\widetilde \bfJ^t( \bfJ_s z)(\bfJ_s^{-1})^t 
 = ({\bfF}(\psi ^{-1}(\bfJ_s z))-\bfF_0)\bfJ^t(\psi ^{-1}(\bfJ_s z))(\bfJ_s^{-1})^t 
\end{align}
for $z \in  H_{\bfJ_s} \cap B_{\frac{\lambda}{\Lambda}  R}(0)$.

\subsection{Decay near the  boundary}
We are now in a postion to state and prove a crucial decay estimate at the  boundary for the gradient of the solution $\bfu$   to the Dirichlet problem \eqref{eq:sysA}. Given    $R>0$ and $x \in \partial \Omega$,
define $\bfA_s \in \setR^{N\times n}$, for $s\in (0,R]$, as
%\color{blue}
%\begin{equation}
%\bfA_r  =\mean{\bfA(\nabla\bfu)}_{B_r\cap\Omega}.
%\end{equation}  
%\todo{ It seems to me that we need this definition in order to  apply the best constant property in the proof,  instead of 
%\seb{NO: We definitely need this definition:}}
\begin{align}
\label{a9}
\begin{cases}
 (\bfA_s)_{ki}=0 & \quad \text{ for } i\in\set{1,...,n-1},k\in \set{1,...,N}  \\
 (\bfA_s)_{kn}=\mean{\bfA_{kn}(\nabla \bfu)}_{\Omega \cap B_s}& \quad \text{ for }k\in \set{1,...,N}.
 \end{cases}
\end{align}
\color{black}

\begin{proposition} \label{pro:boundarydecay}
Let $\Omega$ be a bounded open set in $\setR^n$ and let $x\in \partial \Omega$. Assume that there exists $R>0$ and  local coordinates
  in $\Omega \cap B_R(x)$,  as in Subsection \ref{sec:boundary}, such that $\psi  \in W^1\mathcal L^{\sigma (\cdot)} \cap C^{0,1}$ for some parameter function $\sigma$. Assume that  $\bfF\in L^{p'q}(\Omega)$ for some $q>1$, and let 
$\bfu$ be the weak solution to the Dirichlet problem \eqref{eq:sysA}. 
 Then there exist constants $c>0$ and $\theta\in (0,1)$, depending on
 $n,p,N, \omega, q, \Omega$ such that
\iffalse $n,p,N,\beta,c_\omega,\lambda,\Lambda,q$, such that \fi
%
%
%Let $\bfu$ be the weak solution to the Dirichlet problem \eqref{eq:sysA}. 
%% \color{blue} Assume that $\partial \Omega \in W^1\mathcal L_{\sigma (\cdot)} \cap C^{0,1}$.
%% with \eqref{eq:dirichlet}. 
%Assume that $x\in \partial \Omega$ and let $R>0$ be such that there exist local coordinates
%  in $\Omega \cap B_R(x)$,  with $\psi  \in W^1\mathcal L^{\sigma (\cdot)} \cap C^{0,1}$ for some parameter function $\sigma$,  in the sense of Subsetcion \ref{sec:nonflat}. \color{black}
%Let $\omega$ be a parameter function satisfying condition \eqref{eq:omega condition}, and let $q>1$.
%%
%% {$q\in (1,q_0]$, where $q_0$} is the exponent appearing in the statement of 
%%%or $q_0=q_0(\lambda,\Lambda,p,n,N)$ fixed via 
%%Proposition~\ref{cor:VPL} \todo[inline]{modify here and in the proof}.
% If $\bfF\in L^{p'q}(\Omega)$,
% then there exist constants $c>0$ and $\theta\in (0,1)$, depending on $n,p,N,\beta,c_\omega,\lambda,\Lambda,q$, such that 
\begin{align}\label{Nov11}
&\frac{1}{\omega(\theta s)}\bigg(\dashint_{\Omega \cap  B_{\theta s}(x)}\abs{\bfA(\nabla\bfu)-\bfA_{\theta s} }^{\min{\set{2,p'}}}
\dy\bigg)^{\frac{1}{\min{\set{2,p'}}}}
\\ \nonumber
&\quad \leq
\frac{c\sigma(s)}{\omega(s)}\dashint_{\Omega \cap B_{s}(x)}\abs{\bfA(\nabla{\bfu})}\dy
+\frac{c\sigma(s)+c}{\omega(s)}\bigg(\dashint_{\Omega \cap B_{s}(x)}\abs{\bfF-\bfF_0}^{p'q}\dy \bigg)^\frac1{p'q} 
\\ \nonumber
&\qquad + \frac{1}{2\omega(s)}\bigg(\dashint_{\Omega \cap B_{s}(x)}\abs{\bfA(\nabla{\bfu})-\bfA_{s} }^{\min{\set{2,p'}}}\dy\bigg)^\frac{1}{\min{\set{2,p'}}}
\end{align}
for every matrix $\bfF _0 \in \setR^{N\times n}$ and every $s\in [0,R]$.
%Here $\sigma$ is defined via the Jacobi matrix $\bfJ=\nabla \Psi$ as
%\begin{align}
%\label{eq:Jw}
%\sigma(\rho):=\bigg(\dashint_{ B_\rho^+}\abs{\bfJ-\mean{\bfJ}_{B_\rho^+}}^{p q'}\dx\bigg)^\frac{1}{q'p'}.
%\end{align}
\end{proposition}

The following algebraic inequality will be needed in the proof of Proposition \ref{pro:boundarydecay}.

\begin{lemma}
\label{lem:hammerT}
Assume that the matrices $\bfT _i \in \setR^{n\times n}$,  $i= 1,2$, are such that  
 $\lambda \abs{\bfQ}\leq \abs{\bfT_i \bfQ}\leq \Lambda \abs{\bfQ}$, 
for some $\Lambda >\lambda >0$,  and every $\bfQ\in \setR^{N\times n}$. %then \db{what is $\bfT_i$?} 
Let $\bfA_{\bfT_i}$,  $i= 1,2$, be the functions defined as in \eqref{AT}.
Then there exists a constant $c = c(p,n,N,\lambda , \Lambda)$ such that
\begin{align*}
 \abs{\bfA_{\bfT_1}(\bfQ)-\bfA_{\bfT_2}(\bfQ)}\leq c \abs{\bfQ}^{p-1} \abs{\bfT_1-\bfT_2} \quad \hbox{for $\bfQ\in \setR^{N\times n}$.}
\end{align*}
\end{lemma}
\begin{proof}
One has that
\begin{align}
      \label{eq:hammerd}
      \abs{\bfA(\bfP)-\bfA{(\bfQ)}}\approx \max\set{\abs{\bfP},\abs{\bfQ}}^{p-2}\abs{\bfP-\bfQ} \quad  \hbox{for $\bfP, \bfQ \in \setR^{N\times n}$,}
     \end{align}
up to multiplicative constants depending on $n,N,p$. Hence, there exist constants $c$ and $c'$ such that
\begin{align*}
 &\abs{\bfA_{\bfT_1}(\bfQ)-\bfA_{\bfT_2}(\bfQ)}
 \leq \abs{ \bfA(\bfQ \bfT_1)-\bfA(\bfQ\bfT_2)}\abs{\bfT_1}+ \abs{\bfT_1-\bfT_2}\abs{\bfA(\bfQ \bfT_2)}
 \\
 &\leq \Big(\abs{\bfT_1\bfQ}+\abs{\bfT_2\bfQ}\Big)^{p-2}\abs{\bfQ} \abs{\bfT_1-\bfT_2} +c\abs{\bfQ}^{p-1}\abs{\bfT_1-\bfT_2}
 \leq c'\abs{\bfQ}^{p-1}\abs{\bfT_1-\bfT_2}
 \end{align*}
for every $\bfQ\in \setR^{N\times n}$.
%=\Bigabs{\abs{\bfQ T_1}^{p-2}T_1^t\mathcal T_1 \bfQ
%-\abs{\mathcal T_2 \bfQ}^{p-2}\mathcal T_2^t\mathcal T_2\bfQ}\\
%&\quad\leq \Bigabs{({\mathcal T_1^t-\mathcal T_2^t})\abs{\mathcal T_1 \bfQ}^{p-2}\mathcal T_1\bfQ}
%+\abs{{\mathcal T_2}}\Bigabs{\abs{{\mathcal T_1} \bfQ}^{p-1}\frac{{\mathcal T_1}\bfQ}{\abs{{\mathcal T_1} \bfQ}}
%-\abs{{\mathcal T_2} \bfQ}^{p-1}\frac{{\mathcal T_2} \bfP}{\abs{{\mathcal T_2} \bfP}}}\\
%&\quad \leq c\abs{\bfT_1-\bfT_2}\abs{{\mathcal T_1} \bfQ}^{p-1}
%+c\abs{{\mathcal T_2}}\Big(\abs{\bfQ}+\abs{\mathcal T_1-\mathcal T_2}\abs{\bfQ}\Big)^{p-2}(\abs{{\mathcal T_1} -{\mathcal T_2}}\abs{\bfQ}).
%\end{align*}
% Since $\abs{\bfQ}+\abs{\mathcal T_1-\mathcal T_2}\abs{\bfQ}\approx \abs{\bfQ}$, by the assumptions, we find by \eqref{eq:hammera}, that
%\[
% \abs{\bfA_{\bfT_1}(\bfQ)-\bfA_{\bfT_2}(\bfQ)}\leq  c\abs{\bfT_1-\bfT_2}\abs{\bfQ}^{p-2}.
%\]
\end{proof} 

\begin{proof}[Proof of Proposition \ref{pro:boundarydecay}] Without loss of generality, we may assume that $x=0$, and, 
for simplicity, we denote the ball $B_r(0)$  by $B_r$ throughout this proof. Assume, for the time being, that $q\in (1,q_0]$, where $q_0$ is the exponent appearing in the statement of 
Proposition~\ref{cor:VPL}.
To begin with, note that, since    $\psi  \in W^1\mathcal L^{\sigma (\cdot)}$,   there exists a constant $C$ such that
\begin{equation}\label{Lsigma}
\sup_{r\in (0,R)} \frac 1{\sigma (r)} \bigg(\dashint_{\Omega \cap B_r\color{black}}\abs{\bfJ-\mean{\bfJ}_{ \Omega \cap  B_r \color{black}}}^{p' q'}\dx\bigg)^\frac{1}{q'p'} \leq C,
\end{equation}
where $\bfJ$ is defined as in \eqref{J}. This is a consequence of the definition of Campanato seminorms and of the observation following their definition in \eqref{campnorm}. 
\\
\color{black}
We want to apply Proposition \ref{cor:decay-hs},  with
% $\bfB =\bfI$ and 
$\bfT^{-1}= \bfJ_s$,
 %\eqref{june43-hs}
 to the solution $\overline{\bfu}$, given by \eqref{ubar}, to problem    \eqref{eq:fixscal1}. To this purpose,  define $D_s$ as in \eqref{Ds}, with this choice of %$\bfB$ and
 $\bfT^{-1}$. 
%
%\color{blue} $\bfB \bfT^{-1} = \bfJ_s$\color{black}, 
Namely,
$$D_s = H_{\bfJ_s} \cap B_s.$$ 
Also, we set
 $$\theta D_s=H_{\bfJ_s}\cap B_{\theta s}$$
for $\theta \in (0,1]$.
 Next, define $\overline{\bfA}_\tau \in \setR^{N\times n}$ for $\tau\in (0,s]$ by
\begin{align}
\begin{cases}
  (\overline\bfA_\tau)_{ki}=0 & \quad \text{ for }i\in\set{1,...,n-1},k\in \set{1,...,N}
\\   (\overline\bfA_\tau)_{kn}=\mean{\bfA_{kn}(\nabla \overline\bfu)}_{\frac{\tau}{s}D_s}& \quad \text{ for }k\in \set{1,...,N}.
 \end{cases}
\end{align}
An application of Proposition~\ref{cor:decay-hs} then
%, with \color{blue} $\bfB \bfT^{-1}=\bfJ_s$\color{black}, 
%\db{what is $\bfB$?}.
%
%\color{red} WHY IS THIS STEP NEEDED? DOESN'T Corollary \ref{cor:decay-hs} directly imply inequality \eqref{june43-hs1}?
tells us that  there exist $\theta\in (0,1)$ and  $c>0$ such that
\begin{align}\label{june43-hs1a}
& \frac{1}{\omega(\tau \theta s)}  \bigg(\dashint_{\theta \tau  D_s} \abs{\bfA(\nabla \overline\bfu)-\overline\bfA_{\theta \tau  s} }^{\min\set{2,p'}}\dz\bigg)^\frac{1}{\min{\set{2,p'}}}
 \\  \nonumber &  \leq   \frac{1}{2\omega(\tau s)}
 \bigg(\dashint_{\tau  D_s} \abs{\bfA(\nabla \overline\bfu)-\overline\bfA_{\tau  s} }^{\min\set{2,p'}}\dz\bigg)^\frac{1}{\min{\set{2,p'}}}
 \\ \nonumber & \quad
+   \frac{c}{\omega( \tau  s)} \bigg(\dashint_{\tau  D_s}\abs{\underline \bfF-\bfF_0
}^{p'}\dz\bigg)^\frac1{p'}
% \\ \nonumber
%&\quad 
+  \frac{c}{\omega(\tau  s)} \bigg(\dashint_{\tau  D_s}\abs{\bfA(\nabla\overline{\bfu})-{\bfA_{\bfJ_s^{-1}\overline{\bfJ}}(\nabla\overline{\bfu})}}^{p'}\dz\bigg)^\frac1{p'}
\end{align}
  for every $\bfF_0 \in \setR^{N\times n}$ and every $\tau \in  (0,1)$.
%\color{blue} \todo{This step is not clear to me} 
Given any $\delta\in (0,1)$, let $k\in \setN$ be such that $2^{-k}\leq \delta$. Iterating inequality \eqref{june43-hs1a}, with { $\tau= \theta^j$, $j=\set{0,1,...,k}$}, tells us that there exist $\theta \in (0,1)$ and $c>0$, depending on $\delta$, such that %\db{what happened to $\tau$?}
\begin{align}\label{june43-hs1}
%{\color{orange} (I):=}
\frac{1}{\omega(\theta s)} & \bigg(\dashint_{\theta D_s} \abs{\bfA(\nabla \overline\bfu)-\overline\bfA_{\theta s} }^{\min\set{2,p'}}\dz \bigg)^\frac{1}{\min{\set{2,p'}}}
 \\ \nonumber &  \leq   \frac{\delta}{\omega(s)}
 \bigg(\dashint_{D_s} \abs{\bfA(\nabla \overline\bfu)-\overline\bfA_{ s} }^{\min\set{2,p'}}\dz\bigg)^\frac{1}{\min{\set{2,p'}}}
 \\ \nonumber & \quad
+   \frac{c}{\omega(s)} \bigg(\dashint_{D_s}\abs{{\underline\bfF}}^{p'}\dz\bigg)^\frac1{p'}
% \\ \nonumber
%&\quad 
+  \frac{c}{\omega(s)} \bigg(\dashint_{D_s}\abs{\bfA(\nabla\overline{\bfu})-{ \bfA_{\bfJ_s^{-1}\overline \bfJ}(\nabla\overline{\bfu}})%\db{-\bfA_0}
}^{p'}\dz\bigg)^\frac1{p'}
%\\  \nonumber
%&{\color{orange} =(II)+(III)+(IV).}
\end{align}
 for every $\bfF_0 \in \setR^{N\times n}$.
\color{black}
Fix $r$ such that  ${\frac{\Lambda}{\lambda} r}=\theta s$. By property \eqref{meanq}
applied to {the $n$-th component},  a change of variables, and the fact that $\det(\nabla \Psi^{-1} \bfJ_s)=\det(\nabla \Psi^{-1})\det(\bfJ_s)=1 \cdot 1 = 1$, one has that
\begin{align}\label{Nov13}
%{\color{orange}(i)}&:=
%\\ \nonumber
&\frac{1}{\omega(r)}\bigg(\dashint_{\Omega \cap B_r}\abs{\bfA(\nabla\bfu)-
\bfA_r}^{\min{\set{2,p'}}}
\dx\bigg)^{\frac{1}{\min{\set{2,p'}}}}
\\ \nonumber
&\leq \frac{2}{\omega(r)}\bigg(\dashint_{\Omega\cap B_r}\abs{\bfA(\nabla \bfu)-\overline{\bfA}_{\theta s}}^{\min{\set{2,p'}}}\dx\bigg)^\frac{1}{\min{\set{2,p'}}}
\\ \nonumber &=
\frac{2}{\omega(r)}\bigg(\dashint_{\bfJ_s^{-1}\Psi(\Omega \cap B_r)}\abs{\bfA(\nabla\bfu(\Psi^{-1}(\bfJ_s z)))-\overline{\bfA}_{\theta s}}^{\min{\set{2,p'}}}\abs{\det \nabla \Psi^{-1}\bfJ_s}\dz\bigg)^\frac{1}{\min{\set{2,p'}}}
\\ \nonumber
&\leq \frac{2}{\omega(r)}\bigg(\dashint_{\theta D_{s}}\abs{\bfA(\nabla\overline{\bfu})-\overline{\bfA}_{\theta s} }^{\min{\set{2,p'}}}\dz\bigg)^\frac{1}{\min{\set{2,p'}}}
\\ \nonumber
&\quad +\frac{2}{\omega(r)}\bigg(\dashint_{\theta D_{s}}\abs{\bfA(\nabla\overline{\bfu})-\bfA(\nabla\bfu(\Psi^{-1}(\bfJ_s z)))}^{\min{\set{2,p'}}}\dz\bigg)^\frac{1}{\min{\set{2,p'}}}.
%\\ \nonumber
%&{\color{orange} =:
%c (I)+c(ii)}
\end{align}
Observe that the second inequality holds since, by the second inclusion in \eqref{inclusion1} and the first inequality in \eqref{lL},
$$\bfJ_s^{-1}\Psi(\Omega \cap B_r) \subset \bfJ_s^{-1} B_{\Lambda r}^+ =  H_{\bfJ _s} \cap B_{\Lambda r}\subset  H_{\bfJ _s} \cap B_{\frac \Lambda \lambda r} =  H_{\bfJ _s} \cap B_{\theta s} =\theta D_s.$$\color{black}
 Now, note  the identities
\begin{equation}\label{marzo1}
\nabla_z\overline{\bfu}(z)=\nabla_x \bfu (\Psi^{-1}(\bfJ_s z))\, \bfJ^{-1}(\bfJ_s z)\, \bfJ_s\,
\end{equation}
and
\begin{equation}\label{marzo2}
\bfJ^{-1}(\bfJ_s z) \bfJ(\Psi^{-1}(\bfJ_s z))  = \bfI
\end{equation}
for $z \in D_R$,
% \begin{align*}
%      %\label{eq:hammerd}
%      \abs{\bfA(\bfP)-\bfA{(\bfQ)}}\approx \max\set{\abs{\bfP},\abs{\bfQ}}^{p-2}\abs{\bfP-\bfQ} \quad  \hbox{for $\bfP, \bfQ \in \setR^{N\times n}$,}
%%\end{aligned}
%     \end{align*}
%to equation \eqref{eq:hammerd} 
and the inequality
\begin{equation}\label{marzo3}
\max\{|\bfP|, |\bfP \bfR|\}^{p-2}|\bfP| \leq \max\{1, |\bfR|^{{p-2}}\}|\bfP|^{p-1} \quad \hbox{for $\bfP \in \setR^{N\times n}$ and $\bfR \in \setR^{n\times n}$.}
\end{equation}
Owing to equations \eqref{marzo1}--\eqref{marzo3} and  \eqref{eq:hammerd},
the following chain holds:
\begin{align}\label{Nov14}
&\abs{\bfA(\nabla\overline{\bfu}(z))-\bfA(\nabla\bfu(\Psi^{-1}(\bfJ_s z)))}
\\ \nonumber
& =\abs{\bfA(\nabla_x \bfu (\Psi^{-1}(\bfJ_s z))\, \bfJ^{-1}(\bfJ_s z)\, \bfJ_s )-\bfA(\nabla\bfu(\Psi^{-1}(\bfJ_s z))\bfJ^{-1}(\bfJ_s z)\bfJ(\Psi^{-1}(\bfJ_s z))) } 
%& \approx \phi_{p,\abs{\nabla_x \bfu(\bfPsi^{-1}(\bfJ_s z))}}'\Big(\abs{\nabla_x \bfu(\bfPsi^{-1}(\bfJ_s z))(\bfJ^{-1}(\bfJ_s z)\bfJ_s-\bfI)}\Big)
%\\
%&\quad \leq \phi_{p,\abs{\nabla_x \bfu(\bfPsi^{-1}(\bfJ_s))}}'\Big(\abs{\bfJ(\bfPsi^{-1}(\bfJ_s z))-\bfJ_s}\abs{\bfJ^{-1}(\bfJ_s z)}\abs{\nabla_x \bfu(\Psi^{-1}(\bfJ_s z))}\Big)
%\\
\\ \nonumber
& \approx \max\bigset{\abs{\nabla_x \bfu(\Psi^{-1}(\bfJ_s z))},
\abs{\nabla \bfu (\Psi^{-1}(\bfJ_s z)) \bfJ^{-1}(\bfJ_sz)\bfJ_s}}^{p-2}\\ \nonumber
&\qquad \times 
\abs{\nabla_x \bfu (\Psi^{-1}(\bfJ_s z)\, \bfJ^{-1}(\bfJ_s z)\, \bfJ_s  - \nabla\bfu(\Psi^{-1}(\bfJ_s z)) \bfJ^{-1}(\bfJ_s z) \bfJ(\Psi^{-1}(\bfJ_s z))}
\\ \nonumber
& \leq \max\bigset{\abs{\nabla_x \bfu(\Psi^{-1}(\bfJ_s z))},
\abs{\nabla \bfu (\Psi^{-1}(\bfJ_s z)) \bfJ^{-1}(\bfJ_sz)\bfJ_s}}^{p-2}\\ \nonumber
&\qquad \times 
\abs{\bfJ(\psi^{-1}(\bfJ_sz))-\bfJ_s }\abs{\bfJ^{-1}(\bfJ _s z)}\abs{\nabla_x \bfu(\Psi^{-1}(\bfJ_s z))}
\\ \nonumber
&\quad \leq c\abs{\bfJ(\Psi^{-1}(\bfJ_s z))-\bfJ_s}\abs{\nabla_x \bfu(\Psi^{-1}(\bfJ_s z))}^{p-1} \quad \hbox{for $z \in D_R$,}
\end{align}
for some constant $c$.
%where the last inequality holds since
%$$\max\{|\bfP|, |\bfP \cdot \bfQ|\}^{p-2}|\bfP| \leq \max\{1, |\bfQ|^{p-1}\}|\bfP|^{p-1} \quad \hbox{for $\bfP, \bfQ \in \setR^{N\times n}$.}$$
%\color{black}
%
%{\color{red}
%
%\begin{align*}
%&\abs{\bfA(\nabla\overline{\bfu}(z))-\bfA(\nabla\bfu(\Psi^{-1}(\bfJ_s z)))}
%\approx \phi_{\abs{\nabla_x \bfu(\bfPsi^{-1}(\bfJ_s z))},p}'\Big(\abs{\nabla_x \bfu(\bfPsi(z)\bfJ_s)(\bfJ^{-1}(\bfJ_s z)\bfJ_s-\bfI)}\Big)
%\\
%&\quad \leq \phi_{\abs{\nabla_x \bfu\circ\bfPsi^{-1}(\bfJ_s)},p}'\Big(\abs{\bfJ(\db{\bfJ_s z})-\bfJ_s}\abs{\bfJ^{-1}(\bfJ_s z)}\abs{\nabla_x \bfu(\Psi^{-1}(\bfJ_s z))}\Big)
%\\
%&\quad \approx \Big(\max\bigset{\abs{\nabla_x \bfu(\Psi^{-1}(\bfJ_s z))},\abs{\bfJ(\Psi^{-1}(\bfJ_s z))-\bfJ_s}\abs{\bfJ^{-1}(\bfJ_s z)}\abs{\nabla_x \bfu(\Psi^{-1}(\bfJ_s z))}}\Big)^{p-2}\\
%&\qquad \times \Big(\abs{\bfJ_{\db{s}}(z)-\bfJ_s}\abs{\bfJ_{\db{s}}^{-1}(z)}\abs{\nabla_x \bfu(\Psi^{-1}(\bfJ_s z))}\Big)
%\\
%&\quad \leq c\abs{\bfJ(\Psi^{-1}(\bfJ_s z))-\bfJ_s}\abs{\nabla_x \bfu(\Psi^{-1}(\bfJ_s z))}^{p-1}
%\end{align*}
%}
In the last inequality we have also made use of the fact that $\bfJ_s$ and $\bfJ^{-1}$ are bounded.
%\seb{Next, we find by the above calculation and Proposition~\ref{cor:VPL} for any $F_0$, that}
Thus, one has that
\begin{align}
\label{eq:hoelder}
%(ii)&\leq \frac{c(\theta)}{\omega(\theta R)}\bigg(\dashint_{B_{\theta R}^+}\phi_{\abs{\nabla_x \bfu(\Psi^{-1}(\bfJ_s z))},p}'\Big(\abs{\bfJ-\bfJ_s}\abs{\bfJ^{-1}}\abs{\nabla_x \bfu(\bfPsi^{-1}(\bfJ_s z)}\Big)^{\min{\set{2,p'}}}\dz\bigg)^\frac{1}{\min{\set{2,p'}}}
%\\
%{\color{orange} (ii) = }
\frac{1}{\omega(r)}& \bigg(\dashint_{\theta D_{s}}\abs{\bfA(\nabla\overline{\bfu})-\bfA(\nabla\bfu(\Psi^{-1}(\bfJ_s z)))}^{\min{\set{2,p'}}}\dz\bigg)^\frac{1}{\min{\set{2,p'}}}
\\ &\leq \nonumber
\frac{{\color{red} c}}{\omega(\theta s)}\bigg(\dashint_{\theta D_{s}}\Big(\abs{\bfJ(\Psi^{-1}(\bfJ_s z))-\bfJ_s}\abs{\nabla \bfu(\Psi^{-1}(\bfJ_s z))}^{p-1}\Big)^{p'}\dz\bigg)^\frac{1}{p'}
\\  \nonumber
&\leq \frac{c'}{\omega(s)} \bigg(\dashint_{D_s}\abs{\bfJ(\Psi^{-1}(\bfJ_s z))-\bfJ_s}^{p'q'}\dz\bigg)^\frac1{p'q'}\bigg(\dashint_{D_s}\abs{\bfA(\nabla \bfu(\Psi^{-1}(\bfJ_s z)))}^{p'q}\dz\bigg)^\frac1{p'q} 
\\  \nonumber
&\leq \frac{c''\sigma(s)}{\omega(s)}\bigg(\dashint_{\Omega \cap B_{\frac{\Lambda}{\lambda} s}}\abs{\bfA(\nabla{\bfu})}^{p'q}\dx \bigg)^\frac1{p'q}
\\  \nonumber
&\leq \frac{c'''\sigma(s)}{\omega(s)}\dashint_{\Omega \cap B_{2\frac{\Lambda}{\lambda} s}}\abs{\bfA(\nabla{\bfu})}\dx
+\frac{c'''\sigma(s)}{\omega(s)}\bigg(\dashint_{\Omega \cap B_{2\frac{\Lambda}{\lambda} s}}\abs{\bfF-\bfF_0}^{p'q}\dx \bigg)^\frac1{p'q},
\end{align}
for some constants $c,c', c'', c'''$ and for every $\bfF_0 \in \setR^{N\times n}$. Notice that 
 first inequality in \eqref{eq:hoelder}  is due to \eqref{Nov14} and  \eqref{eq:omega condition} (and to H\"older's inequality if $p\in (1,2)$), the second one to H\"older's inequality and \eqref{eq:omega condition},  the third one  to  \eqref{Lsigma}, to a change of variable, to the  boundedness  of $\bfJ_s$, $\bfJ$ and $\bfJ^{-1}$  and to the fact that 
$$\Psi^{-1}(\bfJ_s D_s)= \Psi^{-1}(\bfJ_s (H_{\bfJ_s}\cap B_s))= \Psi^{-1}(B_s^+)  \subset  \Omega \cap B_{\frac s\lambda} \subset  \Omega \cap B_{\frac{\Lambda }\lambda s }\,,$$\color{black}
and the last one to Proposition~\ref{cor:VPL}. 
%In chain \eqref{eq:hoelder}, the constant $c$ changes from line to line, and also depends on $\theta$.
%
%Thanks to equations \eqref{Nov14} and  \eqref{eq:omega condition},   the  boundedness  of $\bfJ_s$, $\bfJ$ and $\bfJ^{-1}$, one obtains, via   a change of variables and Proposition~\ref{cor:VPL}, that
%
%Note that the first inequality makes use of \eqref{eq:omega condition} (and of H\"older's inequality if $p\in (1,2)$),  and the third inequality relies upon  \eqref{Lsigma}, and on the fact that \color{blue}
%$$\psi^{-1}(\bfJ_2 D_s))= \psi^{-1}(\bfJ_s (B_s \cap H_{\bfJ_s})= \psi^{-1}(B_s^+)  \subset B_{\frac s\lambda}\cap \Omega \subset B_{\frac{\Lambda }\lambda s }\cap \Omega\,.$$\color{black}
%%
%
% we also used \eqref{eq:hammera} in case $p\in (1,2)$ \db{ not in case $p>2$?} Jensen's inequality and \eqref{eq:omega condition}.
%The estimate on $(II)$ is similar to the estimate of $(ii)$; indeed imitating the above implies
Similarly, we have that
\begin{align} \label{Nov15}
%{\color{orange} (II)=}
& \frac{\delta}{\omega(s)}
 \bigg(\dashint_{D_s} \abs{\bfA(\nabla \overline\bfu)-\overline\bfA_{ s} }^{\min\set{2,p'}}\dz\bigg)^\frac{1}{\min{\set{2,p'}}}
\\ & \nonumber
\leq \frac{c\delta}{\omega(s)}\bigg(\dashint_{D_s}\abs{\bfA(\nabla\overline{\bfu})-\bfA_{\frac{\Lambda}{\lambda}{s}}}^{\min{\set{2,p'}}}\dz\bigg)^\frac{1}{\min{\set{2,p'}}}
\\ \nonumber
&\leq \frac{c\delta}{\omega(s)}\bigg(\dashint_{D_s}\abs{\bfA(\nabla\bfu(\Psi^{-1}(\bfJ_s z)))-\bfA_{\frac{\Lambda}{\lambda}{s}}}^{\min{\set{2,p'}}}\dz\bigg)^\frac{1}{\min{\set{2,p'}}}
\\  \nonumber
&\quad +\frac{c\delta}{\omega(s)}\bigg(\dashint_{D_s}\abs{\bfA(\nabla\overline{\bfu})-\bfA(\nabla\bfu(\Psi^{-1}(\bfJ_s z)))}^{\min{\set{2,p'}}}\dz\bigg)^\frac{1}{\min{\set{2,p'}}}
\\  \nonumber
&\leq  \frac{c'\delta}{\omega(\frac{\Lambda}{\lambda}s)}\bigg(\dashint_{\Omega \cap B_{\frac{\Lambda}{\lambda}s}}\abs{\bfA(\nabla{\bfu})-\bfA_{\frac{\Lambda}{\lambda}s}}^{\min{\set{2,p'}}}\dx\bigg)^\frac{1}{\min{\set{2,p'}}}
\\  \nonumber
&\quad 
+\frac{c'\delta\sigma(s)}{\omega(s)}\dashint_{\Omega\cap B_{2\frac{\Lambda}{\lambda} s}}\abs{\bfA(\nabla{\bfu})}\dx
+\frac{c'\delta\sigma(s)}{\omega(s)}\bigg(\dashint_{\Omega \cap B_{2\frac{\Lambda}{\lambda} s}}\abs{\bfF-\bfF_0}^{p'q}\dx \bigg)^\frac1{p'q} 
\end{align}
%please observe that
%\begin{align}
%\bfJ^{-t} \bfJ_s^t=
%  \begin{pmatrix}
%\bfI&0
%\\
%(\nabla \psi)^t &1
%\end{pmatrix} 
% \begin{pmatrix}
%\bfI& 0
%\\
%\mean{(\nabla \psi)^t}_{B^+_s}&1
%\end{pmatrix}
%=  \begin{pmatrix}
%\bfI&0
%\\
%\mean{(\nabla \psi)^t}_{B^+_s}-(\nabla \psi)^t&1.
%\end{pmatrix} 
%\end{align}
%This implies that
for some constants $c,c'$ and for every $\bfF_0 \in \setR^{N\times n}$.
By  the very definition of $\underline\bfF$ in \eqref{eq:fbar}, a change of variables and the boundedness of $\bfJ, \bfJ^{-1}, \bfJ_s$, % ensure that we can estimate the left term in \eqref{june43-hs1} as follows}
\begin{align}\label{Nov16}
%{\color{orange} (III)=}  
 \frac{1}{\omega(s)} \bigg(\dashint_{D_s}\abs{\underline\bfF}^{p'}\dx\bigg)^\frac1{p'}
& =
\frac{1}{\omega(s)} \bigg(\dashint_{D_s}\abs{({\bfF}(\Psi^{-1}(\bfJ_s z))-\bfF_0)\bfJ^t(\Psi^{-1}(\bfJ_s z))(\bfJ_s^{-1})^t}^{p'}\dz\bigg)^\frac1{p'}
\\ \nonumber
&\leq \frac{c}{\omega(s)} \bigg(\dashint_{D_{s}}\abs{{\bfF}(\Psi^{-1}(\bfJ_s z))-\bfF_0)}^{p'}\dz\bigg)^\frac1{p'}
\\ \nonumber
&\leq \frac{c}{\omega(s)} \bigg(\dashint_{\Omega\cap B_{\frac{\Lambda}{\lambda}s}}\abs{{\bfF}-\bfF_0}^{p'}\dx\bigg)^\frac1{p'}
\end{align}
for some constants $c,c'$ and for every $\bfF_0 \in \setR^{N\times n}$.
Finally, observe that
\begin{align}
 \bfJ_s^{-1} { \overline{\bfJ}(z)} =
 \begin{pmatrix}
\bfI&0
\\
\mean{\nabla \psi}_{\Omega \cap B_s}&1
\end{pmatrix}
  \begin{pmatrix}
\bfI& 0
\\
-\nabla \psi ((\Psi^{-1}(\bfJ_s z))')&1
\end{pmatrix}
=  \begin{pmatrix}
\bfI&0
\\
\mean{\nabla \psi}_{ \Omega \cap B_s}-\nabla \psi ((\Psi^{-1}(\bfJ_s z))')&1
\end{pmatrix} 
\end{align}
{for $z \in H_{\bfJ _s} \cap B_{\frac \lambda \Lambda R}$, where $\bfI$ is the unit matrix in $\setR^{(n-1) \times (n-1)}$, and $(\Psi^{-1}(\bfJ_s z))'$ denotes the vector in $\mathbb R^{n-1}$ of the first $(n-1)$ components of $\Psi^{-1}(\bfJ_s z)$.}
Thus, by Lemma~\ref{lem:hammerT}
\begin{align*}
\abs{\bfA(\nabla\overline{\bfu})-{\bfA_{\bfJ_s^{-1} \overline{\bfJ}}(\nabla\overline{\bfu})}}
\leq \abs{\bfI - \bfJ_s^{-1}\overline{\bfJ}}\abs{\nabla\overline{\bfu}}^{p-1}
%\leq \abs{\mean{\nabla \psi}_{\Omega \cap B_s}-\nabla \psi ((\Psi^{-1}(\bfJ_s z))')}\abs{\nabla\overline{\bfu}}^{p-1}
\leq \abs{\bfJ_s - \overline{\bf J}}\abs{\nabla\overline{\bfu}}^{p-1}.
\end{align*}
Hence, 
via the  last two inequalities in equation \eqref{eq:hoelder} %\color{blue}(with $q=1$)\color{black},  
\begin{align}\label{Nov17}
%{\color{orange} (IV)=}
 \frac{1}{\omega(s)}& \bigg(\dashint_{D_s}\abs{\bfA(\nabla\overline{\bfu})-\bfA_{\bfJ_s^{-1}\overline\bfJ}(\nabla\overline{\bfu})%\db{-\bfA_0}
}^{p'}\dz\bigg)^\frac1{p'}
\leq \frac{1}{\omega(s)} \bigg(\dashint_{{D_s}} \abs{\bfJ_s - \overline{\bf J}}^{p'}\abs{\nabla\overline{\bfu}}^{p}\dz\bigg)^\frac1{p'}
\\ \nonumber
&\leq \frac{c\sigma(s)}{\omega(s)}\dashint_{\Omega \cap B_{2\frac{\Lambda}{\lambda} s}}\abs{\bfA(\nabla{\bfu})}\dx
+\frac{c\sigma(s)}{\omega(s)}\bigg(\dashint_{\Omega \cap B_{2\frac{\Lambda}{\lambda} s}}\abs{\bfF-\bfF_0}^{{qp'}}\dx \bigg)^\frac1{{qp'}} 
\end{align}
for some constant $c$ and for every $\bfF_0 \in \setR^{N\times n}$.
Combining inequalities \eqref{june43-hs1}, \eqref{Nov13}, \eqref{eq:hoelder}, \eqref{Nov15}, \eqref{Nov16} and \eqref{Nov17} yields, on enlarging the domains of integration when necessary and making use of 
 \eqref{eq:omega condition},  
%$(i)\leq c(ii)+c\delta(II)+c(III)+c(IV)$.
%This implies, that
\begin{align}
\label{decayfirst}
&\frac{1}{\omega(\frac{\lambda}{\Lambda}\theta s)}\bigg(\dashint_{ \Omega \cap B_{\frac{\lambda}{\Lambda}\theta s}}\abs{\bfA(\nabla\bfu)-\bfA_{\frac{\lambda}{\Lambda}\theta s} }^{\min{\set{2,p'}}}
\dx\bigg)^{\frac{1}{\min{\set{2,p'}}}}
\\ \nonumber
& \leq
\frac{c\sigma(s)}{\omega(s)}\dashint_{\Omega \cap B_{2\frac{\Lambda}{\lambda} s}}\abs{\bfA(\nabla{\bfu})}\dx
+\frac{c\sigma(s)+c}{\omega(s)}\bigg(\dashint_{\Omega\cap B_{2\frac{\Lambda}{\lambda} s}}\abs{\bfF-\bfF_0}^{p'q}\dx \bigg)^\frac1{p'q} 
\\ \nonumber
&\quad + \frac{c\delta}{\omega(\frac{\Lambda}{\lambda}s)}\bigg(\dashint_{\Omega\cap B_{\frac{\Lambda}{\lambda}s}}\abs{\bfA(\nabla{\bfu})-\bfA_{\frac{\Lambda}{\lambda}s} }^{\min{\set{2,p'}}}\dx\bigg)^\frac{1}{\min{\set{2,p'}}}
\end{align}
for some constant $c$ and for every $\bfF_0 \in \setR^{N\times n}$.
Hence, inequality \eqref{Nov11} follows, by fixing $\delta$ sufficiently small and then redefining $\theta$ and $c$ accordingly. The fact that \eqref{Nov11} actually holds for every $q>1$, and  not just for $q \in (1, q_0]$, is a consequence of H\"older's inequality.
\end{proof}

\section{Proof of Theorem ~\ref{thm:campanato}}
%\
%\subsection{A pointwise gradient estimate}

A main step in our proof of Theorem ~\ref{thm:campanato} is a pointwise estimate for a sharp maximal function of   the gradient of   the weak solution $\bfu$ to the Dirichlet problem \eqref{eq:sysA}.
%
%\begin{align}
%\label{eq:sysAloc}
%\begin{cases}
%  -\divergence(|\nabla \bfu|^{p-2} \nabla \bfu) = -\divergence
%  \bfF\, & \quad \hbox{in $\Omega \cap B_{2R}(x)$,}
%\\ \bfu =0 & \quad  \hbox{on $\partial \Omega \cap B_{2R}(x)$.}
%\end{cases}
%\end{align}
The relevant sharp maximal function operator  has a local nature, and involves a parameter function $\omega$. Assume that $\Omega$ is a bounded  open set satisfying condition \eqref{eq:fat}. Let  $R>0$ and $q \in [1, \infty)$.  If $\bff$ is a real, vector or matrix-valued function in  $\Omega$ such that $\bff \in L^q(\Omega)$, we define the function $M^{\sharp,q}_{\omega,\Omega, R} \bff$ on $\setR^n$ as
\begin{equation}\label{Msharp}
(M^{\sharp,q}_{\omega,\Omega, R} \bff)(x)=\sup_{
\begin{tiny}
 \begin{array}{c}{
    B_r\ni x} \\
r\leq R
 \end{array}
  \end{tiny}
}\frac{1}{\omega(r)}\bigg(\dashint_{ \Omega\cap B_r} \abs{\bff-\mean{\bff}_{\Omega \cap B_r}}^q\dy\bigg)^\frac1q\, \quad \hbox{for $x \in \setR^n$.}
\end{equation}

\begin{proposition} \label{pro:wahnsinn}
Let $\Omega$ be a bounded open set in $\setR^n$ such that $\partial \Omega \in W^1\mathcal L^{\sigma (\cdot)} \cap C^{0,1}$   for some parameter function $\sigma$. Let $\omega$ be a parameter function fulfilling condition \eqref{eq:omega condition}. Assume that  $\bfF\in L^{p'q}(\Omega)$ for some $q>1$, and let 
$\bfu$ be the weak solution to the Dirichlet problem \eqref{eq:sysA}.  Then there exist positive constants  $\epsilon$, $R_0$ and $c$, depending on $n,p,N,\omega, q, \Omega$,
such
that, if
\begin{equation}\label{a7}
\sup_{r\in (0,1)}\frac{\sigma(r)}{\omega(r)}\int^1_r\frac{\omega(\rho)}{\rho}\,{\rm d}\rho \leq \epsilon,
\end{equation}
then
\begin{align}\label{a8}
M^{\sharp,\min\set{2, p'}}_{\omega,\Omega,{R/8}}(|\nabla \bfu|^{p-2}\nabla \bfu)(x)\leq cM^{\sharp,p'q}_{\omega,\Omega,2R}(\bfF)(x)+\frac{c}{\omega(R)}\dashint_{ \Omega \cap B_{2R}(x)}\abs{\nabla \bfu}^{p-1}\dy
\end{align}
for every $x \in \partial \Omega$ and $R\in (0,R_0)$.
\color{black}
Hence, 
%\todo{Check the multiple of $R$. SEB: We could even take 1 R, since the big radii are never the problem.}
\begin{align}\label{a15}
\norm{|\nabla \bfu|^{p-2}\nabla \bfu}_{\Lcal^{\omega (\cdot)}({\Omega\cap B_{R/16}}(x))}\leq c\norm{\bfF}_{\Lcal^{\omega (\cdot)}(\Omega \cap B_{2R}(x))}+\frac{c}{\omega(R)}\dashint_{\Omega \cap B_{2R}(x)}\abs{\nabla \bfu}^{p-1}\dy
\end{align}
for every $x \in \partial \Omega$ and $R\in (0,R_0)$.
%if for the Jacobi matrix $\bfJ=\nabla \overline{\omega}$ the quantity
%\begin{align}
%\label{eq:Jw}
%\sigma(\rho):=\bigg(\dashint_{ B_\rho^+}\abs{\bfJ-\mean{\bfJ}_{B_\rho^+}}^{p q'}\dx\bigg)^\frac{1}{q'p'}
%%+\bigg(\dashint_{B_\rho^+}\abs{\bfJ^{-1}-\mean{\bfJ^{-1}}_{B_\rho^+}}^{\min{\set{2,p'}}}\dx\bigg)^\frac{1}{\min{\set{2,p'}}},
%\end{align}
%satisfies,
%\[
%\sup_{r\in (0,R)}\frac{\sigma(r)\overline{\omega}(r)}{\omega(r)}\leq \epsilon_0.
%\]
%\color{orange}Let $\nu$ be the unity normal vector on $\partial\Omega$. \todo{\color{black} Are these inequality and the subsequent remark needed? It seems to me that a proof is even missing. \seb{I do not think they are needed. I included it to emphasize the Dirichlet boundary condition, but we can remove it.}} \color{orange}
%Then we have for any $\tau\in \mathcal{S}^{n-1}$ with $\tau\cdot\nu=0$
%\begin{align}\label{eq:2005}
%M^{\min\set{2, p'}}_{\omega,\Omega,R}(\bfA(\nabla \bfu)\cdot \tau_N)(x)\leq cM^{\sharp,p'q}_{\omega,\Omega,R}(\bfF)(x)+c\dashint_{B_R(x)\cap\Omega}\big(\abs{\bfA(\nabla \bfu)}+\abs{\bfF}\big)\,dx,
%\end{align}
%where $\tau_N=(\tau,\tau,...,\tau)^t\in\mathbb R^{N\times n}$. 
\end{proposition}

The following lemma will be needed in the proof of Proposition \ref{pro:wahnsinn}.
\begin{lemma}
\label{lem:iter} Let $\Omega$ be a bounded Lipschitz domain in $\setR^n$. Let $R\in (0,1)$ be such that, for every $x \in \partial \Omega$ there exists a map $\Psi :  \Omega \cap B_R(x) \to \setR^n$ as in \eqref{PSI}, with $\Psi \in C^{0,1}(\Omega \cap B_R(x))$.
Assume that $q\in [1,\infty)$ and $\omega$ is a parameter function satisfying condition \eqref{eq:omega condition}. Let $\overline \omega : (0, 1) \to [0, \infty)$ be the function defined as 
\begin{equation}\label{omegabar}
\overline{\omega}(r)=\int^1_r\frac{\omega(\rho)}{\rho}\,{\rm d}\rho \quad \hbox{for $r\in (0, 1]$.}
\end{equation}
Then there exists a constant $c=c(n,\beta,c_\omega,\lambda,\Lambda,q)$ such that
\begin{equation}\label{Nov20}
\dashint_{\Omega\cap B_r(x)}\abs{\bff}\dy\leq c\overline{\omega}(r)\sup_{\rho\in(r,R)}\frac{1}{\omega(\rho)}\bigg(\dashint_{ \Omega\cap B_\rho(x)} \abs{\bff-\mean{\bff}_{ \Omega \cap B_\rho}}^{q}\dy\bigg)^\frac{1}{q} + c\mean{\abs{\bff}}_{\Omega \cap B_R(x)} 
\end{equation}
for every 
$\bff \in L^q(\Omega)$, $x\in \partial \Omega$ and  $r \in (0,R]$.
%where $\overline{\omega}(r)=\int^R_r\omega(\rho)\frac{d\rho}{\rho}$.
\end{lemma}
\begin{proof}
%Here, we collect some inequalities involving averages on balls to be used in what follows. 
Let us keep the notations of 
Subsection \ref{sec:boundary} in force.  Moreover, we can assume that all balls are centered at $0$. For simplicity,  the center will thus be dropped in the notation. 
\iffalse

To begin with, observe that, owing to inequality \eqref{meanq}, property \eqref{inclusion1},  and the fact that $\abs{\bfJ^{-1}}=\abs{\nabla \Psi^{-1}} \leq \frac{\Lambda}{\lambda}$,
%
%
%  we apply the best constant property and the fact that the mapping preserves area in the sense that
%\[
%B_{\lambda r}^+\subset \Psi(B_r\cap\Omega)\subset B_{\Lambda r}^+.
%\]
%This yields for any $g\in L^q(\Omega)$
\begin{align}
\label{boundarybmo2}
&\bigg(\dashint_{\Omega \cap B_r}\abs{\bff-\mean{\bff}_{\Omega\cap B_r}}^q\dx\bigg)^\frac1q\leq 2\bigg(\dashint_{\Omega\cap B_r}\abs{\bff-\bfc_0}^q\dx\bigg)^\frac1q
\\ \nonumber &\db{\leq c}\bigg(\dashint_{\Psi(\Omega\cap B_r)}\abs{\widetilde{\bff}-\bfc_0}^q\abs{\det\bfJ^{-1}}\dy\bigg)^\frac1q
\leq c\bigg(\dashint_{B_{\Lambda r}^+}\abs{\widetilde{\bff}-\bfc_0}^q \dy\bigg)^\frac1q
\end{align}
for every $\bfc_0 \in \setR^m$. Hence, in particular,
%
%
%If $\widetilde{g}$ is reflected even,
%i.e. if $\overline{g}(x',h)=\widetilde{g}(x',-h)$ for $h<0$ and $\overline{g}(x',h)=\widetilde{g}(x',h)$ for $h\geq 0$, we find by \eqref{boundarybmo1} by choosing $c_0=\mean{\widetilde{g}}_{B_{\Lambda r}^+}$ \db{holds obviously, why need \eqref{boundarybmo1}?}
\begin{align}
\label{OtoH}
\bigg(\dashint_{\Omega\cap B_r}\abs{\bff-\mean{\bff}_{\Omega\cap B_r}}^q\dx\bigg)^\frac1q
\leq c\bigg(\dashint_{B_{\Lambda r}^+}\abs{\widetilde{\bff}-\mean{\widetilde{\bff}}_{B_{\Lambda r}^+}}^q \dy\bigg)^\frac1q.
\end{align}
%analogously, we find since \db{why need $\overline{g}$ in \eqref{OtoH}?}

\fi
Owing to properties \eqref{meanq} and \eqref{inclusion2}, a change of variables and the fact that $|\det \bfJ|=1$, we have that
%
%we have \seb{by the best constant property and a change of variables that }%\todo{check}
%\[
%B_{\frac{r}{\Lambda}}\cap \Omega \subset \Psi^{-1}(B_r^+)\subset B_{\frac{r}{\lambda}}\cap \Omega,
%\]
%that
\begin{align}
\label{HtoO}
\bigg(\dashint_{B_{ \lambda r}^+}\abs{\widetilde{\bff}-\mean{\widetilde{\bff}}_{B_{ \lambda r}}}^q \dy\bigg)^\frac1q
\leq c \bigg(\dashint_{B_{ \lambda r}}\abs{\widetilde{\bff}-\mean{\widetilde{\bff}}_{B_{ \lambda r}}}^q \dy\bigg)^\frac1q
\leq c'
%\seb{c
%(\lambda,\Lambda)}
 \bigg(\dashint_{\Omega \cap B_{r}}\abs{\bff-\mean{\bff}_{\Omega\cap B_r}}^q\dx\bigg)^\frac1q.
\end{align}
for some constants $c$ and $c'$.
%
%Furthermore, we collect some necessary preliminary analysis, which makes a useful link between telescope sum arguments and mean oscillations. The following basic calculations will indeed emphasize possible assumptions on boundary conditions (or assumptions on the coefficients of the flattened system) with respect to Campanato spaces, which eventually will turn out to be optimal.
Next, let $B_R$ be any ball in $\setR^n$ and let   $\bfg \in L^q(B_R)$.  We claim that 
\begin{align}
\label{logiter}
\abs{\mean{\abs{\bfg}}_{ B_r}-\mean{\abs{\bfg}}_{B_R}}
\leq 2^{2n+2}\int^R_r \bigg(\dashint_{B_\rho}\abs{\bfg-\mean{\bfg}_{ B_\rho}}^q\, dx\bigg)^\frac1q\frac{d\rho}\rho
%\\ \nonumber
%\leq 2^{2n+2}\log(\db{R/r})\sup_{\tau\in [r,R]}\bigg(\dashint_{B_\tau}\abs{\bfg-\mean{\bfg}_{B_\tau}}^q\dx\bigg)^\frac1q 
\quad \hbox{for $r \in (0, R]$,}
\end{align}
and that  inequality \eqref{logiter} contiunes to hold if balls are replaced by half-balls.
 To prove our claim, observe that
\begin{align*}
  \big|\mean{\bfg}_{\frac 12 B_R}-\mean{\bfg}_{B_R}\big|\leq
  \dashint_{\frac 12 B_R}\abs{\bfg-\mean{\bfg}_{B_R}}\,\dx \leq  2^n\dashint_{{B_R}}\abs{\bfg-\mean{\bfg}_{B_R}}\, \dx.
\end{align*}
Let $m \in \setN$. By iterating the previous inequality, we obtain that
\begin{align}
  \label{eq:iteration}
  \abs{\mean{\bfg}_{2^{-m}B_R}-\mean{\bfg}_{B_R}}&\leq 2^n\sum_{i=0}^{m-1}
    \dashint_{2^{-i}B_R}\abs{\bfg-\mean{\bfg}_{2^{-i}B_R}}\dx.
%\\ \nonumber & 
%\leq m2^n\max_{0\leq i\leq m-1}\dashint_{2^{-i}B_R}\abs{\bfg-\mean{\bfg}_{2^{-i}B_R}}\dx.
\end{align}
Given any $\tau\in (2^{-m},2^{1-m}]$,  property \eqref{meanq} ensures that
\[
\dashint_{2^{-m}B_R}\abs{\bfg-\mean{\bfg}_{2^{-m}B_R}}\dx\leq  {2^{n+1}}\dashint_{ \tau B_R}\abs{\bfg-\mean{\bfg}_{\tau B_R}}\dx.
%\leq  \color{blue} 4  \cdot \color{black}  \db{2^{2n}}\dashint_{ 2^{1-m}B_R}\abs{\bfg-\mean{\bfg}_{2^{1-m}B_R}}\dx.
\]
Consequently, 
\begin{align}\label{a4}
\dashint_{2^{-m}B_R}\abs{\bfg-\mean{\bfg}_{2^{-m}B_R}}\dx & \leq 2^{n+1}\dashint_{2^{-m}}^{2^{1-m}}\dashint_{ \tau B_R}\abs{\bfg-\mean{\bfg}_{\tau B_R}}\dx\,{\rm d}\tau
\\ \nonumber & \leq 2^{n+1}\int_{2^{-m}}^{2^{1-m}}\dashint_{ \tau B_R}\abs{\bfg-\mean{\bfg}_{\tau B_R}}\dx\frac{{\rm d}\tau}{\tau}.
\end{align}
Moreover,
\begin{align}\label{a5}
\abs{\mean{\bfg}_{ \tau B_R}-\mean{\bfg}_{2^{1-m}B_R}}%= \bigg|\dashint _{\tau B} (g- \mean{g}_{2^{1-m}B})\, \dx\bigg| 
\leq \dashint _{\tau B_R} |\bfg- \mean{\bfg}_{2^{1-m}B_R}|\, \dx \leq 2^n \dashint _{2^{1-m} B_R} |\bfg- \mean{\bfg}_{2^{1-m}B_R}|\, \dx\,.
\end{align}
Given $\theta \in (0, 1)$, choose $m \in \mathbb N$ \color{black} in such a way that $\theta\in (2^{-m},2^{1-m}]$. Then, we deduce from \eqref{a5}, \eqref{eq:iteration} and \eqref{a4} that
\begin{align}\label{a6}
\abs{\mean{\bfg}_{ \theta B_R}-\mean{\bfg}_{B_R}}
%\\
&
\leq\abs{\mean{\bfg}_{ \theta B_R}-\mean{\bfg}_{2^{1-m}B_R}}+\abs{\mean{\bfg}_{ 2^{1-m} B_R}-\mean{\bfg}_{B_R}}
\\ \nonumber
&
\leq 2^n\dashint_{{2^{1-m}B_R}}\abs{\bfg-\mean{\bfg}_{2^{1-m}B_R}}\dx+2^n\sum_{i=0}^{m-2}   \dashint_{2^{-i}B_R}\abs{\bfg-\mean{\bfg}_{2^{-i}B_R}}\dx
\\  \nonumber
&\leq 2^n\sum_{i=0}^{m-1}   \dashint_{2^{-i}B_R}\abs{\bfg-\mean{\bfg}_{2^{-i}B_R}}\dx
\leq  2^{2n+1}\int_{2^{1-m}}^1\dashint_{ \tau B_R}\abs{\bfg-\mean{\bfg}_{\tau B_R}}\dx\frac{{\rm d}\tau}{\tau}
\\   \nonumber
&\leq 2^{2n+1}\int_{\theta}^1\dashint_{ \tau B_R}\abs{\bfg-\mean{\bfg}_{\tau B_R}}\dx\frac{{\rm d}\tau}{\tau}
%\leq 2^{2n+1}\log(1/\theta)\sup_{\tau\in [\theta,1]}\dashint_{\tau B}\abs{g-\mean{g}_{\tau B}}dx
\end{align}
Property \eqref{meanq} ensures that
\begin{align}
\label{eq:abs}
 \dashint_{\tau B_R}\abs{\abs{\bfg}-\mean{\abs{\bfg}}_{\tau B_R}}\dx\leq 2\dashint_{\tau B_R}\abs{\abs{\bfg}-\abs{\mean{\bfg}_{\tau B_R}}}\dx \leq 2\dashint_{\tau B_R}\abs{\bfg-\mean{\bfg}_{\tau B_R}}\dx 
\end{align}
for $\tau \in (0,1]$.
Inequality \eqref{a6} (applied with $\bfg$ replaced by $|\bfg|$),  inequality \eqref{eq:abs} and H\"older's inequality imply that
\begin{align}\label{logiter1}
& \abs{\mean{\abs{\bfg}}_{ \theta B_R}-\mean{\abs{\bfg}}_{B_R}} \leq 2^{2n+1}\int^1_\theta \dashint_{\tau B_R}\abs{|\bfg|-\mean{|\bfg|}_{\tau B_R}}\, \dx\frac{{\rm d}\tau}\tau
\\ \nonumber
& \quad \quad 
\leq 2^{2n+2}\int^1_\theta \dashint_{\tau B_R}\abs{\bfg-\mean{\bfg}_{\tau B_R}}\, \dx\frac{{\rm d}\tau}\tau
%\\ \nonumber
%& 
\leq 2^{2n+2}\int^1_\theta \bigg(\dashint_{\tau B_R}\abs{\bfg-\mean{\bfg}_{\tau B_R}}^q\, \dx\bigg)^{\frac 1q}\frac{{\rm d}\tau}\tau 
%\\ \nonumber & \leq 2^{2n+2}\log(1/\theta)\sup_{\tau\in [\theta,1]}\bigg(\dashint_{\tau B_R}\abs{\bfg-\mean{\bfg}_{\tau B_R}}^q\dx\bigg)^{\frac 1q}.
\end{align}
for every $q \in [1,\infty)$.
%And by Jensen's inequality we get, for $q\in [1,\infty)$ 
%\begin{align}
%\label{logiter1}
%\abs{\mean{\abs{g}}_{ \theta B}-\mean{\abs{g}}_{B}}
%&\leq 2^{2n+2}\int^1_\theta \bigg(\dashint_{\tau B}\abs{g-\mean{g}_{\tau B}}^q\, dx\bigg)^\frac1q\frac{d\tau}\tau
%\\ \nonumber
%&\leq 2^{2n+2}\log(1/\theta)\sup_{\tau\in [\theta,1]}\bigg(\dashint_{\tau B}\abs{g-\mean{g}_{\tau B}}^q\, dx\bigg)^\frac1q.
%\end{align}
Given any $r\in (0, {R}]$, the choice $\theta = \frac rR$ in \eqref{logiter1} and a change of variables in the last integral yield \eqref{logiter}.
%, By mere scaling, we find for  \color{blue} WHY   $\frac{R}2$? \seb{It is Ok for $R$}. \color{black}  \todo{I think it is better to prove \eqref{logiter1} with $r$ and $R$ directly}
\\ Having inequality \eqref{logiter} at disposal, we are ready to  accomplish the proof of inequality \eqref{Nov20}. 
If $r\in [\frac{\lambda R}{\Lambda},R]$, then inequality  \eqref{Nov20}  follows by enlarging the domain of integration.
Assume, next, that $r\in [0,\frac{\lambda R\color{black}}{\Lambda}]$. From   a change of variables and inequality
\eqref{logiter}
%\db{isn't it the second last inequality of (4.17) rather than (4.11)?}, applied with $\bfg = \widetilde \bff$ and
 in its version for half-balls, we deduce  that  there exist constants $c$ and $c'$ such that
%for \seb{$\widetilde{g}(y)=g\circ\Psi^{-1}(y)$ for $y\in \Psi(B_R(0))$.}
 %\db{why $\overline g$?}
\begin{align}\label{marzo4}
\dashint_{B_{ r}\cap \Omega }\abs{\bff}\dx
&\leq c\dashint_{B_{\Lambda r}^+}\abs{\widetilde \bff\color{black}}\dy
\leq c\abs{\mean{\abs{\widetilde  \bff}}_{B_{\Lambda r}^+}-\mean{\abs{\widetilde  \bff}}_{B_{{\lambda R}}^+}}+c\mean{\abs{\widetilde  \bff}}_{B_{\lambda R}^+}
%\\ \nonumber
%&\leq c\int_{\Lambda r}^{\lambda R}\ \bigg(\dashint_{B_\rho^+} \abs{\widetilde  \bff-\mean{\widetilde  \bff}_{B_\rho^+}}^{q}\dx\bigg)^\frac{1}{q}\frac{d\rho}{\rho} 
%+ c\mean{\abs{\widetilde  \bff}}_{B_{\lambda R}^+}.
\\ \nonumber
&\leq c'\int_{\Lambda r}^{\lambda R}\frac{\omega(\rho)}{\rho}\,{\rm d}\rho \sup_{\tau\in({ {\Lambda}r, {\lambda}R})}\frac{1}{\omega(\tau)} \bigg(\dashint_{B_\tau^+} \abs{\widetilde  \bff-\mean{\widetilde  \bff}_{B_\tau^+}}^q\dy\bigg)^\frac{1}{q} 
+ c\mean{\abs{\widetilde  \bff}}_{B_{\lambda R}^+}
\\ \nonumber
&\leq c'\, \overline{\omega}(r)\sup_{\tau\in({ {\Lambda}r,\lambda R)}}\frac{1}{\omega(\tau)}\bigg(\dashint_{B_\tau^+} \abs{\widetilde  \bff-\mean{\widetilde  \bff}_{B_\tau^+}}^{q}\dy\bigg)^\frac{1}{q} 
+ c\mean{\abs{\widetilde  \bff}}_{B_{\lambda R}^+}.
\end{align} 
Note that in the last inequality we have made use of the fact that $\overline \omega (r) \geq \int_{\Lambda r}^{\lambda R}\frac{\omega(\rho)d\rho}{\rho}$, thanks to assumption \eqref{lL}.
Owing to inequalities \eqref{marzo4},  \eqref{HtoO}, \eqref{inclusion1} and to condition \eqref{eq:omega condition},  there exist constants $c$ and $c'$ such that
% (\color{blue} since $\overline{\omega}$ is non-decreasing  ?????\color{black})
\begin{align*}
\dashint_{\Omega \cap B_{ r}}\abs{\bff}\dx
&\leq c\, \overline{\omega}(r)\sup_{\rho\in({ {\Lambda}r,\lambda R)}}\frac{1}{\omega(\rho)}\bigg(\dashint_{ \Omega \cap B_{\rho/ {\lambda}}} \abs{ \bff-\mean{ \bff}_{\Omega \cap B_{\rho/{\lambda}}}}^{q}\dx\bigg)^\frac{1}{q} 
+ c\mean{\abs{\bff}}_{\Omega\cap B_{R}}
\\ & = 
c\, \overline{\omega}(r)\sup_{\rho\in({{\frac \Lambda\lambda }r,  R)}}\frac{1}{\omega(\lambda \rho)}\bigg(\dashint_{ \Omega \cap B_{{\rho}}} \abs{\bff-\mean{\bff}_{ \Omega \cap B_{\rho}}}^{q}\dx\bigg)^\frac{1}{q} 
+ c\mean{\abs{\bff}}_{\Omega \cap B_{R}}
\\ & \leq 
c'\, \overline{\omega}(r)\sup_{\rho\in({r, R)}}\frac{1}{\omega(\rho)}\bigg(\dashint_{\Omega\cap B_{\rho}}  \abs{\bff-\mean{\bff}_{\Omega \cap B_{\rho}}}^{q}\dx\bigg)^\frac{1}{q} 
+ c\mean{\abs{\bff}}_{\Omega\cap B_{R}},
\end{align*} 
namely \eqref{Nov20}.
\end{proof}
%\begin{lemma}
%\label{lem:iter}
%Let $f\in L^1(B_R^+)$. Then for $r\in (0,\frac{R}{2})$ we find
%\[
%\dashint_{B_r}\abs{f}\dx\leq c\overline{\omega}(r)\sup_{\rho\in(r,R)}\frac{1}{\omega(\rho)}\bigg(\dashint_{B_\rho^+} \abs{f-\mean{f}_{B_\rho^+}}^{q}\dx\bigg)^\frac{1}{q} + c\mean{\abs{f}}_{B_R^+},
%\]
%where $\overline{\omega}(r)=\int^R_r\omega(\rho)\frac{d\rho}{\rho}$.
%\end{lemma}
%\begin{proof}
%We will use \eqref{logiter} for $r\in (0,R/2]$, that
%\begin{align*}
%\dashint_{B_r^+}\abs{f}\dx&\leq \abs{\mean{\abs{f}}_{B_r^+}-\mean{\abs{f}}_{B_R^+}}+\mean{\abs{f}}_{B_R^+}
%\\
%&\leq \int_{r}^R \bigg(\dashint_{B_\rho^+} \abs{f-\mean{f}_{B_\rho^+}}^{q}\dx\bigg)^\frac{1}{q}\frac{d\rho}{\rho} + \mean{\abs{f}}_{B_R^+}.
%\\
%&\leq c\int_{r}^R\frac{\omega(\rho)d\rho}{\rho} \sup_{\rho\in(r,R)}\frac{1}{\omega(\rho)} \bigg(\dashint_{B_\rho^+} \abs{f-\mean{f}_{B_\rho}^+}^q\dx\bigg)^\frac{1}{q} + \mean{\abs{f}}_{B_R^+}.
%\\
%&\leq c \overline{\omega}(r)\sup_{\rho\in(r,R)}\frac{1}{\omega(\rho)}\bigg(\dashint_{B_\rho^+} \abs{f-\mean{f}_{B_\rho^+}}^{q}\dx\bigg)^\frac{1}{q} + \mean{\abs{f}}_{B_R^+}.
%\end{align*}
%\end{proof}

%\subsection{Pointwise Estimates for $\bfA(\nabla\bfu)$}\label{pf-main}

\begin{proof}[Proof of Proposition \ref{pro:wahnsinn}]
We keep the notations of the previous sections in force. %In particular we assume that $x=0$.
%   and \seb{$\nu(0)=(0,-1)^t$}. 
Let $R_0$ denote the minimum among $1$ and the radii $R$ for which the assumptions of Proposition \ref{pro:boundarydecay} are fulfilled. Observe that $R_0>0$, since $\partial \Omega$ is compact.
Our first aim is to estimate  the quantity %\db{recall def. of $\bfA_s$}
\[
\sup_{s\in (0,R)} \frac{1}{\omega(s)}\bigg(\dashint_{ \Omega\cap B_{s}(x)}\abs{\bfA(\nabla\bfu)-\bfA_{s} }^{\min{\set{2,p'}}}
\dx\bigg)^{\frac{1}{\min{\set{2,p'}}}},
\] 
where $R\in (0, R_0)$, and $\bfA_s$ is defined as in \eqref{a9}. Let $\theta \in (0,1)$ be the parameter appearing in the statement of Proposition~\ref{pro:boundarydecay}. 
First, assume that $s\in [\frac{\theta R}{2},R]$. By enlarging the domain of integration,  and exploiting condition \eqref{eq:omega condition},   triangle inequality, \color{black}
% the best constant property, 
%Proposition~\ref{cor:VPL} 
inequality \eqref{Nov3}
(applied with $\bfF_0=\mean{\bfF}_{ \Omega\cap B_{2R}}$ {and $q=1$}) and H\"older's inequality, we infer that
\begin{align}
\label{eq:bigballs}
\frac{1}{\omega(s)}&\bigg(\dashint_{ \Omega \cap B_{s}(x)}\abs{\bfA(\nabla\bfu)-\bfA_{s}}^{\min{\set{2,p'}}}
\dy\bigg)^{\frac{1}{\min{\set{2,p'}}}}
\\ \nonumber
&
\leq \frac{c}{\omega(R)}\bigg(\dashint_{\Omega \cap  B_{R}(x)}\abs{\bfA(\nabla\bfu)}^{\min\set{2,p'}}
\dy\bigg)^{\frac{1}{\min{\set{2,p'}}}}
\\ \nonumber
&
\leq \frac{c'}{\omega(R)}\dashint_{ \Omega \cap B_{2R}(x)}\abs{\bfA(\nabla\bfu)}\dy + c'M^{\sharp,p'q}_{\omega,\Omega, {2}R}(\bfF)(x)
\end{align}
for some constants $c$ and $c'$.
Hence,
\begin{align}\label{a11}
\sup _{s \in (\frac{\theta R}2, R)}\frac{1}{\omega(s)}&\bigg(\dashint_{\Omega \cap  B_{s}(x)}\abs{\bfA(\nabla\bfu)-\bfA_{s}}^{\min{\set{2,p'}}}
\dy\bigg)^{\frac{1}{\min{\set{2,p'}}}}
%\\
% \nonumber
%&\leq \frac{c}{\omega(R)}\bigg(\dashint_{ B_{R}\cap\Omega}\abs{\bfA(\nabla\bfu)}^{\min\set{2,p'}}
%\dx\bigg)^{\frac{1}{\min{\set{2,p'}}}}
\\ \nonumber
&\leq \frac{c'}{\omega(R)}\dashint_{\Omega \cap  B_{2R}(x)}\abs{\bfA(\nabla\bfu)}\dy + c'M^{\sharp,p'q}_{\omega,\Omega, {2}R}(\bfF)(x).
\end{align}
Assume next that $s\in (0, \frac{\theta R}{2})$.  By Proposition~\ref{pro:boundarydecay}, there exists a costant $c$ such that
%if   $r\in (0,\frac{R}{2})$ \db{$r\in (0,\frac{R}{2})$ only needed below}, then
\begin{align}\label{eq:0607}
&\frac{1}{\omega(\theta s)}\bigg(\dashint_{\Omega \cap  B_{{\theta}s}(x)}\abs{\bfA(\nabla\bfu)-\bfA_{\theta s}}^{\min{\set{2,p'}}}
\dy\bigg)^{\frac{1}{\min{\set{2,p'}}}}
\\ \nonumber
&\quad \leq
\frac{c\sigma(s)}{\omega(s)}\dashint_{\Omega \cap B_{s}(x)}\abs{\bfA(\nabla{\bfu})}\dy
+\frac{c\sigma(s)+c}{\omega(s)}\bigg(\dashint_{\Omega \cap B_{s}(x)}\abs{\bfF-\bfF_0}^{p'q}\dy \bigg)^\frac1{p'q} 
\\ \nonumber
&\qquad + \frac{1}{2\omega(s)}\bigg(\dashint_{\Omega \cap B_s(x)}\abs{\bfA(\nabla{\bfu})-\bfA_{s}}^{\min{\set{2,p'}}}\dy\bigg)^\frac{1}{\min{\set{2,p'}}}.
\end{align}
Hence, via Lemma~\ref{lem:iter},   there exists a constant $c$ such that %\db{same as formula above?}\db{why need lemma?}
\begin{align}
\label{forwahnsinn1}
&\frac{1}{\omega(\theta s)}\bigg(\dashint_{ \Omega \cap B_{\theta s}(x)}\abs{\bfA(\nabla\bfu)-\bfA_{\theta s}}^{\min{\set{2,p'}}}
\dy\bigg)^{\frac{1}{\min{\set{2,p'}}}}
\\ \nonumber
&\quad \leq
\frac{c\sigma(s)\overline{\omega}(s)}{\omega(s)}\sup_{\rho\in (s,R)}\frac{1}{\omega(\rho)}\bigg(\dashint_{\Omega \cap B_{\rho}(x)}\abs{\bfA(\nabla{\bfu})-\bfA_{\rho}}^{\min{\set{2,p'}}}\dy\bigg)^\frac{1}{\min{\set{2,p'}}}
\\ \nonumber
&\qquad +\frac{c\sigma(s)}{\omega(s)}\dashint_{ \Omega \cap B_R(x)}\abs{\bfA(\nabla \bfu)}\, \dy
+\frac{c\sigma(s)+c}{\omega(s)}\bigg(\dashint_{\Omega \cap B_{s}(x)}\abs{\bfF-\bfF_0}^{p'q}\dy \bigg)^\frac1{p'q} 
\\ \nonumber
&\qquad + \frac{1}{2\omega(s)}\bigg(\dashint_{ \Omega \cap B_{s}(x)}\abs{\bfA(\nabla{\bfu})-\bfA_{s} }^{\min{\set{2,p'}}}\dy\bigg)^\frac{1}{\min{\set{2,p'}}}.
\end{align}
Let $c$ be the constant appearing in inequality \eqref{forwahnsinn1}. Let us choose $\epsilon$ so small  in condition \eqref{a7}  that
%
%We assume that for the constant above (depending on $\theta$ which is fixed via Proposition~\ref{pro:boundarydecay})  $\epsilon$ is small enough, such that
\begin{equation}\label{Nov23}
\sup _{s \in (0,1)}\frac{c\,\sigma(s)\overline{\omega}(s)}{\omega(s)}\leq \frac{1}{4}.
\end{equation}
By \eqref{eq:omega condition},  
\begin{equation}\label{Nov24}
\overline{\omega}(s)\geq \int^R_\frac{R}{2}\frac{\omega(\rho)}{\rho}\,{\rm \rho} s\geq\frac12 \dashint^R_\frac{R}{2}{\omega(\rho)\,{\rm d}\rho}\geq c\,\omega(R) \quad \hbox{if $s \in (0, \frac R2)$,}
\end{equation}
for some positive constant $c$.
By inequalities \eqref{Nov23} and \eqref{Nov24}, there exists a constant $c$ such that
\begin{equation}\label{Nov25}
\frac{\sigma(s)}{\omega(s)}\leq
% \frac{1}{4\overline{\omega}(s)}\leq 
\frac{c}{\omega(R)} \quad \hbox{if $s \in (0, \frac R2)$.}
\end{equation}
From inequality \eqref{forwahnsinn1} with $\bfF_0=\mean{\bfF}_{ \Omega \cap B_s(x)}$,  inequalities \eqref{Nov23} and \eqref{Nov25}, and the boundedness of $\sigma$, one deduces that
\begin{align}\label{a10}
&\frac{1}{\omega(\theta s)}\bigg(\dashint_{ \Omega \cap B_{\theta s}(x)}\abs{\bfA(\nabla\bfu)-\bfA_{\theta s} }^{\min{\set{2,p'}}}
\dy\bigg)^{\frac{1}{\min{\set{2,p'}}}}
\\ \nonumber
&\quad \leq
\frac{3}{4}\sup_{\rho\in (s,R)}\frac{1}{\omega(\rho)}\bigg(\dashint_{ \Omega \cap B_{\rho}(x)}\abs{\bfA(\nabla{\bfu})-\bfA_{\rho} }^{\min{\set{2,p'}}}\dy\bigg)^\frac{1}{\min{\set{2,p'}}}
\\ \nonumber
&+\frac{c}{\omega(R)}\dashint_{\Omega \cap B_R(x)}\abs{\bfA(\nabla \bfu)}\, \dy
+\frac{c}{\omega(s)}\bigg(\dashint_{\Omega \cap B_{s}(x)}\abs{\bfF-\mean{\bfF}_{ \Omega \cap B_s}}^{p'q}\dy \bigg)^\frac1{p'q}  
\end{align}
%\color{blue} In particular, note that the last addend on the right-hand side of inequality \eqref{forwahnsinn1} is absorbed in first term on the right-hand side of \eqref{a10}. 
%\color{orange} This seems too implicit (and not clear to me). I would replace it by the text in blue below. \seb{Looks good. Please edit it.}
%\\
%Next fix $a\in (0,R)$. We take the supremum of the above over $(a,R)$ and use \eqref{eq:bigballs} to gain 
% \begin{align*}
%&\sup_{s\in (a,R)}\frac{1}{\omega(\theta s)}\bigg(\dashint_{ B_{\theta s}\cap\Omega}\abs{\bfA(\nabla\bfu)-\bfA_{\theta s} }^{\min{\set{2,p'}}}
%\dx\bigg)^{\frac{1}{\min{\set{2,p'}}}}
%\\
%&\quad \leq
%\frac{3}{4}\sup_{s\in (a,\theta R)}\frac{1}{\omega(s)}\bigg(\dashint_{B_{\rho}\cap \Omega}\abs{\bfA(\nabla{\bfu})-\bfA_{s} }^{\min{\set{2,p'}}}\dz\bigg)^\frac{1}{\min{\set{2,p'}}}+\frac{c}{\omega(2R)}\dashint_{B_{2R}\cap \Omega}\abs{\bfA(\nabla \bfu)}\, dx
%\\
%&\qquad
%+cM^{\sharp,p'q}_{\omega,\Omega,R}(\bfF)(0).
%\end{align*}
%We absorb the first term from the right-hand side and obtain the result by letting $a\to 0$.
for some constant $c$, provided that $s \in (0, \frac R2)$.
\color{black}
%\todo{\seb{I removed an older version}}
Inasmuch as $\theta s \in (0, \tfrac{\theta R}2)$ if and only if $s\in (0, \tfrac R2)$, inequality \eqref{a10} implies that
%, if $s \in (0, \tfrac{\theta R}2)$, then
%\begin{align}\label{a12}
%&\frac{1}{\omega(s)}\bigg(\dashint_{ B_{s}\cap\Omega}\abs{\bfA(\nabla\bfu)-\bfA_{s} }^{\min{\set{2,p'}}}
%\dx\bigg)^{\frac{1}{\min{\set{2,p'}}}}
%\\ \nonumber
%&\quad \leq
%\frac{3}{4}\sup_{\rho\in (\frac s \theta, R)}\frac{1}{\omega(\rho)}\bigg(\dashint_{B_{\rho}\cap \Omega}\abs{\bfA(\nabla{\bfu})-\bfA_{\rho} }^{\min{\set{2,p'}}}\dz\bigg)^\frac{1}{\min{\set{2,p'}}}+\frac{c}{\omega(R)}\dashint_{B_R\cap \Omega}\abs{\bfA(\nabla \bfu)}\, dx
%\\ \nonumber
%&\qquad
%+\frac{c}{\omega(\frac s \theta)}\bigg(\dashint_{B_{\frac s \theta}\cap\Omega}\abs{\bfF-\mean{\bfF}_{B_r\cap \Omega}}^{p'q}\dx \bigg)^\frac1{p'q} .
%\end{align}
%Hence, 
\begin{align}\label{a13}
&\sup _{s \in (0, \frac{\theta R}2)}\frac{1}{\omega(s)}\bigg(\dashint_{\Omega \cap  B_{s}(x)}\abs{\bfA(\nabla\bfu)-\bfA_{s} }^{\min{\set{2,p'}}}
\dy\bigg)^{\frac{1}{\min{\set{2,p'}}}}
\\ \nonumber
&\quad \leq
\frac{3}{4}\sup_{\rho\in (0, R)}\frac{1}{\omega(\rho)}\bigg(\dashint_{ \Omega \cap B_{\rho}(x)}\abs{\bfA(\nabla{\bfu})-\bfA_{\rho} }^{\min{\set{2,p'}}}\dy\bigg)^\frac{1}{\min{\set{2,p'}}}
\\ \nonumber
&\quad  \quad  +\frac{c}{\omega(R)}\dashint_{\Omega \cap B_R(x)}\abs{\bfA(\nabla \bfu)}\, \dy
+\sup _{s \in (0, \frac{\theta R}2)}\frac{c}{\omega(\frac s \theta)}\bigg(\dashint_{\Omega \cap B_{\frac s \theta}(x)}\abs{\bfF-\mean{\bfF}_{\Omega \cap B_{\frac s \theta}}}^{p'q}\dy \bigg)^\frac1{p'q}.
\end{align}
Here, we have made use of the fact that $\sup _{s \in (0, \frac{\theta R}2)} \sup_{\rho\in (\frac s \theta, R)}{(\cdot)} = \sup_{\rho\in (0, R)}{(\cdot)}$.  Since the last term on the right-hand side of \eqref{a13} does not exceed  $M^{\sharp,p'q}_{\omega,\Omega,{2}R}(\bfF)(x)$ times a suitable constant, 
coupling inequalities \eqref{a11} and \eqref{a13} yields
\begin{align}\label{a14}
\sup _{s \in (0, R)}\frac{1}{\omega(s)}&\bigg(\dashint_{ \Omega \cap B_{s}(x)}\abs{\bfA(\nabla\bfu)-\bfA_{s} }^{\min{\set{2,p'}}}
\dy\bigg)^{\frac{1}{\min{\set{2,p'}}}}
\\ \nonumber
&\quad \leq
\frac{3}{4}\sup_{\rho\in (0, R)}\frac{1}{\omega(\rho)}\bigg(\dashint_{\Omega \cap B_{\rho}(x)}\abs{\bfA(\nabla{\bfu})-\bfA_{\rho} }^{\min{\set{2,p'}}}\dy\bigg)^\frac{1}{\min{\set{2,p'}}}
\\ \nonumber
&\qquad
+\frac{c}{\omega(R)}\dashint_{\Omega \cap B_R(x)}\abs{\bfA(\nabla \bfu)}\, \dy
+ cM^{\sharp,p'q}_{\omega,\Omega,{2}R}(\bfF)(x)
\end{align}
for some constant $c$.
Absorbing the first term on the right-hand side of inequality \eqref{a14} in the left-hand side tells us that
\begin{align}\label{marzo6}
\sup _{s \in (0, R)}\frac{1}{\omega(s)}&\bigg(\dashint_{ \Omega \cap B_{s}(x)}\abs{\bfA(\nabla\bfu)-\bfA_{s} }^{\min{\set{2,p'}}}
\dy\bigg)^{\frac{1}{\min{\set{2,p'}}}}
\\ \nonumber
&\qquad
\leq \frac{4c}{\omega(R)}\dashint_{\Omega \cap B_R(x)}\abs{\bfA(\nabla \bfu)}\, \dy
+ 4cM^{\sharp,p'q}_{\omega,\Omega, {2}R}(\bfF)(x).
\end{align}
Now, let $y \in B_{R/8}(x)$, and let $r\leq R/8$. Set $s_y= {\rm dist}(y, \partial \Omega)$. Assume first that $r < s_y/2$, whence $B_{2r}(y) \subset \Omega$. By the local inner estimate of \cite[Theorem 1.3 and Remark 1.4]{BCDKS}, there exists a constant $c$ such that
\begin{align}\label{marzo7}
\frac{1}{\omega(r)}&\bigg(\dashint_{B_{r}(y)}\abs{\bfA(\nabla\bfu)-\langle\bfA(\nabla\bfu)\rangle_{B_{r}(y)}}^{\min{\set{2,p'}}}
\dz\bigg)^{\frac{1}{\min{\set{2,p'}}}} \\ \nonumber & \leq 
\frac{c}{\omega(r)}\dashint_{ B_{2r}(y)}\abs{\bfA(\nabla \bfu)}\, \dz
+ cM^{\sharp,p'}_{\omega,\Omega,r}(\bfF)(y).
\end{align}
Since $B_{kr}(y) \subset B_{kR/8}(y) \subset B_{kR/4}(x)$ for $k>0$,
%$r \leq R$ and $B_r(y) \subset B_{R/8}(y) \subset B_{R/4}(x)$, 
we infer from inequality \eqref{marzo7} and condition \eqref{eq:omega condition} that
\begin{align}\label{marzo8}
\frac{1}{\omega(r)}&\bigg(\dashint_{\Omega \cap B_{r}(y)}\abs{\bfA(\nabla\bfu)-\langle\bfA(\nabla\bfu)\rangle_{\Omega \cap B_{r}(y)} }^{\min{\set{2,p'}}}
\dz\bigg)^{\frac{1}{\min{\set{2,p'}}}} \\ \nonumber & \leq 
\frac{c}{\omega(R)}\dashint_{ B_{R/2}(x)}\abs{\bfA(\nabla \bfu)}\, \dz
+ cM^{\sharp,p'}_{\omega,\Omega,R/4}(\bfF)(x)
\end{align}
for some constant $c$.
Suppose next that $r\geq s_y/2$. Then there exists $x_y \in \partial \Omega \cap B_{R/8}{\color{purple}(x)}$ such that $B_{2r}(y) \subset B_{4r}(x_y)$. Therefore, by inequality \eqref{marzo6} and condition \eqref{eq:omega condition} again, 
\begin{align}\label{marzo9}
\frac{1}{\omega(r)}&\bigg(\dashint_{ \Omega \cap B_{r}(y)}\abs{\bfA(\nabla\bfu)-\bfA_{s} }^{\min{\set{2,p'}}}
\dz\bigg)^{\frac{1}{\min{\set{2,p'}}}}
\\ \nonumber & \leq 
\frac{c}{\omega(r)}\bigg(\dashint_{ \Omega \cap B_{4r}(x_y)}\abs{\bfA(\nabla\bfu)-\bfA_{s} }^{\min{\set{2,p'}}}
\dz\bigg)^{\frac{1}{\min{\set{2,p'}}}}
\\ \nonumber & 
\leq \frac{c'}{\omega(R)}\dashint_{\Omega \cap B_{R/2}(x_y)}\abs{\bfA(\nabla \bfu)}\, \dz
+ c'M^{\sharp,p'q}_{\omega,\Omega,R}(\bfF)(x_y)
\\ \nonumber & \
\leq \frac{c''}{\omega(R)}\dashint_{\Omega \cap B_{R}(x)}\abs{\bfA(\nabla \bfu)}\, \dz
+ c'M^{\sharp,p'q}_{\omega,\Omega,2R}(\bfF)(x)
\end{align}
for some constants $c, c', c''$.
Inequalty \eqref{a8} follows from \eqref{marzo8} and \eqref{marzo9}, via the very definition of sharp maximal function \eqref{Msharp} and property \eqref{meanq}. Inequality  \eqref{a15}  is a consequence of inequality \eqref{a8} and of the definition of Campanato seminorm.
\end{proof}

We are now in a position to accomplish the proof of Theorem \ref{thm:campanato}.

\begin{proof}[Proof of Theorem~\ref{thm:campanato}]

\color{black}
A basic energy estimate  obtained by choosing $u$ as a test function in equation \eqref{weaksol}, and 
 \cite[Theorem~5.23]{DieRuzSch08} tell us that  %\todo{{\color{black} The exponent of $|\nabla u|$ is wrong}}
\begin{align}
\label{eq:a-priori}
\frac{1}{\omega(\diam(\Omega))}\bigg(\int_{\Omega}\bfA(\abs{\nabla \bfu})^{p'}\, \dx\bigg)^{ \frac 1{p'}}\leq \frac{1}{\omega(\diam(\Omega))}\bigg(\int_{\Omega}\abs{\bfF - \mean{\bfF}_\Omega}^{p'}\, \dx\bigg)^{\frac 1{p'}}\leq c\norm{\bfF}_{\Lcal^{\omega (\cdot)}(\Omega)}
\end{align}
for some constant $c$. Let $R_0$ be the radius provided by Proposition \ref{pro:wahnsinn}.
%Since $\partial \Omega \in W^{1}\mathcal L^{\sigma (\cdot)}\cap C^{0,1}$, there exists $R\in (0,1)$ such that the assumptions of Proposition \ref{pro:wahnsinn} are satisfied for every $x \in \partial \Omega$.
Trivially,
\begin{align}\label{a16}
\norm{\bfA(\nabla \bfu)}_{\Lcal^{\omega (\cdot)}(\Omega)} \leq   \max\bigg\{&\sup_{x\in \Omega} \sup _{r \in (0, R_0/16)}\frac{1}{\omega(r)}\dashint_{ \Omega \cap B_{r}(x)}\abs{\bfA(\nabla\bfu)-\mean{\bfA(\nabla\bfu)}_{ \Omega \cap B_{r}(x)} }
\dy, 
\\ \nonumber & 
\sup_{x\in \Omega} \sup _{r \geq R_0/16}\frac{1}{\omega(r)}\dashint_{\Omega \cap  B_{r}(x)}\abs{\bfA(\nabla\bfu)-\mean{\bfA(\nabla\bfu)}_{ \Omega \cap B_{r}(x)}}
\dy\bigg\}\,.
\end{align}
If $r\geq R_0/16$, then, by \eqref{eq:omega condition}, H\"older's inequality and  \eqref{eq:a-priori}, there exist constants $c$ and $c'$ such
that
\begin{align}\label{a17}
\frac{1}{\omega(r)}\dashint_{\Omega \cap B_r(x)}\abs{\bfA(\nabla {\bfu})-\mean{\bfA(\nabla\bfu)}_{ \Omega \cap B_{r}(x)}}\dy
 \leq  c
%\frac{c}{\omega(R)}
\dashint_{\Omega}\abs{\bfA(\nabla \bfu)}\dy
%\leq c\bigg(\dashint_{\Omega}\abs{\bfA(\nabla \bfu)}^{p'}\dy\bigg)^\frac1{p'} 
%\\ \nonumber & 
%\leq c'\bigg(\dashint_{\Omega}\abs{\bfF- \mean{\bfF}}^{p'}\dy\bigg)^\frac1{p'}
%\\ \nonumber & 
\leq c' \norm{\bfF}_{\Lcal^{\omega (\cdot)}(\Omega)}.
\end{align}
If $r \in (0,R_0/16)$, then, by  \eqref{a15}, H\"older's inequality and  \eqref{eq:a-priori}, there exist constants $c$ and $c'$ such that
\begin{align}\label{a18}
\frac{1}{\omega(r)}\dashint_{ \Omega \cap B_{r}(x)}& \abs{\bfA(\nabla\bfu)-\mean{\bfA(\nabla\bfu)}_{ \Omega \cap B_{r}(x)}}
\dy 
 \\ 
\nonumber & 
\leq c\norm{\bfF}_{\Lcal^{\omega (\cdot)}(\Omega)}+c\dashint_{\Omega}\abs{\bfA(\nabla \bfu)}\, \dy
\leq c'\norm{\bfF}_{\Lcal^{\omega (\cdot)}(\Omega)}\,.
\end{align}
Inequality \eqref{feb1} follows from \eqref{a16}--\eqref{a18}.
\color{black}
\end{proof}

\medskip
{
\begin{proof}[Proof of Corollary \ref{bmovmo}] The assertion about the case when $\bfF\in \setBMO(\Omega)$ is a straighforward consequence of Theorem \ref{thm:campanato}, since the integral on the left-hand side of equation \eqref{balance} agrees with $\log \frac 1r$ if $\omega (r)=1$.
\\ Assume next that
 $\bfF\in \setVMO(\Omega)$, and define the function $\varrho : [0, \infty) \to [0, \infty)$ as 
$$\varrho (r) =
\sup_{
\begin{tiny}
 \begin{array}{c}{
    x \in \Omega} \\
0 < s \leq r
 \end{array}
  \end{tiny}
}
\, \dashint _{\Omega\cap B_s(x)} |\bfF - \mean{\bfF}_{\Omega \cap B_s(x)}|\,
  \dy
 \quad \hbox{if $r>0$,}$$
and $\varrho (0)=0$
Then $\varrho$ is a non-decreasing bounded function, such that $\lim _{r\to 0^+} \varrho (r)=0$. Now, let  $\omega : [0, \infty) \to [0, \infty)$ be the function given by 
\begin{equation}\label{dec35}
\omega (r) = r^{\beta _0} \sup_{s \geq r} \Big[s^{-\beta _0}  \sup_{0<\tau\leq s}\max \big\{\varrho (\tau),   \sqrt{\sigma (\tau) \log 1/\tau}\big\}\Big] \quad \hbox{if $r\in (0,1]$,}
\end{equation}
$\omega (0)=0$, and $\omega (r) =\omega (1)$ if $r>1$. It is easily verified that $\omega$ is a continuous parameter function fulfilling condition \eqref{eq:omega condition}. The function $\omega$ is also non-decreasing. This follows from an argument analogous to that employed in the proof of assertion \eqref{july33} below. 
Moreover, 
\begin{equation}\label{dec33}
\sqrt{\sigma(r)\log(1/r)}\leq \omega(r) \quad \hbox{for $r \in (0,1)$,}
\end{equation}
and hence
\begin{equation}\label{dec34}
\frac{\sigma(r)}{\omega(r)}\int_r^1\frac{\omega(\tau)}{\tau}\, {\rm d}\tau\leq \frac{\omega (1)\sigma(r)\log(1/r)}{\omega(r)}\ \leq   \omega (1) \sqrt{\sigma(r)\log(1/r)}  \quad \hbox{for $r \in (0,1)$.}
\end{equation}
Thus, by assumption \eqref{vanishlog}, condition \eqref{balance} is fulfilled with $\omega$ given by \eqref{dec35}. An application of Theorem \ref{thm:campanato} tells us that $\abs{\nabla \bfu}^{p-2}\nabla \bfu \in \Lcal^\omega(\Omega)$ , whence, in particular, $\abs{\nabla \bfu}^{p-2}\nabla \bfu \in \setVMO(\Omega)$.
%
% that there is a monotone continuously increasing parameter function $\omega$ with $\omega(0)=0$, such that $\bfF\in \Lcal^\omega(\Omega)$. With no loss of generality, we may assume that
%\[
%\sqrt{\sigma(r)\log(1/r)}\leq \omega(r)\leq C.
%\]
%Now, since
%\[
%\frac{\sigma(r)}{\omega(r)}\int_r^1\frac{\omega(r)}{r}\leq \frac{C\sigma(r)\log(1/r)}{\omega(r)}\leq C\sqrt{\sigma(r)\log(1/r)}\leq \delta,
%\]
%for $r$ sufficiently close to $0$, we find that $\abs{\nabla \bfu}^{p-2}\nabla \bfu \in \Lcal^\omega(\Omega)$ and the result is proved.
%
\end{proof}}

\section{Proof of Theorem~\ref{continuity}}
%\color{blue}  THE CONTENT AND PROOF OF THIS SUBSECTION ARE NOT CLEAR TO ME
\color{black}

A critical result in view of a proof of Theorem \ref{continuity} is an analogue of the pointwise estimate for the gradient of solutions to problem \eqref{eq:sysA} 
established in Proposition \ref{pro:wahnsinn}, but under condition \eqref{dini} and the a priori assumption that the gradient is bounded. The punctum of this result, contained in the next proposition, is that the mere finiteness of the supremum on the left-hand side of equation \eqref{balance} suffices under such an assumption.

\begin{proposition} \label{pro:wahnsinn1}
%Let $\bfu$ be a solution to \eqref{eq:sysA}. Let $x$ be a boundary point, such that local coordinates $\Psi$ exists (in the sense of section \ref{sec:boundary}) in $B_R(x)\cap\Omega$.
%Let $\omega$ satisfying \eqref{eq:omega condition}. Further let $q\in (1,q_0]$ for $q_0=q_0(\lambda,\Lambda,p,n,N)$ fixed via Corollary~\ref{cor:VPL}.
%%with $\bfT\equiv\nabla \Psi$. 
%If $\bfF\in L^{p'q}(\Omega)$ and $\bfJ=\nabla \Psi$ holds $\bfJ\in \Lcal_\omega(B_R^+(\bfx))$,
% then there exist constants, $c=c(n,p,N,\beta,c_\omega,\lambda,\Lambda,q)$ and $\epsilon_0=\epsilon_0(n,p,N,\beta,c_\omega,\lambda,\Lambda,q)$, such
%that \db{where $\epsilon_0$?}
 Let $\Omega$, $x, R, \sigma, \omega, q, \bfF$ and $\bfu$ be as in  Proposition~\ref{pro:wahnsinn}, save that  assumption \eqref{a7} is replaced by
\begin{equation}\label{july8}
\sup_{r\in (0,R)}\frac{\sigma(r)}{\omega(r)}\leq \frac{C}{\omega(R)}
\end{equation}
for some positive constant $C$.
Assume in addition, that $\nabla \bfu\in L^\infty(\Omega)$. Then there exist positive constants $c$ and $R_0$, depending on $n,N,p, \omega, q, \Omega$,  such that
\begin{align}\label{july9}
M^{\sharp,\min\set{2, p'}}_{\omega,\Omega,R/8}(|\nabla \bfu|^{p-2}\nabla \bfu)(x)\leq cM^{\sharp,p'q}_{\omega,\Omega,2R}(\bfF)(x)+\frac{c}{\omega(R)}\norm{\nabla\bfu}_{L^\infty(\Omega \cap B_{2R}(x))}^{p-1}\,.%+\mean{\abs{\bfF}}_{(B_{2R}(x)\cap\Omega)}\Big).
\end{align}
for every $x \in \partial \Omega$ and $R\in (0, R_0)$.
Hence,  %supposing additionally the assumptions of Theorem \ref{thm:7.4} we have
\begin{align}\label{a15'}
\norm{|\nabla \bfu|^{p-2}\nabla \bfu}_{\Lcal^{\omega (\cdot)}(\Omega \cap {B_{R/16}(x)})}\leq c\norm{\bfF}_{\Lcal^{\omega (\cdot)}(\Omega \cap B_{2R}(x))}+\frac{c}{\omega(R)}\|\nabla \bfu\|_{L^\infty(\Omega \cap B_{2R}(x))}^{p-1}
\end{align}
for every $x \in \partial \Omega$ and $R\in (0, R_0)$.
\iffalse

Assume in addition, that $\nabla u\in L^\infty(\Omega \cap B_{2R}(x))$. Then there exists a constant $c=c(n,p,N,\beta,c_\omega,\lambda,\Lambda,q)$ such that
\begin{align}\label{july9}
M^{\sharp,\min\set{2, p'}}_{\omega,\Omega,R/8}(|\nabla \bfu|^{p-2}\nabla \bfu)(x)\leq cM^{\sharp,p'q}_{\omega,\Omega,2R}(\bfF)(x)+\frac{c}{\omega(R)}\norm{\nabla\bfu}_{L^\infty(\Omega \cap B_{2R}(x))}^{p-1}\,.%+\mean{\abs{\bfF}}_{(B_{2R}(x)\cap\Omega)}\Big).
\end{align}
Hence,  %supposing additionally the assumptions of Theorem \ref{thm:7.4} we have
\begin{align}\label{a15'}
\norm{|\nabla \bfu|^{p-2}\nabla \bfu}_{\Lcal^{\omega (\cdot)}(\Omega \cap {\color{blue} B_{R/16}(x)})}\leq c\norm{\bfF}_{\Lcal^{\omega (\cdot)}(\Omega \cap B_{2R}(x))}+\frac{c}{\omega(R)}\|\nabla \bfu\|_{L^\infty(\Omega \cap B_{2R}(x))}^{p-1}.
\end{align}
\fi
\end{proposition}
\begin{proof}
We employ the notations of Proposition~\ref{pro:wahnsinn}. 
%We are now in the situation of a bounded gradient of the solutions and such the proof is actually simpler.
%In accordance with the Proposition~\ref{pro:wahnsinn} and the John-Nirenberg estimate we define for $\bfJ=\nabla \Psi$  \db{already defined!}
%As before, we set
%\begin{align}
%\label{eq:Jw}
%\sigma(\rho):=\bigg(\dashint_{ B_\rho^+}\abs{\bfJ-\mean{\bfJ}_{B_\rho^+}}^{p q'}\dx\bigg)^\frac{1}{q'p'}
%%+\bigg(\dashint_{B_\rho^+}\abs{\bfJ^{-1}-\mean{\bfJ^{-1}}_{B_\rho^+}}^{\min{\set{2,p'}}}\dx\bigg)^\frac{1}{\min{\set{2,p'}}}.
%\end{align}
%Since we assume, that $\partial\Omega\in \Lcal~\ref{pro:wahnsinn}, w
%just assume \db{why possible?}
%\db{why $\Lambda$?}
%We find by the assumptions, that
%\[
%\sup_{r\in (0,R)}\frac{\sigma(r)}{\omega(r)}\leq \frac{c}{\omega(R)}.
%\]
The proof of inequality \eqref{july9} proceeds along the same lines as that of inequality \eqref{a8} of Proposition~\ref{pro:wahnsinn}. The situation is now actually simpler, owing to the boundedness assumption on the gradient. In fact, 
inequality  \eqref{eq:0607} and  assumption \eqref{july8} immediately imply that
%since $\sigma=\omega$ and $\psi$ is bounded, that
\begin{align*}
\begin{aligned}
&\frac{1}{\omega(\theta s)}\bigg(\dashint_{ \Omega \cap B_{\theta s}(x)}\abs{\bfA(\nabla\bfu)-\bfA_{\theta s} }^{\min{\set{2,p'}}}
\dx\bigg)^{\frac{1}{\min{\set{2,p'}}}}
\\
%&\quad \leq
%\frac{c\sigma(s)\psi(s)}{\omega(s)}\sup_{\rho\in (s,R)}\frac{1}{\omega(\rho)}\bigg(\dashint_{B_{\rho}\cap \Omega}\abs{\bfA(\nabla{\bfu})-\bfA_{\rho} }^{\min{\set{2,p'}}}\dz\bigg)^\frac{1}{\min{\set{2,p'}}}
%\\
&\leq \,{\frac{c}{\omega (R)}}\norm{ \bfA(\nabla\bfu)}_{L^\infty(\Omega \cap B_{s}(x))}
+\frac{c}{\omega(s)}\bigg(\dashint_{\Omega \cap B_{s}(x)}\abs{\bfF-\bfF_0}^{p'q}\dx \bigg)^\frac1{p'q} 
\\
&\qquad + \frac{1}{2\omega(s)}\bigg(\dashint_{\Omega \cap B_{s}(x)}\abs{\bfA(\nabla{\bfu})-\bfA_{s} }^{\min{\set{2,p'}}}\dx\bigg)^\frac{1}{\min{\set{2,p'}}} \qquad \quad \hbox{for $s \in (0, \tfrac R2)$.}
%&\quad \leq
%\frac{c\sigma(r)\psi(r)}{\omega(r)}\sup_{\rho\in (r,R)}\frac{1}{\omega(\rho)}\bigg(\dashint_{B_{\rho}\cap \Omega}\abs{\bfA(\nabla{\bfu})-\bfA_{\rho} }^{\min{\set{2,p'}}}\dz\bigg)^\frac{1}{\min{\set{2,p'}}}
%\\
%&\qquad +\frac{c}{\omega(2\Lambda R)}\norm{ \bfA(\nabla\bfu)}_{L^\infty(B_{2\Lambda R}(x)\cap\Omega)}
%+\frac{c}{\omega(r)}\bigg(\dashint_{B_{r}\cap\Omega}\abs{\bfF-\bfF_0}^{p'q}\dx \bigg)^\frac1{p'q} 
%\\
%&\qquad + \frac{1}{2\omega(r)}\bigg(\dashint_{B_{r}\cap \Omega}\abs{\bfA(\nabla{\bfu})-\bfA_{r} }^{\min{\set{2,p'}}}\dz\bigg)^\frac{1}{\min{\set{2,p'}}}.
\end{aligned}
\end{align*}
This inequality replaces \eqref{a10}. The rest of the argument in the proof of inequalities \eqref{july9} and  \eqref{a15'}   is the same as that of inequalities  \eqref{a8} and  \eqref{a15} in Proposition~\ref{pro:wahnsinn}.  
%\\ Inequality  \eqref{a15'} follows from inequality \eqref{july9} in the same way as inequality \eqref{a15} follows from inequality  \eqref{a8}.
%
%
%We take the supremum over $r\leq R$ and absorb the last term in the left-hand side.
%The result follows. The estimate of \eqref{a15'} follows analogous to the proof of Theorem~\ref{thm:7.4}.
\end{proof}

\begin{proof}[Proof of Theorem ~\ref{continuity}]
We begin by showing that $|\nabla \bfu| \in L^\infty (\Omega)$. To this purpose, observe that, owing to assumption \eqref{dini},
%since we are assuming that 
%$$\int_0 \frac{\omega(r)}r\, dr <\infty\,$$
there exists an increasing function $\eta : (0, 1) \to (0, \infty)$ such that $\lim _{r \to 0^+} \eta (r) =0$, and still
\begin{equation}\label{july30}
\int_0 \frac{\omega(r)}{\eta (r)r}\, \mathrm{d}r <\infty\,.
\end{equation}
For instance, one can choose $\eta (r) = 	{\underline \omega} (r)^{1-\gamma}$ for some $\gamma \in (0,1)$.
Next, define the non-decreasing function $\omega_1 : (0, 1) \to (0,\infty)$ as 
$$\omega_1 (r) = \inf _{s\geq r} \frac{\omega (s)}{\eta (s)} \quad \hbox{for $r> 0$.}$$
One has that
\begin{equation}\label{july31}
c \omega (r) \leq \omega_1 (r) \leq \frac{\omega (r)}{\eta (r)}\quad \hbox{for $r >0$,}
\end{equation}
for some positive constant $c$. In particular, the second inequality in \eqref{july31} and condition \eqref{july30} ensure that condition \eqref{dini} is still satisfied with $\omega$ replaced by $\omega _1$, namely
\begin{equation}\label{july37}
\int _0\frac{\omega_1(r)}{r}\, \mathrm{d}r < \infty\,.
\end{equation}
Moreover, the first inequality in \eqref{july31} tells us that $\partial \Omega \in W^1 \mathcal L^{\omega_1(\cdot)}$.
Now, consider the function $\eta _1: (1, \infty) \to [0, \infty)$ given by
$$\eta _1(r) = \frac{\omega(r)}{\omega_1 (r)} \quad \hbox{for $r > 0$,}$$
and observe that
$$\omega _1(r) = \frac {\omega (r)}{\eta_1(r)}  \quad \hbox{for $r > 0$,}$$
and 
\begin{equation}\label{july32}
\eta _1(r) \geq \eta(r)  \quad \hbox{for $r > 0$.}
\end{equation}
Also, %\todo{Details to be checked and added} 
we claim that 
\begin{equation}\label{july33}
\hbox{the function $\eta_1$ is non-decreasing.}
\end{equation} 
To verify this claim, note that there exists a  (possibly empty) family $\{(a_i, b_i)\}$, with $i \in I \subset \setN$, of disjoint intervals in $(0,1)$, with $\lim _{i\to \infty} a_i  =0$ if $I$ is infinite, such that, if $r \in (0,1)$, then either $\omega_1(r) = \frac{\omega (r)}{\eta (r)}$, or $r \in [a_i, b_i]$ for some $i \in I$ and $\omega _1(r) = \omega _1(s)= \frac{\omega (a_i)}{\eta (a_i)}$ for every $s \in  [a_i, b_i]$. 
Now, let $r, s \in (0,1)$ be such that $r<s$. If $\omega_1(r) = \frac{\omega (r)}{\eta (r)}$ and $\omega_1(s) = \frac{\omega (s)}{\eta (s)}$, then
$$\eta _1 (r) = \eta (r) \leq \eta (s) = \eta _1(s)\,,$$
since $\eta$ is non-decreasing. If $r, s \in [a_i, b_i]$ for some $i\in I$, then 
$$\eta _1 (r) = \frac{\omega (r)}{\omega _1(r)} = \frac{\omega (r)}{\omega _1(s)}\leq  \frac{\omega (s)}{\omega _1(s)}= \eta _1(s)\,,$$
since $\omega$ is non-decreasing. If $r \in [a_i, b_i]$ for some $i\in I$, $s \geq b_i$ and 
$\omega_1(s) = \frac{\omega (s)}{\eta (s)}$, then 
$$\eta_1(r) = \frac{\omega (r)}{\omega _1(r)}  = \frac{\omega (r)}{\omega _1(b_i)} \leq \frac{\omega (b_i)}{\omega _1(b_i)} = \eta (b_i) \leq \eta (s) \leq \eta_1(s)\,.$$
Finally, if   $s \in  [a_i, b_i]$ for some $i \in I$,  $r \leq a_i$ and $\omega_1(r) = \frac{\omega (r)}{\eta (r)}$, then
$$\eta _1 (r) = \eta (r) \leq \eta (a_i) \leq \eta _1(a_i) = \frac{\omega (a_i)}{\omega _1(a_i)} =  \frac{\omega (a_i)}{\omega _1(s)} \leq  \frac{\omega (s)}{\omega _1(s)}= \eta _1(s)\,.$$
Property \eqref{july33} is thus established. This
property  ensures that the function $\omega_1$ fulfills assumtpion \eqref{eq:omega condition} with the same   exponent $\beta$  as $\omega$, since
\begin{equation}\label{july34}
\frac{\omega _1(r)}{r^\beta} = \frac{\omega (r)}{r^\beta} \frac 1{\eta_1(r)}  \quad \hbox{for $r > 0$.}
\end{equation}
Next, we claim that
\begin{equation}\label{july35}
\lim_{r \to 0^+}\eta _1(r) =0\,.
\end{equation}
To verify this claim, assume, by contradiction, that \eqref{july35} fails. Thus, there exists a positive number $c$ such that
\begin{equation}\label{july36}
\omega(r) \geq c \inf _{s\geq r} \frac{\omega (s)}{\eta (s)} \quad \hbox{for $r \in (0,1)$.}
\end{equation}
As a consequence of the properties of the family $\{(a_i, b_i)\}$ introduced above, 
there exists a sequence $\{r_k\}$, with  $\lim _{k\to \infty} r_k  =0$, such that
$$\omega_1(r_k) =  \inf _{s\geq r_k} \frac{\omega (s)}{\eta (s)} = \frac{\omega (r_k)}{\eta (r_k)}$$
for $k \in \mathbb N$. From \eqref{july36} with $r=r_k$ we infer that
$$\eta (r_k) \geq c \quad \hbox{for $k \in \mathbb N$.}$$
This contradicts the fact that $\lim _{r\to 0^+}\eta (r) =0$.
\\ By properties \eqref{july35} and \eqref{july37},
\begin{equation}\label{july38}
\lim _{r \to 0^+} \frac{\omega (r)}{\omega _1(r)} \int _r^1\frac{\omega_1(s)}{s}\,\mathrm{d}s \leq \lim _{r \to 0^+} \eta_1(r) \int _0^1\frac{\omega_1(s)}{s}\,\mathrm{d}s =0\,.
\end{equation}
This fact ensures that assumption \eqref{balance} is fulfilled with $\sigma$ replaced by $\omega$, and $\omega$ replaced by $\omega_1$. This property, combined with the properties of $\omega _1$ established above,  enables us to apply Theorem \ref{thm:campanato} with the same replacements for $\sigma$ and $\omega$. In particular, we infer that $|\nabla \bfu|^{p-2}\nabla \bfu \in \Lcal^{\omega_1(\cdot)}({\Omega})$, whence, by condition \eqref{july37} and inclusion \eqref{spanne1} with $\omega$ replaced by $\omega_1$, we have  that $|\nabla \bfu| \in L^\infty (\Omega)$. 
\\ We are now in a position 
 to apply Proposition~\ref{pro:wahnsinn1}.  Notice that condition \eqref{july8} is satisfied, with $\sigma = \omega$, owing to assumption \eqref{dini}. To be precise, from \eqref{dini} one can deduce  that condition  \eqref{july8} holds with $\sigma  = \omega$ and the quantity $\int_0^1\tfrac{\omega(\rho)}\rho\, {\rm d} \rho$ on the right-hand side.
%
%a constant (independent of $R$) right-hand side. This yields inequalities \eqref{july9} and \eqref{a15'} with constants independent of $R$ on the right-hand side. 
Starting from inequality \eqref{a15'}, and arguing as in the proof of Theorem
\ref{thm:campanato} yields the conclusion.
\end{proof}

%\begin{theorem}
%\label{thm:pw}
% Let $q\in (1,q_0]$ for a fixed $q_0=q_0(\lambda,\Lambda,p,n,N)$, by Proposition~\ref{cor:VPL}
% and $\bfF\in L^{p'q}(B)$. Assume that the weight $\omega$ satisfies \eqref{eq:omega condition}.
%  Assume that $\bfu$ is a solution of
%  \eqref{eq:sysA} coupled with either \eqref{eq:dirichlet} with $\partial \Omega \in W^1\Lcal_\sigma(\Omega)\cap C^{0,1}$ .
% Then, there is an $\epsilon_0\in(0,\infty)$ and a constant $c$ depending on $n,N,p,q,\lambda,\Lambda, c_\omega,\beta$, such that if $\sigma$ satisfies
% \[
% \sup_{r\in (0,R)}\frac{\sigma(r)\psi(r)}{\omega(r)}\leq \epsilon_0,
% \]
%  then 
%\begin{align}\label{eq:omegadecay}
%M^{\sharp,\min\set{p',2}}_{\omega,\Omega, R} \bfA(\nabla\bfu)(\bfx)\leq c M^{\sharp,p'q}_{\omega,\Omega, R} \bfF(\bfx)+ \frac{c}{\omega(R)}\dashint_{B_R\cap \Omega(\bfx)}\abs{\bfA(\nabla \bfu)}+\abs{\bfF}\dx
%\end{align}
%for every boundary point $\bfx$.
%\end{theorem}

   \section{Sharpness of results}
%
%In the present  section we demonstrate the sharpness of Theorem \ref{thm:campanato}, our main result. This can be shown by considering the simplest situation when $n=2$, $N=1$ and $p=2$, namely the scalar Dirichlet problem for the Poisson equation in the plane
%\begin{equation}\label{poisson}
%\begin{cases}
%- \Delta u = f & \quad \hbox{in $\Omega$}
%\\
%u =0 & \quad \hbox{on $\partial \Omega$.}
%\end{cases}
%\end{equation}
%
%\begin{theorem}\label{sharpness}
%{\color{blue} Let $\omega \in  C^1(0, \infty)$ be any concave parameter function 
%satisfying condition \eqref{eq:omega condition},  such that   $\lim_{r\to 0^+}\frac{r\omega '(r)}{\omega (r)}$ exists, } and
%\begin{equation}\label{divint}
%\int_0\frac{\omega(r)}r\, {\rm d}r=\infty.
%\end{equation}
%Then there exist a parameter function $\sigma$, a bounded open set $\Omega \subset \setR^2$, and a function $f \in C^\infty (\Omega)$ such that, if $u$ is the solution to the Dirichlet problem \eqref{poisson},
%then:
%\begin{equation}\label{Nov30}
%\partial \Omega \in W^1\Lcal^{\sigma (\cdot)} \cap C^{0,1};
%\end{equation}
%\begin{equation}\label{Nov31}
% \sup _{r \in (0,1)}\frac{\sigma(r)}{\omega(r)}\int^1_r\frac{\omega(\rho)}{\rho} \,d\rho< \infty;
%\end{equation}
%\begin{equation}\label{Nov32}
%  \nabla u \notin \mathcal L ^{\omega (\cdot)} (\Omega ).
%\end{equation}
%\end{theorem}

Our proof  of Theorem \ref{sharpness} is 
%inspired by an example from \cite[Section
%14.6.1]{MazyaShap}, and is 
based on precise information on conformal transformations of certain planar domains established in \cite{Wa}, coupled with the  embedding theorem contained in the following proposition. In its statement, $W^{1}_0\mathcal L ^{\omega (\cdot)}(\Omega)$ denotes the Sobolev type space of those functions in $\Omega$ whose continuation by $0$ outside $\Omega$ belongs to  $W^{1}\mathcal L ^{\omega (\cdot)}(\setR^n)$.

\begin{proposition}\label{embedding}
Let $\Omega$ be a bounded open set in $\mathbb R^n$, let $\omega : (0,1) \to [0, \infty)$ be a parameter  function fulfilling condition \eqref{divint}, and let $\overline \omega$ be the function associated with $\omega$ as in \eqref{omegabar}.
%$$\overline \omega (r) = \int _r^1\frac{\omega (\rho)}{\rho}\, d\rho\,.$$
%Assume that $\lim_{r \to 0^+} \overline \omega (r) = \infty$, and 
Let $\upsilon  : (0,1) \to [0, \infty)$ be the function defined by 
\begin{align}\label{sep17}
\upsilon (r) = r \,\overline \omega (r) \quad \hbox{for $r \in (0,1)$.}
\end{align}
Then 
\begin{equation}\label{sep18}
W^{1}_0\mathcal L ^{\omega (\cdot)}(\Omega) \to C^{0, \upsilon (\cdot)}(\Omega).
\end{equation}
%where the arrow $\lq\lq \to "$ stands for continuous embedding.
%
%If $u \in W^{1}_0\mathcal L ^{\omega (\cdot)}(\Omega)$, then 
%\begin{equation}\label{sep18}
%u \in C^{0, \upsilon (\cdot)}(\Omega).
%\end{equation}
\end{proposition}

\medskip
\begin{proof}
Assume, without loss of generality, that $|\Omega|=1$.
Define the (increasing) function $\varsigma : (0,1) \to [0, \infty)$ as 
\begin{align}\label{sep16}
\varsigma (r) = \frac 1{\overline \omega (r^{\frac 1n})} \quad \hbox{for $r \in (0,1)$,}
\end{align}
and  denote by $M^{\varsigma (\cdot)}(\Omega)$ the Marcinkiewicz space associated with $\varsigma$, and consisting of those measurable functions $\bff$ on $ \Omega $ such that
\begin{equation}\label{sep15}
\sup _{0<r<1} \varsigma (r) \bff^*(r) < \infty\,.
\end{equation}
Here, $\bff^*$ denotes the decreasing rearrangement of $\bff$, defined on $(0, 1)$ as
$$\bff^*(r) = \inf\{t\geq 0: |\{|\bff|> t\}|\leq r\} \qquad \hbox{for $r\in (0,1)$,}$$
where $|E|$ denotes Lebesgue measure of a set $E \subset \setR^n$.
We claim that the supremum in \eqref{sep15}
  is equivalent, up to multiplicative constants independent of $\bff$, to the rearrangement-invariant norm --  in the sense of Luxemburg (see \cite{BS}) --  defined as 
\begin{equation}\label{sep19}
\|\bff\|_{M^{\varsigma (\cdot)}(\Omega)} = \sup _{0<r<1} \varsigma (r) \bff^{**}(r)\,,
\end{equation}
where we have set $\bff^{**}(r) = \tfrac 1r \int_0^r\bff^*(s)\, {\rm d}s$ for $r \in (0,1)$.
Thus, $M^{\varsigma (\cdot)}(\Omega)$ is a rearrangement-invariant space equipped with this norm. Since $\bff^* \leq \bff^{**}$, this equivalence will follow if we show that
\begin{align}\label{sep20}
\bigg\|\frac{\varsigma (r)}r \int _0^r g(\rho)\, \mathrm{d}\rho  \bigg\|_{L^\infty(0,1)} \leq C \|\varsigma (r) g(r)\|_{L^\infty(0,1)}
\end{align}
for some constant $C$ and for every measurable function $g : (0,1) \to [0, \infty)$. A characterization of weighted Hardy type inequalities tells us that inequality \eqref{sep20} holds if (and only if)
\begin{align}\label{sep21}
\sup _{0<s<1} \bigg\|\frac{\varsigma (r)}r   \bigg\|_{L^\infty(s,1)}   \bigg\|\frac 1{\varsigma (r)} \bigg\|_{L^1(0,s)}\leq C
\end{align}
for some constant $C$ -- see e.g. \cite[Theorem 1.3.2/2]{Mazya_book}. Since the function $\varsigma$ is increasing, it suffices to verify inequality \eqref{sep21} with the supremum extended to values of $s$ in a sufficiently small right neighbourhood of  $0$. Given $a>1$, one has that
\begin{align}\label{sep21'}\big(r \,\overline\omega (r^{\frac 1n})\big)' & = \int _{r^{\frac 1n}}^1\frac{\omega (\rho)}\rho\, \mathrm{d}\rho - \frac{\omega (r^{\frac 1n})}n \geq  \int _{r^{\frac 1n}}^{ar^{\frac 1n}}\frac{\omega (\rho)}\rho\,\mathrm{d} \rho - \frac{\omega (r^{\frac 1n})}n \\  \nonumber & \geq \omega (r^{\frac 1n}) \int _{r^{\frac 1n}}^{ar^{\frac 1n}}\frac{\mathrm{d}\rho}\rho  - \frac{\omega (r^{\frac 1n})}n = \omega (r^{\frac 1n}) (\log a - \tfrac 1n) >0\,,
\end{align}
provided that $a>e^{\frac 1n}$, and $r\in (0, \tfrac{1}{{a^{n}}})$. Hence, the function $\frac{\varsigma (r)}r $ is decreasing in $(0, \tfrac{1}{{a^{n}}})$. Therefore, inequality \eqref{sep21} will follow if we show that 
\begin{align}\label{sep22}
\frac 1r \int _0^r \frac {\mathrm{d}\rho}{\varsigma (\rho)} \leq  \frac {C}{\varsigma (r)} 
\end{align}
for some constant $C$ and for small $r$. An application of Fubini's theorem tells us that
\begin{align}\label{sep23}
\frac 1r \int _0^r \frac {\mathrm{d}\rho}{\varsigma (\rho)} = \frac 1r \int _0^{r^{\frac 1n}}\omega (\rho) \rho ^{n-1}\, \mathrm{d}\rho + \int _{r^{\frac 1n}}^1 \frac{\omega (\rho)}\rho \, d\rho\, \quad \hbox{for $r \in (0,1)$.}
\end{align}
The second addend on the right-hand side of \eqref{sep23} agrees with $ \frac {1}{\varsigma (r)}$. On the other hand,
$$\frac 1r \int _0^{r^{\frac 1n}}\omega (\rho) \rho ^{n-1}\, \mathrm{d}\rho \leq \frac 1{r^{\frac 1n}} \int _0^{r^{\frac 1n}}\omega (\rho) \, \mathrm{d}\rho \quad \hbox{for $r \in (0,1)$.}$$
Since the right-hand side of this inequality is bounded, whereas $\frac {1}{\varsigma (r)}$ diverges to $\infty$ as $r \to 0^+$, inequality \eqref{sep22} follows via \eqref{sep23}. 
\\ We have thus established that $M^{\varsigma (\cdot)}(\Omega)$ is a rearrangement-invariant space. Now, recall that the  the fundamental function   
$\varphi _{M^{\varsigma (\cdot)}}: [0, 1] \to [0, \infty)$ of $M^{\varsigma (\cdot)}(\Omega)$ is defined as 
$$\varphi _{M^{\varsigma (\cdot)}} (r) =\|\chi _{E}\|_{M^{\varsigma (\cdot)}(\Omega)} \quad \hbox{for $r \in (0,1)$,}$$
where $E$ is any measurable set contained in $\Omega$ and such that $|E|=r$, and $\chi _{E}$ stands for its characteristic function. 
Since, by \eqref{sep21'}, the function $\varsigma$ is quasi-concave, \cite[Chapter 2, Proposition 5.8]{BS} tells us that
\begin{equation}\label{sep24}
\varphi _{M^{\varsigma (\cdot)}} (r) = \varsigma (r) \quad \hbox{for $r \in (0,1)$.}
\end{equation}
\\ From  \cite[Theorem 1]{Spa65}  one infers  that 
\begin{align}\label{sep25}
W^1_0\mathcal L^{\omega (\cdot)}(\Omega) \to W^1_0 M^{\varsigma (\cdot)}(\Omega).
\end{align}
%
%\\ If $u \in W^1_0\mathcal L^{\omega (\cdot)}(\Omega)$, from  \cite[Theorem 1]{Spa65}  one infers  that 
%\begin{align}\label{sep25}
%u \in W^1_0 M^{\varsigma (\cdot)}(\Omega).
%\end{align}
%By the equivalence of the supremum in \eqref{sep15} and of $\|u\|_{M^{\varsigma (\cdot)}(\Omega)}$,  and the monotonicity of $\varsigma$, one has  that
%\begin{equation}\label{sep24}
%\varphi _{M^{\varsigma (\cdot)}} (r) \approx \varsigma (r)\,.
%\end{equation}
The associate space of $M^{\varsigma (\cdot)}(\Omega)$ is the Lorentz space $\Lambda (\Omega)$ equipped with the norm defined as 
$$\|\bff\|_{\Lambda (\Omega)} = \int _0^1 \bff^*(r)\Big(\frac r{\varphi _{M^{\varsigma (\cdot)}} (r)}\Big)'\, \mathrm{d}r$$
for a measurable function $\bff$ in $\Omega$ --
see \cite[Corollary 1.9]{GogPick}. Owing to
\cite[Theorem 3.4]{cianchi-randolfi IUMJ},
\begin{equation}\label{sep29}
W^{1}_0 M ^{\varsigma (\cdot)}(\Omega) \to  C^{0, \nu(\cdot)}(\Omega),
\end{equation}
%
%any function $u$ fulfilling \eqref{sep25} 
%satisfies 
%\begin{equation}\label{sep29} u \in C^{0, \kappa(\cdot)}(\Omega),
%\end{equation}
where $\nu : [0, \infty) \to [0, \infty)$  is the function given by
\begin{align}\label{sep26}
\nu (r) = \|\rho ^{-\frac 1{n'}}\chi _{(0, r^n)}(\rho)\|_{\Lambda (0,1)} \quad \hbox{for $r \in [0, 1]$}
\end{align}
and by $\nu(r)=\nu(1)$ for $r>1$,
provided that the norm on the right-hand side of \eqref{sep26} is finite for $r \in [0,1]$ and tends to $0$ as $r \to 0^+$.
We have that 
\begin{align}\label{sep27}
\nu (r) & = \int _0^{r^n} \rho ^{-\frac 1{n'}}  \Big(\frac \rho{\varphi _{M^{\varsigma (\cdot)}} (\rho)}\Big)'\, {\rm d}\rho \\ \nonumber & = \frac{r}{\varsigma (r^n)} +  \int _0^{r^n} \frac{d\rho}{\varsigma (\rho) \, \rho ^{\frac 1{n'}}}\approx  \int _0^{r^n} \frac{\mathrm{d}\rho}{\varsigma (\rho) \, \rho ^{\frac 1{n'}}}\quad \hbox{for $r \in (0,1)$,}
\end{align}
up to multiplicative norms independent of $r$.
Observe that the second equality holds by an integration by parts,  owing to \eqref{sep24} and to the fact that
%\db{I agree that it converges to zero, but I don't see the inequality. I'd estimate by $-r\log r$ using that $\omega$ is bounded.}
$$\lim _{r\to 0^+} \frac{r}{\varsigma (r^n)} = \lim _{r\to 0^+} r \int _r^1\frac{\omega (\rho)}{\rho}\,\mathrm{d}\rho \leq \lim _{r\to 0^+} \omega (1)  r \log \frac 1r 
%\frac{r}{-\log r}\omega (1)
 =0\,,$$
and the   equivalence inasmuch as
$$  \int _0^{r^n} \frac{\mathrm{d}\rho}{\varsigma (\rho) \, \rho ^{\frac 1{n'}}} \geq \frac 1{\varsigma (r^n)} \int _0^{r^n}  \frac{\mathrm{d}\rho}{\rho ^{\frac 1{n'}}} = n \frac{r}{\varsigma (r^n)}\quad \hbox{for $r \in (0,1)$.}$$
On the other hand, 
\begin{align}\label{sep28}
\int _0^{r^n} \frac{\mathrm{d}\rho}{\varsigma (\rho) \, \rho ^{\frac 1{n'}}} & = \int _0^{r^n} \frac{1}{  \rho ^{\frac 1{n'}}} \int_{\rho ^{\frac 1n}}^1\frac{\omega (s)}{s}\, ds\, \mathrm{d}\rho \\
\nonumber & = n \int _0^r\omega (s)\, \mathrm{d}s + n r \int _r^1 \frac{\omega (s)}{s}\, \mathrm{d}s \approx r\, \overline \omega (r)\quad \hbox{for $r \in (0,1)$,}
\end{align}
where the second equality follows from Fubini's theorem, and the equivalence by the fact that
$$\lim _{r\to 0^+} \bigg(\int _0^r\omega (s)\, \mathrm{d}s \bigg)\bigg(r \int _r^1 \frac{\omega (s)}{s}\, ds\bigg)^{-1} =0\,,$$
as is easily seen via an application of De L'Hopital's rule.
\\ Embedding \eqref{sep18} is a consequence of \eqref{sep25}, \eqref{sep29}, \eqref{sep27} and \eqref{sep28}.
\end{proof}

\medskip

\begin{proof}[Proof of Theorem \ref{sharpness}] 
%Let $\omega\,:\, (0, \frac 12) \to \setR$ be given (the most interesting case is $\omega=1$), satisfying condition \eqref{eq:omega condition}. Hence, in particular, $\omega \in \Delta _2$.
%\\ 
Let $\overline \omega$ be the function associated with $\omega$ as in \eqref{omegabar}.
 Define the parameter function $\widehat \sigma$ as
%
%We define $\sigma\,:\, (0, \frac 12) \to \setR$ as
\begin{align}
  \label{eq:sigma}
  \widehat\sigma (r)= 
%\lambda\, \omega(r)  \bigg(\int _r ^1
%  \frac{\omega (\rho)}{\rho}\, d\rho\bigg)^{-1} = 
\gamma\,
  \frac{\omega(r)}{\overline{\omega}(r)} \quad \hbox{for  $r \in \big(0, \tfrac 12\big]$,}
\end{align}
$ \widehat\sigma (0)=0$ and $ \widehat\sigma (r)= \widehat\sigma (\tfrac 12)$ if $r >\tfrac 12$.
Here, $\gamma$  is a positive number to be chosen later. 
%(leading to $\sigma(\rho)=-\tfrac{\gamma}{\log\rho}$). 
 Note that $\widehat \sigma$ is an increasing continuous, function, being the product of increasing continuous functions.  Also, $\widehat \sigma \in C^1(0, \infty)$, since we are assuming that $\omega \in C^1(0, \infty)$. {We claim that the function $\widehat \sigma (r)/r$ is non-increasing  in $(0, \delta)$ for suffciently small $\delta$.  
Indeed, if $\omega (0) >0$, then we may assume, without loss of generality, that $\omega (r) = 1$, and our claim follows trivially. Suppose next that $\omega (0)=0$. As a preliminary observation, note 
 that the existence of $\lim_{r\to 0^+}\frac{r\omega '(r)}{\omega (r)}$, coupled with assumption \eqref{divint}, ensures that in fact
\begin{equation}\label{dec1028}
\lim_{r\to 0^+}\frac{r\omega '(r)}{\omega (r)} =0\,.
\end{equation}
Indeed, failure of  \eqref{dec1028} would imply that
$$\int _0^\delta  \frac{\omega(r)}r\, \mathrm{d}r \leq c \int _0^\delta \omega '(r)\, \mathrm{d}r = c\,\omega (\delta)  < \infty$$
for some positive constants $c$ ad $\delta$, thus contradicting \eqref{divint}. Now,
\begin{align}\label{dec1021}
(\widehat \sigma (r)/r)' = (r \overline \omega (r))^{-2}\big[(r\omega '(r) - \omega (r)) \overline \omega (r) + \omega (r)^2\big]\quad \hbox{for $r \in (0,\tfrac 12)$.}
\end{align}
Since 
\begin{equation}\label{dec1026}
\omega (r) = \int _0^r\omega '(s)\, \mathrm{d}s \quad \hbox{for $r>0$,}
\end{equation}
an integration by parts tells us that
\begin{equation}\label{dec1025}
\overline \omega(r) = \omega (r) \log \frac 1r  + \int _r^1 \omega '(s) \log \frac 1s \, \mathrm{d}s \quad \hbox{for $r\in (0, \tfrac 12)$.}
\end{equation}
%From equation \eqref{dec1026} and the fact that, since $\omega$ is concave, the function $\omega'$ is non-increasing, one also has that
%\begin{equation}\label{dec1027}
%\omega (r) - r \omega '(r) \geq 0  \quad \hbox{for $r>0$.}
%\end{equation}
Equations \eqref{dec1021} and \eqref{dec1025} imply that
\begin{align}\label{dec1029}
(\widehat \sigma (r)/r)' =  (r \overline \omega (r))^{-2}\omega (r)^2 \bigg[\bigg(\frac{r\omega '(r)}{\omega (r)}-1\bigg)\bigg(\log \frac 1r  + \frac{1}{\omega (r)}\int _r^1 \omega '(s) \log \frac 1s \, \mathrm{d}s\bigg) + 1\bigg]
\end{align}
for $r\in (0,\tfrac 12)$. Thanks to \eqref{dec1028}, there exists $\delta >0$ such that the right-hand side of equation \eqref{dec1029} is negative for $r \in (0, \delta)$. Our claim is thus verified.
  \\ The increasing monotonicity of the function $\widehat  \sigma(r)$ and  the decreasing monotoncity of the function $\widehat \sigma (r)/r $ ensure that $\widehat \sigma$ is equivalent to a concave function $\sigma \in C^1(0, \delta)$, in the sense that
\begin{equation}\label{marzo11} \frac 12 \sigma (r) \leq \widehat \sigma (r) \leq \sigma (r) \quad \hbox{for $r \in (0, \delta)$.}
\end{equation}
One can choose, for instance, $\sigma = \Xi^{-1}$, where $\Xi : [0, \delta) \to \setR$ is the function given by 
$\Xi (r) = \int _0^r \frac {\widehat \sigma ^{-1}(s)}s\, ds$ for $r \in [0, \delta)$.
%
%
% there exist positive constants $c_1$ and $c_2$ such that
%\begin{equation}\label{marzo11} c_1 \sigma (r) \leq \widehat \sigma (r) \leq c_2 \sigma (r) \quad \hbox{for $r \in (0, \delta)$,}
%\end{equation}
%see e.g. \cite[Chapter 2, Lemma 5.10]{BS}. 
%The function $\sigma$ can be chosen in such a way that $\sigma \in C^1(0, \delta)$.
} 
%\\ Owing to condition \eqref{eq:omega condition}, the function $\omega$ satisfies the so called $ \Delta _2$-condition. Namely, there exists a constant $c$ such that
%$$\omega (2r) \leq c \omega (r) \quad \hbox{for $r>0$.}$$
%As a consequence, the function $\widehat \sigma$ satisfies the $\Delta _2$-condition as well.
Let $\sigma$ be extended to a continuously differentiable, concave increasing bounded function in  $(0, \infty)$, %satsfying the $\Delta _2$-condition, 
still denoted by $\sigma$. In particular, the function $\frac{\sigma (r)}r$ is non-increasing. 
\\
Define the function $\psi :\, [0,\infty) \to [0,\infty)$ as 
%by $\psi(0)=0$ and
%$\psi'(t) = \sigma(t)$, i.e. 
$$\psi(r) = \int_0^r \sigma(\rho)\,{\rm d}\rho \quad \hbox{for $r \geq 0$.}$$
 Then
\begin{align}\label{eq:sigmapsi}
  r\, \sigma(r) &\geq \psi(r) \geq \frac r2 \sigma \Big( \frac r
                  2\Big) \geq \frac 14 r\, \sigma(r)\quad \hbox{for $r>0$,}
\end{align}
where the first two inequalities hold since $\sigma $ is increasing, and  the last one owing to the
monotonicity of the function  $\frac{\sigma (r)}r$.
%
% $\Delta_2$-condition for  $\sigma$.
Inasmuch as $\psi ' = \sigma$, and  the latter is a concave function, one can verify that $\psi \in C^{1, \sigma(\cdot)}$,  and hence $\psi \in W^1\mathcal L^{\sigma(\cdot)}\cap C^{0,1}$.  
\\ Let us identify  $\setR^2$  with $\setC$, and define the
domain $\Omega \subset \setR^2$ as 
\begin{align*}
  \Omega & = \bigset{\xi = x_1 + ix_2 \,:\, \abs{\xi} < \tfrac 12, x_2 > -
           \psi(\abs{x_1})}.
\end{align*}
%\commentlars{Rotate domain to get rid of additional~$i$?}
Let $\zeta(\xi)$ denote the conformal map of~$\Omega$ onto the
half-disc $\mathbb{D}^+ = \set{\zeta \,:\, \mathrm{Im}(\zeta)>0,
  \abs{\zeta}<1 }$, with fixed point~$\xi =0$. 
Let 
%$\xi = i \rho \exp(i\theta)$% 
$\phi : [0, \infty) \to [0, \infty)$ be the function such that
 $\Omega$ is given in polar coordinates~$(\rho,
\theta)$ by
\begin{align*}
  \Omega &= \bigset{ \xi= i \rho \exp(i \theta) \,:\, \rho < \tfrac 12 ,
           \abs{\theta} < \tfrac \pi{2} + \phi(\rho)}.
\end{align*}
We need  to describe the precise behaviour of the function $\phi$ as $\rho \to 0^+$. To this purpose, 
define the function $\widetilde \sigma : (0, \infty) \to [0, \infty)$ as 
$$\widetilde\sigma(r) = \frac{\psi(r)}{r}\quad \hbox{for $r>0$,}$$
and observe that, thanks to equation \eqref{eq:sigmapsi}, 
%
%
% since $\sigma$ satisfies the  $\Delta _2$-condition,
there exists a constant $c>0$ such that
\begin{equation}\label{Nov33}
c\sigma (r) \leq \widetilde \sigma (r) \leq \sigma (r) \quad \hbox{for $r>0$.}
\end{equation}
%
%$$c \sigma (r) \leq \tfrac 12 \sigma (\tfrac r2) \leq \frac 1r \int_{\frac r2}^r \sigma (s)\, ds \leq \widetilde \sigma (r) \quad \hbox{for $r>0$.}$$
Moreover, the function $ \widetilde \sigma \in C^2(0, \infty)$, and  $\widetilde \sigma'$ is locally Lipschitz continuous in $(0, \infty)$, since $\sigma$ is concave.  
In particular,
\begin{equation}\label{sep60}
r \widetilde \sigma ' (r) = \sigma (r) - \frac 1{r}\int _0^r\sigma (s)\, \mathrm{d}s \leq \sigma (r) \leq \tfrac 1c \widetilde \sigma (r)\,,
\end{equation}
and
\begin{equation}\label{dec1011}
\widetilde \sigma '' (r) = \frac{\sigma '(r)}r - 2\frac{\sigma (r) - \widetilde \sigma (r)}{r^2} 
\end{equation}
for $r>0$. Notice also that 
$$\widetilde\sigma (r) = \int _0^1\sigma (r s)\, \mathrm{d}s \quad \hbox{for $r\geq 0$,}
$$
and hence $\widetilde \sigma $ is concave, since $\sigma $ so is.
\\ For sufficiently small $|\xi|$, one has that $\xi \in \partial \Omega$ if and only if $\tan\theta=\frac{\psi(|x_1|)}{|x_1|}$, or, equivalently, 
\begin{align}\label{dec1001}
%\tan\theta=\frac{\psi(|x_1|)}{|x_1|}=\widetilde\sigma(|x_1|)\Rightarrow 
\theta=\arctan\tilde\sigma(|x_1|). 
%=\arctan\tilde\sigma(\rho\cos\theta).
\end{align}
%We want to solve this for $\theta$ (such that $\theta=\phi(\rho)$).
By Taylor's formula,
\begin{align}\label{eq:1109}
\arctan\tilde\sigma(|x_1|)=\widetilde\sigma(|x_1|)+O(\widetilde\sigma(|x_1|)^3) \quad \hbox{as $x_1 \to 0$.}
\end{align}
Here, then notation $O(\varpi (r))$ as $r\to 0^+$  for some function $\varpi$ means that $O(\varpi (r))\approx \varpi (r)$ near $0^+$.
Since $|x_1|= \rho\cos\theta$, %we infer from equation \eqref{eq:1109} that
\begin{align*}
\widetilde\sigma(|x_1|)=\widetilde\sigma(\rho\cos\theta)=\widetilde\sigma(\rho+\rho(\cos\theta-1)) & =\widetilde\sigma(\rho)+\widetilde\sigma'(\rho_\theta)\rho(\cos\theta-1)\\
&=\widetilde\sigma(\rho)+\widetilde\sigma'(\rho_\theta)\rho_\theta\frac{\rho}{\rho_\theta}(\cos\theta-1) 
\end{align*}
for some $\rho_\theta\in(\rho\cos\theta,\rho)$ and for sufficiently small $\rho$.
 Furthermore, from equation \eqref{sep60} 
% 
%{\color{red} $\widetilde \sigma'(\rho)\rho\leq c \widetilde \sigma(\rho)$}   \todo{This has to be assumed?} 
 and the fact that $\rho_\theta\approx \rho$ one has that
\begin{align*}
\widetilde\sigma'(\rho_\theta)\rho_\theta\frac{\rho}{\rho_\theta}(\cos\theta-1)\leq\,c \widetilde\sigma(\rho)\, \theta^2
\end{align*}
for some constant $c$, for sufficiently small $\rho$ and $\theta$.
Finally,  since  $\widetilde\sigma$ is increasing, equation \eqref{eq:1109} implies that $\theta\leq \,c\,\widetilde\sigma(\rho)$, and hence
\begin{align}\label{eq:1109b}
\widetilde\sigma(|x_1|)&=\widetilde\sigma(\rho)+O(\widetilde\sigma(\rho)^3) \quad \hbox{as $\rho \to 0^+$.}
\end{align}
Coupling equation \eqref{eq:1109} with \eqref{eq:1109b}, and recalling  \eqref{Nov33} yield
% \todo{Dominic: I'd prefer to give $\Omega$ in polar coordinates directly with $\widetilde\sigma$ instead of $\phi$. Then the ``original'' $\Omega$ above (which is not need anyway) gets $\psi$+error}
\begin{align}\label{Nov34}
  \phi(\rho) &= \widetilde\sigma(\rho) + \mathcal{O}\big( \widetilde\sigma(\rho)^3\big) =\widetilde\sigma(\rho) + \mathcal{O}\big( \sigma(\rho)^3\big)\quad \hbox{as $\rho \to 0^+$.}
\end{align}
Next, on defining the function $F: [0, \infty) \times (-\pi, \pi) \to \setR$ as 
$$F(\rho,\theta) = \theta-\arctan \widetilde\sigma(\rho\cos\theta) \quad \hbox{for $(\rho, \theta) \in [0, \infty) \times (-\pi, \pi)$,}$$
equation \eqref{dec1001} can be rewritten as
\begin{align}\label{dec1000}
F(\rho,\theta)=0.
%\theta-\arctan \widetilde\sigma(\rho\cos\theta)=0.
\end{align}
Therefore,   the function $\phi(\rho )$ agrees with the function $\theta (\rho)$  implicitly defined by \eqref{dec1000}. One has that
\begin{align}
F_\rho (\rho, \theta)&=-\frac{1}{1+\widetilde\sigma(\rho\cos\theta)^2}\widetilde\sigma'(\rho\cos\theta)\cos\theta,
\end{align}
\begin{align}\label{dec1006}
F_\theta(\rho, \theta) &=1+\frac{1}{1+\widetilde\sigma(\rho\cos\theta)^2}\widetilde\sigma'(\rho\cos\theta)\rho\sin\theta\\
&=\frac{1+\tilde\sigma(\rho\cos\theta)^2+\widetilde\sigma'(\rho\cos\theta)\rho\cos\theta\tan\theta}{1+\widetilde\sigma(\rho\cos\theta)^2},
\end{align}
for $(\rho, \theta) \in [0, \infty) \times (-\pi, \pi)$.
  Hence,
  \begin{align}\label{dec1009}
\phi'(\rho)&=\frac{\widetilde\sigma'(\rho\cos\theta)\cos\theta}{1+\widetilde\sigma(\rho\cos\theta)^2+\widetilde\sigma'(\rho\cos\theta)\rho\cos\theta\tan\theta}\\ \nonumber
&=\frac{\widetilde\sigma'(\rho\cos\phi(\rho))\cos\varphi(\rho)}{1+\widetilde\sigma(\rho\cos\varphi(\rho))^2+\widetilde\sigma'(\rho\cos\phi(\rho))\rho\cos\phi(\rho)\tan\phi(\rho)}
\end{align}
and 
\begin{align}\label{dec1007}
\phi'(\rho)
\approx \widetilde\sigma'(\rho\cos\phi(\rho))
\end{align}
for sufficiently small $\rho$. By equation \eqref{sep60} and the monotonicity of $\sigma$,
\begin{align}\label{dec1015}
\phi'(\rho)&\leq c \frac{\sigma(\rho\cos\phi (\rho))}{\rho\cos\phi(\rho)}\leq \,c'\, \frac{\sigma(\rho)}{\rho}
\end{align}
for some constants $c$ and $c'$ and for sufficiently small $\rho$.
\\ Now,
\begin{align}\label{dec1002}
F_{\rho \rho}(\rho, \theta) = - \cos^2 \theta \frac{\widetilde\sigma ''(\rho\cos \theta) [\widetilde\sigma (\rho \cos \theta)^2+1] - 2 \widetilde\sigma(\rho \cos\theta) \widetilde\sigma'(\rho \cos\theta) ^2}{[\widetilde\sigma (\rho \cos\theta)^2+1]^{{2}}}\,,
\end{align}
\begin{multline}\label{dec1003}
F_{\rho \theta}(\rho, \theta) \\ = \frac{[\widetilde \sigma ''(\rho \cos \theta) \rho \sin \theta \cos \theta + \widetilde \sigma'(\rho  \cos\theta) \sin \theta][\widetilde \sigma (\rho \cos \theta)^2 +1 ] - 2 \widetilde \sigma (\rho \cos \theta) \widetilde \sigma' (\rho \cos \theta) ^2 \rho \sin \theta \cos \theta}{[\widetilde \sigma (\rho \cos \theta)^2 +1]^2}\,,
\end{multline}
\begin{multline}\label{dec1004}
F_{\theta  \theta}(\rho, \theta) \\ = \frac{[- \widetilde \sigma ''(\rho \cos \theta)\rho^2 \sin ^2\theta + \widetilde \sigma ' (\rho \cos \theta) \rho\cos \theta] [\widetilde \sigma (\rho \cos \theta)^2 +1] + 2 \widetilde \sigma (\rho \cos \theta) \widetilde \sigma '(\rho \cos \theta)^2 \rho^2 \sin ^2 \theta}{[\widetilde \sigma (\rho \cos \theta)^2 +1]^2}\,
\end{multline}
for $(\rho, \theta) \in [0, \infty) \times (-\pi, \pi)$.
Thus, on denoting $\widetilde\sigma(\rho \cos\phi (\rho))$, $\widetilde\sigma'(\rho \cos \phi (\rho))$ and $\widetilde\sigma''(\rho \cos\phi (\rho))$  simply by $\widetilde\sigma$, $\widetilde\sigma'$ and $\widetilde\sigma''$, one infers from \eqref{dec1006} and \eqref{dec1002}--\eqref{dec1004} that
\begin{align}\label{dec1005}
\phi ''(\rho) & =
- \frac{F_{\rho\rho}F_\theta ^2 - 2 F_{\rho \theta}F_\rho F_\theta + F_{\theta \theta}F_\rho^2}{F_\theta ^3}
\\ \nonumber & =
 -\big[1+\tilde\sigma^2+\widetilde\sigma'\rho\sin\phi (\rho) \big]^{-3}(\widetilde\sigma ^2 + 1)^{-{{1}}}
\\ \nonumber 
& \quad \times
\bigg\{-\cos ^2\phi (\rho)
\Big[
\widetilde\sigma '' (\widetilde\sigma ^2+1))- 2 \widetilde\sigma  \widetilde\sigma'^2\big]\big[1+\tilde\sigma^2+\widetilde\sigma'\rho\sin\phi (\rho)\big]^2
\\ 
\nonumber 
& \quad \quad \quad - 2 \widetilde \sigma ' \cos \phi (\rho) \big[(\widetilde\sigma '' \rho \sin \phi (\rho)\cos \phi (\rho)
 + \widetilde \sigma' \sin \phi (\rho))(\widetilde \sigma ^2 +1 ) - 2 \widetilde \sigma  (\widetilde \sigma' )^2 \rho \sin \phi (\rho)\cos \phi (\rho)\big] 
\\ \nonumber & \quad \quad  \quad \times \big[1+\tilde\sigma^2+\widetilde\sigma'\rho\sin\phi (\rho)\big]
\\ 
\nonumber 
& \quad \quad \quad + (\widetilde \sigma' )^2  \cos^2 \phi (\rho) \big[
(- \widetilde \sigma '' \rho^2 \sin ^2\phi (\rho) + \widetilde \sigma '  \rho\cos \phi (\rho)) (\widetilde \sigma ^2 +1) + 2 \widetilde \sigma  (\widetilde \sigma ' )^2 \rho^2 \sin ^2 \phi (\rho)\big]\bigg\}
\,
\end{align} 
for $\rho \geq 0$.
On setting $$\varsigma = \rho \cos \phi (\rho),$$  and making use of  \eqref{sep60}, \eqref{dec1011} and \eqref{Nov34} and of the fact that $\sigma $ and $\widetilde \sigma$ are bounded, we deduce  that
\begin{align}\label{dec1010}
\rho & \Big|-\cos ^2\phi (\rho)
\Big[
\widetilde\sigma '' (\varsigma)(\widetilde\sigma (\varsigma)^2+1))- 2 \widetilde\sigma (\varsigma) \widetilde\sigma'(\varsigma)^2\Big]\Big[1+\tilde\sigma(\varsigma)^2+\widetilde\sigma'(\varsigma)\rho\sin\phi (\rho)\Big]^2\Big| 
\\ \nonumber & = \rho  \bigg|-\cos ^2\phi (\rho)
\bigg[\bigg(\frac{\sigma '(\varsigma)}\varsigma - 2\frac{\sigma (\varsigma ) - \widetilde \sigma (\varsigma)}{\varsigma ^2}\bigg)  (\widetilde\sigma (\varsigma)^2+1))-  2 \widetilde\sigma (\varsigma)\frac{\sigma (\varsigma)-\widetilde \sigma (\varsigma)}{\varsigma} \widetilde\sigma'(\varsigma)\bigg]
\\ \nonumber & \quad \times\bigg[1+\tilde\sigma(\varsigma)^2+   \frac{\sigma (\varsigma)-\widetilde \sigma (\varsigma)}{\varsigma} \rho\sin\phi (\rho)\bigg]^2\bigg| 
\\ \nonumber
\leq & c \bigg[\bigg(\sigma ' (\varsigma) + \cos \phi (\rho) \frac{\sigma (\varsigma ) - \widetilde \sigma (\varsigma )}{\varsigma}\bigg) + \frac{\varsigma  \widetilde \sigma '(\varsigma )}{\cos\phi (\rho)}\widetilde \sigma  (\varsigma) \widetilde \sigma '(\varsigma )\bigg]\bigg(1+ \frac{\varsigma \widetilde \sigma '(\varsigma)}{\cos\phi (\rho)}  \widetilde \sigma (\varsigma)\bigg)^2
\\ \nonumber & \leq c'\big(\sigma ' (\varsigma) + \widetilde \sigma ' (\varsigma)\big)
\end{align}
for some constants $c$ and $c'$ and for sufficiently small $\rho$. Similar estimates tell us that  
\begin{align}\label{dec1012}
\rho & \Big|
- 2 \widetilde \sigma ' (\varsigma)\cos \phi (\rho) \big[(\widetilde\sigma '' (\varsigma) \rho \sin \phi (\rho)\cos \phi (\rho)
 + \widetilde \sigma' (\varsigma)\sin \phi (\rho))(\widetilde \sigma (\varsigma)^2 +1 ) \\ \nonumber 
& \quad - 2 \widetilde \sigma (\varsigma) \widetilde \sigma' (\varsigma)^2 \rho \sin \phi (\rho)\cos \phi (\rho)\big]\Big|
\\ \nonumber & \leq c
\big(\sigma ' (\varsigma) + \widetilde \sigma ' (\varsigma)\big),
\end{align}
and 
\begin{align}\label{dec1013}
\rho & \Big|\widetilde \sigma'  (\varsigma) ^2  \cos^2 \phi (\rho) \Big[
\big(- \widetilde \sigma '' (\varsigma)  \rho^2 \sin ^2\phi (\rho) \\
\nonumber & \quad + \widetilde \sigma '   (\varsigma) \rho\cos \phi (\rho)\big) (\widetilde \sigma  (\varsigma) ^2 +1) + 2 \widetilde \sigma  (\varsigma)  \widetilde \sigma '  (\varsigma) ^2 \rho^2 \sin ^2 \phi (\rho)\Big]\Big|
\\ \nonumber & \leq c
\big(\sigma ' (\varsigma) + \widetilde \sigma ' (\varsigma)\big)
\end{align}
for some constant $c$ and for sufficiently small $\rho$. Since $\sigma $ and $\widetilde \sigma $ are concave, their derivatives ar non-increasing, and hence
\begin{align}\label{dec1014}
\sigma ' (\varsigma) = \sigma '(\rho \cos \phi (\rho)) \leq \sigma ' (\rho /2) \quad \hbox{and} \quad 
\widetilde \sigma ' (\varsigma) = \widetilde \sigma '(\rho \cos \phi (\rho)) \leq \widetilde\sigma ' (\rho /2)
\end{align}
for sufficiently small $\rho$. From equations \eqref{dec1005}--\eqref{dec1014} one can infer that there exists $\delta >0$ such that
\begin{align}\label{dec1016}
|\rho \phi ''(\rho)| \leq c \big(\sigma ' (\rho /2) + \widetilde\sigma ' (\rho /2)\big) \quad \hbox{if $\rho \in (0, \delta]$.}
\end{align}
%\bigskip
%{\color{blue}  In equation \eqref{Nov35} below, there should be $\phi$ instead of $\overline \sigma$. Since the computations below involve $\widetilde \sigma '$, information on the relation between  $\widetilde \sigma '$ and $\phi'$ is needed. I am afraid  that \eqref{Nov34} is not sufficient.}
\\
%
%Setting $\widetilde\sigma(\rho)=\frac{\psi(\rho)}{\rho}$ it is easily checked that
%\begin{align*}
%  \phi(\rho) &= \widetilde\sigma(\rho) + \mathcal{O}\big( (\widetilde\sigma(\rho))^3\big)=\widetilde\sigma(\rho) + \mathcal{O}\big( (\sigma(\rho))^3\big),
%\end{align*}
%recall \eqref{eq:sigmapsi}.
%The domain $\Omega$ can also be described by
Let us define the functions $\Phi_-, \Phi_+, \Theta : (0, \delta] \to \setR$ as 
\begin{align}\label{Nov35}
%  0<\rho<\tfrac{1}{2},\quad \Phi_{-}(\rho)<\theta<\Phi_+(\rho), \nonumber
%  \\
  \Phi_-(\rho)=-\tfrac{\pi}{2}-\phi(\rho),\quad
  \Phi_+(\rho)=\tfrac{\pi}{2}+\phi(\rho),
\end{align}
and
\begin{align*}
\Theta(\rho)&=\Phi_+(\rho)-\Phi_-(\rho)%=\pi+2\varphi(\rho)
%,\\
%\Psi(\rho)&=\Phi_-(\rho)+\Phi_+(\rho)=0.
\end{align*}
for $\rho \in (0, \delta]$. Note that $\Phi_+(\rho)+\Phi_-(\rho)=0$. Moreover,
%we are in the framework of Warschawski. 
by \eqref{dec1015}, $0\leq \rho\Phi_+'(\rho)=\rho \phi'(\rho)\leq c \sigma (\rho)$ for some constant $c$ and for sufficiently small $\rho$, whence
\begin{align*}
\lim _{\rho \to 0^+}\rho\Phi_+'(\rho)=0.
%=\rho \phi'(\rho)\leq c \sigma (\rho)
%  \xrightarrow{\rho \to 0}
%  0,
\end{align*}
 Similarly $\lim _{\rho \to 0^+} \rho\Phi_-'(\rho)=0$. 
%So, we have $\gamma=0$. In order to apply 
Thus the assumptions of \cite[Thm. XI(B)]{Wa} will be verified, if we show that, in addition,
%\footnote{Note that the third integral contains a typo in \cite{Wa}: the factor $\rho$ is missing.} 
%we have to check that the integrals %(the first two are Stieltjes integrals)
\begin{align}\label{marzo10}
\int_0 \big| (\rho\Phi_+'(\rho))'\big| \,{\rm d}\rho <\infty,\quad\int_0 \big|(\rho\Phi_-'(\rho))'\big|\,{\rm d}\rho < \infty,\quad \int_0 \frac{\Theta'(\rho)^2}{\Theta(\rho)}\rho\,\mathrm{d}\rho < \infty\,.
\end{align}
Incidentally, let us note that the last integral is affected by a typo in the statement of \cite[Thm. XI(B)]{Wa}, where the factor $\rho$ is missing.
\\ Since
 $\phi ' \geq 0$, by  equation \eqref{dec1016}
\begin{align*}
\int_0 ^\delta\big| (\rho\Phi_+'(\rho))'\big|\,\mathrm{d}\rho&=\int_0^\delta \big|(\rho\Phi_-'(\rho))'\big|\,\mathrm{d}\rho
=\int_0^\delta\big|\phi '(\rho) + \rho\phi ''(\rho))\big|\mathrm{d}\rho  \leq \int_0^\delta \phi '(\rho) \,\mathrm{d}\rho + \int_0^\delta \rho|\phi ''(\rho)|\mathrm{d}\rho
\\&
\leq  \int_0^\delta \phi '(\rho) \,\mathrm{d}\rho + c \int_0^\delta \sigma ' (\rho /2) + \widetilde\sigma ' (\rho /2) \, \mathrm{d}\rho = \phi (\delta) + c \sigma (\delta/2) + c \widetilde \sigma (\delta /2).
\end{align*}
%{\color{red} This is rigorous if $\rho \widetilde \sigma '(\rho)$ is monotone}.
%For the first integral we compute
%\begin{align*}
%\frac{d(\rho\Phi_+'(\rho))}{d\rho}=\sigma'(\rho)+\rho\sigma''(\rho).
%\end{align*}
%In the worst case we have $\sigma(\rho)=-\tfrac{1}{\log\rho}$ such that
%\begin{align*}
%\sigma'(\rho)&=\frac{1}{ \rho(\log\rho)^2},\\ \sigma''(\rho)&=-\frac{1}{\rho^2(\log\rho)^4}\big((\log\rho)^2+2\log\rho\big)\approx\frac{1}{ \rho(\log\rho)^2}.
%\end{align*}
%So, the first two integrals are fine. General case: we need
%\begin{align*}
%\sigma'(\rho)\leq\,c\frac{\sigma(\rho)}{\rho},\quad |\sigma''(\rho)|\leq\,c\frac{\sigma(\rho)}{\rho^2},
%\end{align*}
%which follows if we make the same assumption for $\omega$.\\
On the other hand, owing to \eqref{dec1015},
\begin{align}\label{dec1018}
 \int_0^{\frac{1}{2}}\frac{\Theta'(\rho)^2}{\Theta(\rho)}\rho\,\mathrm{d}\rho&=4
 \int_0^{\frac{1}{2}}\frac{\phi'(\rho)^2}{\pi+2\phi(\rho)}\rho\,\mathrm{d}\rho\leq c \int_0^{\frac{1}{2}}\phi'(\rho)^2\rho\,\mathrm{d}\rho\
%\ \nonumber
%&
\leq c'\int_0^{\frac{1}{2}}\frac{\sigma^2(\rho)}{\rho}\,\mathrm{d}\rho\,
\end{align}
for some constants $c$ and $c'$.
%where we have made use of \eqref{eq:sigmapsi}, as well as of the definition of $\widetilde\sigma$.
Now, 
\begin{align*}
\overline\omega(\rho)=\int_\rho^1\frac{\omega(r)}{r}\,\mathrm{d}r\geq \omega(\rho)\int_\rho^1\frac{1}{r}\,\mathrm{d}r=\omega(\rho)\log\frac 1\rho\, \quad \hbox{for $\rho \in (0,1)$,}
\end{align*}
whence, by \eqref{eq:sigma} and \eqref{marzo11}, $\sigma(\rho)\leq \tfrac{\gamma}{\log(1/\rho)}$ for $\rho \in (0, \delta]$. Consequently, the integral on the leftmost side of equation \eqref{dec1018} is finite.
It follows from
~\cite[Theorem~(XI)(b)]{Wa} that
\begin{align}
  \label{eq:war-asymp}
  |\zeta(\xi)| &= c\, \exp \Big( - \pi
                           \int_\rho^{\frac 12} \frac{\mathrm{d}r}{r (\pi + 2
                           \phi(r))}+o(1)\Big) \quad \hbox{as $\xi \to 0$}
\end{align}
for some positive constant $c$.
Note that
\begin{align*}
  \frac{1}{\rho (\pi +2 \phi (r))}
  &=
    \frac{1}{{ \pi }\rho} \bigg( 1 - \frac{2}{\pi}
    \phi(\rho) + O\big(\phi(\rho)^2\big)\bigg) \quad \hbox{as $\rho \to 0^+$.}
\end{align*}
Let us define the function $\mu : (0, \delta] \to \setR$ as 
\begin{align*}
  \mu(\rho) = - \pi \int_\rho^1  \frac{1}{r (\pi +2 \phi(r))}\,\mathrm{d}r + 
  \int_\rho^{1/2} \frac{1}{r} \bigg( 1- \frac{2 \phi(r)}{\pi}
  \bigg)\,\mathrm{d}r \quad \hbox{for $\rho \in (0, \delta]$.}
\end{align*}
Consequently,
\begin{align*}
  -\pi \int_\rho^{1/2} \frac{1}{r (\pi +2 \phi(r))}\,\mathrm{d}r
  &=
    - \int_\rho^{1/2} \frac{1}{r} \bigg( 1 - \frac{2}{\pi}
    \phi(r)\bigg)\,\mathrm{d}r + \mu(\rho)
  \\
  &=  \log (2\rho) +
    \frac{2}{\pi} \int_\rho^{1/2} \frac{\phi(r)}{r} \,\mathrm{d}r + \mu(\rho)  \quad \hbox{as $\rho \to 0^+$.}
\end{align*}
We have that
 $\lim_{\rho \to 0^+} \mu(\rho) <\infty$, since, by \eqref{Nov34} and \eqref{Nov33}, there exist constants $c$ and $c'$ such that
\begin{align}
  \label{eq:conv}
\int_0^{1/2} \frac{\phi (r)^2}{r}\, \mathrm{d}r & \leq c
   \int_0^{1/2} \frac{\widetilde\sigma (r)^2}{r}\, \mathrm{d}r  \leq c'  \int_0^{1/2} \frac{\sigma (r)^2}{r}\, \mathrm{d}r  
\\ \nonumber &
= c'
    \gamma^2\int_0^{\frac 12} \frac{\omega (r)^2}r \bigg(\int _r ^1
    \frac{\omega (s)}{s}\, \mathrm{d}s\bigg)^{-2}\, \mathrm{d}r
    \\ \nonumber
    & \leq c'\gamma^2 \omega(1/2) \int_0^{1/2} \frac{\omega (r)}r
    \bigg(\int _r ^1 \frac{\omega (s)}{s}\, \mathrm{d}s\bigg)^{-2}\, \mathrm{d}r
    \\ \nonumber
    &= c'
    \gamma^2 \omega(1/2) \bigg(\int _{1/2} ^1 \frac{\omega (r)}{r}\,
    \mathrm{d}r\bigg)^{-1} = c' \gamma^2
    \frac{\omega(1/2)}{\overline{\omega}(1/2)}.
\end{align}
Observe that the last but one equality holds thanks to the fact   that $\lim _{r \to 0^+}\overline{\omega}(r)=\infty$.
Moreover, by \eqref{Nov34},   there exists a positive constant $\kappa$ such that
\begin{align}\label{dec1020}
  \kappa \int_\rho^{1/2} \frac{\phi(r)}{r} \,\mathrm{d}r&\geq  \int_\rho^{1/2} \frac{\widetilde\sigma(r)}{r} \,\mathrm{d}r = \int_\rho^{1/2}\int_0^r\frac{\sigma(s)}{r^2}\,\mathrm{d}s\,\mathrm{d}r=\int_0^{\frac{1}{2}}\int_{\max\{\rho,s\}}^{\frac{1}{2}}\frac{1}{r^2}\,\mathrm{d}r\,\sigma(s)\,\mathrm{d}s
%\\ \nonumber 
%&=\int_0^{\rho}\Big[-\frac{1}{r}\Big]^{r=\frac{1}{2}}_{r=\rho}\,\sigma(s)\,\mathrm{d}s+\int_{\rho}^{\frac{1}{2}}\Big[-\frac{1}{r}\Big]^{r=\frac{1}{2}}_{r=s}\,\sigma(s)\,\mathrm{d}s\\ \nonumber 
%&=\int_0^{\rho}\Big(-2+\frac{1}{\rho}\Big)\,\sigma(s)\,\mathrm{d}s+\int_{\rho}^{\frac{1}{2}}\Big(-2+\frac{1}{s}\Big)\,\sigma(s)\,\mathrm{d}s
\\ \nonumber
&=-2\int_0^{\frac{1}{2}}\sigma(s)\,\mathrm{d}s+\frac{1}{\rho}\int_0^\rho\sigma(s)\,\mathrm{d}s+\int_\rho^{1/2} \frac{\sigma(s)}{s} \,\mathrm{d}s\\ \nonumber
&= -2\int_0^{\frac{1}{2}}\sigma(s)\,\mathrm{d}s+\widetilde\sigma(\rho)+\int_\rho^{1/2} \frac{\sigma(s)}{s} \,\mathrm{d}s \quad \hbox{for $\rho \in (0, \tfrac 12)$.}
\end{align}
%Here, $C=-2\int_0^{\frac{1}{2}}\sigma(s)\,\mathrm{d}s$.
% is a constant and and $\widetilde\sigma(\rho)\approx\sigma(\rho)\rightarrow0$.
Next,
\begin{align}
  \label{eq:asympt}
    \frac{2}{\pi} \int_\rho^{1/2} \frac{\sigma(r)}{r} \,\mathrm{d}r &=
    \frac{2\gamma}{\pi}\int _\rho ^{1/2} \frac{\omega(r)}{r}
    \bigg(\int _r ^1 \frac{\omega (s)}{s}\, \mathrm{d}s\bigg)^{-1}\, \mathrm{d}r
    \\ \nonumber
    & = \frac{2\gamma}{\pi} \log \bigg( \int _\rho ^1 \frac{\omega
      (s)}{s}\, \mathrm{d}s \bigg/ \int _{1/2}^1 \frac{\omega (s)}{s}\, \mathrm{d}s
    \bigg)
%    \\ \nonumber
%    &
= \frac{2\gamma}{\pi} \log\bigg(
    \frac{\overline{\omega}(\rho)}{\overline{\omega}(1/2)} \bigg)
\end{align}
for $\rho \in (0, \tfrac 12)$.
Thus, there exists a positive constant $c$ such that
%\begin{align*}
%  -\pi \int_\rho^{1/2} \frac{1}{r (\pi +2 \sigma(r))}\,dr
%  &=  \log (2\rho) +  \frac{2\lambda}{\pi} \log\bigg(
%    \frac{\overline{\omega}(\rho)}{\overline{\omega}(1/2)} \bigg) + \mu(\rho).
%\end{align*}
%Together with~\eqref{eq:war-asymp} this implies
%\begin{align*}
%  \textrm{Im} \zeta(\xi)
%  &= c\, \rho\,  
%    \big(\overline{\omega}(\rho) \big)^{\frac{2\lambda}{\pi}}
%    \exp
%    \big(\mu(\rho)\big)
%    \Big( \cos
%    \frac{\pi \theta}{\pi +2 \sigma(\rho)} +
%    o(1)\bigg). 
%\end{align*}
%We try to recover $\zeta$ from its imaginary part using
%$$\mathrm{arg}\zeta(\rho,\theta)=\frac{\pi\theta}{\Theta(\rho)}+o(1)$$
%as $\rho\rightarrow0$, cf. \cite[Thm. XI(A) vi)]{Wa}.
%This implies
\begin{align}\label{sep11}
   |\zeta(\xi)|
  & \geq c\, \rho\,  
    \big(\overline{\omega}(\rho) \big)^{\frac{2\gamma}{\kappa \pi}}
    \exp
    \big(\widetilde\mu(\rho)\big),
%\quad\widetilde\alpha (\rho)=\mu (\rho)+{\color{red} \tfrac 2 \pi} \widetilde\sigma (\rho)+C+o(1)\,,
\end{align}
for sufficiently small $\rho$, 
where $\kappa$ is the constant appearing in \eqref{dec1020}, and
$$\widetilde\mu (\rho)=\mu (\rho)+{\tfrac 2 \pi} \widetilde\sigma (\rho)+C+o(1)\quad \hbox{as $\rho \to 0^+$,}$$
for some constant $C$.
{  Define the function $v : \Omega \to \setR$ as 
$$v(\xi) = \mathrm{Im}(\zeta(\xi)) \qquad \hbox{for $\xi \in \Omega$.}$$
Since the function $\zeta$ is
holomorphic, the function $v$ is harmonic, and vanishes on $\{x_2= - \psi (|x_1|)\}$.   Furthermore, by the Cauchy-Riemann equations,
\begin{equation}\label{lastminute}
|\nabla v(\xi)|^2 = {\rm det} \,{\rm J} \zeta (\xi) \qquad \hbox{for $\xi \in \Omega$,}
\end{equation}
where ${\rm J}\zeta$ denotes the Jacobian matrix of the conformal map $\zeta$.
Let $w \in C^\infty _0(B_{\frac 12}(0))$ be such that $w=1$ in $B_{\frac 14}(0)$. Denote by $u: \Omega \to \mathbb R$ the function given by 
$$u=v\,w.$$ 
Thanks to \eqref{lastminute}, the function $u\in W^{1,2}_0(\Omega)$, and solves the Dirichlet problem \eqref{poisson}, with $\bfF = (w-1) \nabla v + v \nabla w$.  We claim that,  
$\bF \in C^{0, \omega(\cdot)}(\Omega)$
 and hence, thanks to \eqref{embcamp},
\begin{equation}\label{dec1036}
\bF \in \mathcal L^{ \omega(\cdot)}(\Omega).
\end{equation}
To verify out claim, notice that,
owing to \eqref{dec1028}, 
$$\lim_{r\to 0^+} \frac{r^\epsilon}{\omega (r)} =0$$
for every $\epsilon>0$. Thus, \eqref{dec1036} will follow if we  show that there exists $\epsilon \in (0, 1)$ such that 
\begin{equation}\label{dec1035}\bF \in C^{0, \epsilon}(\Omega).
\end{equation}
 As observed above, one has that $\psi \in C^{1, \sigma(\cdot)}$, whence $\psi \in C^{1, 0}$. Thus, since $v$ is harmonic in $\Omega$, one has that $v\in C^\epsilon (\Omega)$ for some $\epsilon >0$. On the other hand,   $\sigma $ is concave in $(0, \infty)$,  and consequently $\sigma \in C^{0,1}(\tfrac 18, \infty)$. Thus, $\psi \in C^{1,1}(\tfrac 14, \infty)$, and therefore $\nabla v \in C^{0, \epsilon}(\Omega\setminus B_{\frac 18}(0))$ for some $\epsilon >0$. Since the function $1-w$ vanishes in $B_{\frac 14}(0)$,  there exists $\epsilon  >0$ that renders equation \eqref{dec1035} true .
\\ In order to conclude the proof, it remains to show that
\begin{equation}\label{sep10}
\nabla u \notin \mathcal L ^{\omega (\cdot)}(\Omega)\,
\end{equation}
for a suitable choice of $\gamma$. To prove this assertion, observe that,
by \eqref{sep11}, there exists a positive constant $c$ such that
\begin{align}\label{sep12}
|u(\xi)| = |v(\xi)| \geq c \rho\,  
    \big(\overline{\omega}(\rho) \big)^{\frac{2\gamma}{\kappa \pi}}
\end{align}}
for sufficiently small $\rho$, provided that ${\rm arg}\, \zeta(\xi) = \tfrac \pi2$, namely if $\theta =0$. Assume, by contradiction, that
%\begin{equation}\label{sep13}
$\nabla u \in \mathcal L ^{\omega (\cdot)}(\Omega)\,.$
%\end{equation}
Then, by Proposition \ref{embedding}, there exists a constant $c$ such that
$$|u(\xi)| \leq c \rho\, \overline{\omega}(\rho) \quad \hbox{for $\rho \in (0, 1)$.}$$
This contradicts inequality \eqref{sep12} if $\gamma > \tfrac {\kappa \pi} 2$.
\end{proof}

\section*{Compliance with Ethical Standards}\label{conflicts}

\smallskip
\par\noindent 
{\bf Funding}. This research was partly funded by:  
\\(i)   Grant LDS 2012-05 of Leopoldina (German National Academy of Science);
\\ (ii) Research Project of the
Italian Ministry of University and Research (MIUR) Prin 2015 \lq\lq  Partial differential equations and related analytic-geometric inequalities"  (grant number 2015HY8JCC);     
\\ (iii) GNAMPA   of the Italian INdAM - National Institute of High Mathematics (grant number not available).
\\ (iv) Primus Research Programme of  Charles University, Prague (grant number PRIMUS/19/SCI/01).
%\todo[inline]{Any other grant to mention?}
%{\color{blue} (iii) The work of DB was partially supported by the Grant LDS 2012-05 of Leopoldina (German National Academy of Science).}

\smallskip
\par\noindent
{\bf Conflict of Interest}. The authors declare that they have no conflict of interest.

\end{document}